\documentclass{article}

\usepackage{amsmath,amssymb,amsthm}
\usepackage{bm}
\usepackage{mathrsfs}
\usepackage[all]{xy}
\usepackage{booktabs}
\usepackage{fullpage}
\usepackage{hyperref}
\usepackage{mathtools}
\hypersetup{
colorlinks=false,
linkcolor=red,
citecolor=green
}

\newcommand{\itg}{\mathbb{Z}}
\newcommand{\rtn}{\mathbb{Q}}

\newcommand{\rl}{\mathbb{R}}
\newcommand{\cx}{\mathbb{C}}

\newcommand{\ijbar}{i \bar{j}}
\newcommand{\ocal}{\mathcal{O}}

\newcommand{\ai}{\sqrt{-1}}

\newcommand{\surj}{\twoheadrightarrow}

\newcommand{\cst}{\mathrm{const}}
\newcommand{\kah}{K\"ahler }
\newcommand{\ke}{K\"ahler--Einstein }

\newcommand{\ddbar}{\partial \bar{\partial}}

\newcommand{\actson}{\curvearrowright}
\newcommand{\prj}{\mathbb{P}}

\theoremstyle{plain}
\newtheorem{theorem}{Theorem}[section]
\newtheorem{lemma}[theorem]{Lemma}

\newtheorem{corollary}[theorem]{Corollary}
\theoremstyle{definition}
\newtheorem{definition}[theorem]{Definition}
\newtheorem{conjecture}[theorem]{Conjecture}

\theoremstyle{definition}
\newtheorem{remark}[theorem]{Remark}

\begin{document}

\title{Scalar curvature and Futaki invariant of K\"ahler metrics with cone singularities along a divisor}
\author{Yoshinori Hashimoto}

\maketitle

\begin{abstract}
  We study the scalar curvature of K\"ahler metrics that have cone singularities along a divisor, with a particular focus on certain specific classes of such metrics that enjoy some curvature estimates. Our main result is that, on the projective completion of a pluricanonical bundle over a product of K\"ahler--Einstein Fano manifolds with the second Betti number 1, momentum-constructed constant scalar curvature K\"ahler metrics with cone singularities along the $\infty$-section exist if and only if the log Futaki invariant vanishes on the fibrewise $\mathbb{C}^*$-action, giving a supporting evidence to the log version of the Yau--Tian--Donaldson conjecture for general polarisations. \\
We also show that, for these classes of conically singular metrics, the scalar curvature can be defined on the whole manifold as a current, so that we can compute the log Futaki invariant with respect to them. Finally, we prove some partial invariance results for them.
\end{abstract}

\tableofcontents

\section{Introduction and the statement of the results}

\subsection{K\"ahler metrics with cone singularities along a divisor and log $K$-stability}


Let $D$ be a smooth effective divisor on a polarised \kah manifold $(X,L)$ of dimension $n$. Our aim is to study \kah metrics that have cone singularities along $D$, which can be defined as follows (cf.~\cite[\S2]{jmr}).

\begin{definition} \label{defcsmcgen}
A \textbf{K\"ahler metric with cone singularities along $D$ with cone angle $2 \pi \beta$} is a smooth \kah metric on $X \setminus D$ which satisfies the following conditions when we write $\omega_{sing} =  \sum_{i,j} g_{i \bar{j}} \ai dz_i \wedge d \bar{z}_j$ in terms of the local holomorphic coordinates $(z_1 , \dots , z_n)$ on a neighbourhood $U \subset X$ with $D \cap U = \{ z_1 =0 \}$:
\begin{enumerate}
\item $g_{1 \bar{1}} = F |z_1|^{2 \beta -2}$ for some strictly positive smooth bounded function $F$ on $X \setminus D$, 
\item $g_{1 \bar{j}} = g_{i \bar{1}}=  O( |z_1|^{ 2 \beta  -1})$, 
\item $g_{i \bar{j}} = O(1)$ for $i,j \neq 1$.
\end{enumerate}
\end{definition}

Although this definition makes sense for any $\beta \in \rl$, we are primarily interested in the case $0 < \beta <1$ (cf.~\cite{donconic}). On the other hand, we sometimes need to consider the case $ \beta > 1$ (cf.~Remark \ref{remcangle01}), while some results (e.g.~Theorem \ref{mtinvfutthird}) will hold only for $0 < \beta < 3/4$. We thus set our convention as follows: we shall assume $0 < \beta < 1$ in what follows, and specifically point out when this assumption is violated.


\begin{remark} \label{remusdefcsmet}
We recall that the usual (cf.~\cite{cds1, jmr, sw} amongst many others) definition of the conically singular \kah metric $\omega_{sing}$ is that $\omega_{sing}$ is a smooth \kah metric on $X \setminus D$ which is asymptotically quasi-isometric to the model cone metric $|z_1|^{2 \beta -2} \ai dz_1 \wedge d \bar{z}_1 + \sum_{i=2}^n \ai d z_i \wedge d \bar{z}_i$ around $D$, with coordinates $(z_1 , \dots , z_n)$ as above. The above definition is more restrictive than this usual definition, but will include all the cases that we shall treat in this paper (cf.~Definition \ref{csmefmcdef}).
\end{remark}

\begin{remark} \label{csmetcurr}
We can regard a conically singular metric $\omega_{sing}$ as a $(1,1)$-current on $X$, and hence can make sense of its cohomology class $[\omega_{sing}] \in H^2 (X , \rl)$.
\end{remark}



K\"ahler--Einstein metrics with cone singularities along a divisor, studied initially in \cite{jeffres,mcowen,tian96,troyanov}, attracted renewed interest since the foundational work of Donaldson \cite{donconic} on the linear theory of \ke metrics with cone singularities along a divisor. Since then, there has already been a huge accumulation of research on such metrics.


We now recall the \textbf{log $K$-stability}, which was introduced by Donaldson \cite{donconic} and played a crucially important role in proving the Yau--Tian--Donaldson conjecture (Conjecture \ref{usualdtyconj}) for Fano manifolds; see Remark \ref{logdtyfanosolv}. 
We first recall (cf.~Theorem \ref{cldfinvdiffform}) that the notion of $K$-stability can be regarded as an ``algebro-geometric generalisation'' of the vanishing of the Futaki invariant
\begin{equation*}
\mathrm{Fut} (\Xi_f , [ \omega ]) = \int_X f (S(\omega) - \bar{S}) \frac{\omega^n}{n!} 
\end{equation*}
in the sense that $\mathrm{Fut} (\Xi_f , [ \omega ]) =0$ is equivalent to $DF( \mathcal{X} , \mathcal{L})=0$ for the product test configuration $( \mathcal{X} , \mathcal{L})$ generated by $\Xi_f$ (cf.~Remark \ref{1psinaut0eqprtc}). 
Looking at the product \textit{log} test configurations, we have an analogue of the Futaki invariant in the log case, which was first introduced by Donaldson \cite{donconic}. It is defined as 
\begin{equation*}
\text{Fut}_{D ,\beta} (\Xi_f , [\omega]) = \frac{1}{2 \pi} \int_X f (S(\omega) - \bar{S}) \frac{\omega^n}{n!} - (1- \beta) \left( \int_D f \frac{\omega^{n-1}}{(n-1)!} - \frac{\text{Vol} (D, \omega)}{\text{Vol}(X , \omega)} \int_X f \frac{\omega^n}{n!} \right),
\end{equation*}
and may be called the \textbf{log Futaki invariant} (cf.~\S \ref{sclgkstab}, particularly Theorem \ref{dfdiff}). As in the case of the (classical) Futaki invariant, $\text{Fut}_{D,\beta}$ is expected to vanish on \kah classes which contain a \ke or constant scalar curvature \kah metric with cone singularities along $D$ with cone angle $2 \pi \beta$.\footnote{This certainly holds for \ke metrics on Fano manifolds; see \cite[Theorem 2.1]{sw} and also \cite[Theorem 7]{cds3}.}

Now, in view of the work of Donaldson \cite{dongauge, donproj1, dontoric}, we are naturally led to the idea of replacing the ample $-K_X$ by an arbitrary ample line bundle $L$, on a general smooth projective variety $X$, and consider the constant scalar curvature K\"ahler metrics in $c_1 (L)$ with cone singularities along a divisor $D$ (cf.~Remark \ref{csmetcurr}). Conically singular metrics having the constant scalar curvature can be defined as follows.

\begin{definition} \label{defofcsckncpt}
A \kah metric $\omega_{sing}$ with cone singularities along $D$ with cone angle $2 \pi \beta$ is said to be of \textbf{constant scalar curvature K\"ahler} or $\textbf{cscK}$ if its scalar curvature $S(\omega_{sing})$, which is a well-defined smooth function on $X \setminus D$, satisfies $S(\omega_{sing}) = \cst$ on $X \setminus D$.
\end{definition}

\begin{remark}
There are several important points when we consider cscK metrics with cone singularities in $c_1(L)$, which we list as follows.
\begin{enumerate}
	\item  Unlike in the Fano case where $D \in |- \lambda K_X|$ for some $\lambda \in \mathbb{N}$ is natural, $D$ and $L$ can be chosen completely independently; $D$ can be any smooth effective divisor in $X$ and the corresponding line bundle $\ocal_X (D)$ does not even have to be ample.
	\item There are several definitions of cscK metrics with cone singularities for general polarisations that appeared in the literature, such as \cite{kellerzheng, li15, lizheng}. 
	\item Compared with the conically singular \ke metrics that are discussed above, there seem to be relatively few results concerning conically singular \kah metrics in a general polarisation and many basic properties of conically singular cscK metrics seem yet to be clarified. In particular, there are very few known examples of such metrics. There is, however, a growing number of results \cite{cz, keller14, kellerzheng, lih12, li15, lizheng, zheng15} on this problem appearing in the literature.
\end{enumerate}
\end{remark}

\begin{remark} \label{remvolcs}
In general, if $\omega_{sing}$ is a metric with cone singularities along $D$ (as in Remark \ref{remusdefcsmet}), then it follows that any $f \in C^{\infty} (X, \rl)$ is integrable with respect to the measure $\omega^n_{sing}$ on any open set $U \subset X \setminus D$; this is because there exist positive constants $C_1 , C_2$ such that
\begin{equation*}
C_1 |z_1|^{2 \beta -2} \prod_{i=1}^n \ai dz_i \wedge d \bar{z}_i \le \omega^n_{sing} \le C_2 |z_1|^{2 \beta -2} \prod_{i=1}^n \ai dz_i \wedge d \bar{z}_i 
\end{equation*}
locally around $D$, which is locally integrable on $\cx^n \cap \{ z_1 \neq 0 \}$.

In particular, the volume $\int_{X \setminus D} \omega^{n}_{sing}$ of $X \setminus D$ is finite. By regarding $\omega^{n}_{sing}$ as an absolutely continuous measure on the whole of $X$, we shall write $\mathrm{Vol} (X , \omega_{sing}) := \int_{X \setminus D} \omega^{n}_{sing}$ in what follows.

\end{remark}

\subsection{Momentum-constructed metrics and log Futaki invariant} \label{intromc}

The study of cscK metrics is considered to be much harder than that of \ke metrics, since there is no analogue of the complex Monge--Amp\`ere equation which reduces the fourth order fully nonlinear partial differential equation (PDE) to a second order fully nonlinear PDE. However, when the space $X$ is endowed with some symmetry, it is often possible to simplify the PDE by exploiting the symmetry of the space $X$. One such example, which we shall treat in detail in what follows, is the \textbf{momentum construction} introduced by Hwang \cite{hwang} and generalised as in \cite{acgi, acgtfii, acgtfiii, hs} which works, for example, when $X$ is the projective completion $\prj (\mathcal{F} \oplus \cx)$ of a pluricanonical bundle $\mathcal{F}$ over a product of \ke manifolds (see \S \ref{bgmc} for details). The point is that this theory converts the cscK equation to a \textit{second order linear ordinary differential equation (ODE)}, as we recall in \S \ref{bgmc}.


Moreover, it is also possible to describe the cone singularities in terms of the boundary value of the function called momentum profile; a detailed discussion on this can be found in \S \ref{mccone}. This means that we have on $\prj(\mathcal{F} \oplus \cx)$ a particular class of conically singular metrics, which we may call \textbf{momentum-constructed conically singular metrics}, whose scalar curvature is easy to handle.

By using the above theory of momentum construction, we obtain the following main result of this paper. Suppose that $(M , \omega_M)$ is a product of \ke Fano manifolds $(M_i , \omega_i)$, $i =1, \dots , r$, each with $b_2 (M_i) = 1$, and of dimension $n_i$ so that $n-1 = \sum_{i=1}^r n_i$. Let $\mathcal{F} := \bigotimes_{i=1}^r p_i^* K_i^{\otimes l_i} $, $l_i \in \itg$, $K_i$ be the canonical bundle of $M_i$, and $p_i : M \surj M_i$ be the obvious projection. The statement is as follows. 

\begin{theorem} \label{mainthmmc}
Let $\mathbb{X} := \prj(\mathcal{F} \oplus \cx)$, and write $D$ for the $\infty$-section of $\prj(\mathcal{F} \oplus \cx)$ and $\Xi$ for the generator of the fibrewise $\cx^*$-action. Then, each \kah class $[ \omega ] \in H^2 (\mathbb{X} , \rl)$ of $\mathbb{X}$ admits a momentum-constructed cscK metric with cone singularities along $D$ with cone angle $2 \pi \beta \in [0,  \infty )$ if and only if $\textup{Fut}_{D ,\beta} (\Xi , [ \omega ]) =0$. 
\end{theorem}

In fact, $\beta \neq 1$ since $\prj(\mathcal{F} \oplus \cx)$ admits no cscK metrics as $\mathcal{F} \oplus \cx$ is Mumford unstable \cite[Theorem 5.13]{rt}. The reader is referred to \S \ref{bgmc} for more details on this statement, including where the various hypotheses on $\mathbb{X}$ came from. See also Remark \ref{remcangle01} for some examples.

\begin{remark} \label{intremcangle01}
Note that the value of $\beta$ for which this happens is unique in each \kah class $[ \omega ] \in H^2 (\mathbb{X} , \rl)$, given by the equation $\textup{Fut}_{D ,\beta} (\Xi , [ \omega ]) =0$ which we can re-write as
\begin{equation*}
\beta  = 1- \mathrm{Fut} (\Xi , [ \omega ]) \left(\int_D f \frac{\omega^{n-1}}{(n-1)!} - \frac{\mathrm{Vol}(D , \omega)}{\mathrm{Vol}(\mathbb{X}, \omega)} \int_{\mathbb{X}} f \frac{\omega^n}{n!} \right)^{-1} ,
\end{equation*}
where $f$ is the holomorphy potential of $\Xi$; the denominator in the second term is equal to $Q(b) (b -B/A)$ in the notation of (\ref{lfutcangledeninv}), which is strictly positive. We also need to note that we do \textit{not} necessarily have $0 < \beta < 1$; although we can show $\beta \ge 0$, there are examples where $\beta >1$. See Remark \ref{remcangle01} for more details.

\end{remark}

\begin{remark}
A naive re-phrasing of the above result is that each rational \kah class (or polarisation) of $\mathbb{X} = \prj(\mathcal{F} \oplus \cx)$ admits a momentum-constructed cscK metric with cone singularities along $D$ with cone angle $2 \pi \beta$ if and only if it is log $K$-polystable with cone angle $2 \pi \beta$ with respect to the product log test configuration generated by the fibrewise $\cx^*$-action on $\mathbb{X}$. To the best of the author's knowledge, this is the first supporting evidence for the log Yau--Tian--Donaldson conjecture (Conjecture \ref{logdtyconj}) for the polarisations that are not anticanonical.

\end{remark}



\subsection{Log Futaki invariant computed with respect to the conically singular metrics} \label{intscmsfi}
Although the log Futaki invariant is conjectured to be related to the existence of conically singular cscK metrics, the log Futaki invariant itself is computed with respect to a \textit{smooth} \kah metric in $c_1 (L)$. We now consider the following question: what is the value of the log Futaki invariant if we compute it with respect to a conically singular \kah metric?\footnote{Auvray \cite{auvray} established an analogous result for the Poincar\'e type metric, which can be regarded as the $\beta =0$ case.} Namely, we wish to compute $\mathrm{Fut}_{D , \beta} (\Xi_f , \omega_{sing})$ defined as
\begin{align*}
&\int_{X} f \left( \mathrm{Ric} (\omega_{sing}) - \frac{\overline{S} (\omega_{sing})}{n} \omega_{sing} \right) \wedge \frac{\omega^{n-1}_{sing}}{(n-1)!} \\
&\ \ \ \ \ \ \ \ \ \  - 2 \pi (1 - \beta) \left( \int_D f \frac{\omega^{n-1}_{sing}}{(n-1)!}  - \frac{\text{Vol} (D ,\omega_{sing})}{\text{Vol} (X , \omega_{sing})} \int_X f \frac{\omega^n_{sing}}{n!} \right) ,
\end{align*}
where $\overline{S} (\omega_{sing}) : = \frac{1}{\text{Vol} (X , \omega_{sing})} \int_X \mathrm{Ric} (\omega_{sing}) \wedge \frac{\omega^{n-1}_{sing}}{(n-1)!}$. However, this is not a priori well-defined for any conically singular metric $\omega_{sing}$; first of all $\int_D f \frac{\omega^{n-1}_{sing}}{(n-1)!}$ does not naively make sense as $\omega_{sing}$ is not well-defined on $D$, and also it is not obvious that the integral $ \int_X \mathrm{Ric} (\omega_{sing}) \wedge \frac{\omega^{n-1}_{sing}}{(n-1)!}$ or $\int_{X} f  \mathrm{Ric} (\omega_{sing})  \wedge \frac{\omega^{n-1}_{sing}}{(n-1)!}$ makes sense.\footnote{Note that $\text{Vol}(X , \omega_{sing})$ does make sense by Remark \ref{remvolcs}.}

In what follows, we do not claim any result on this problem that is true for \textit{all} conically singular metrics, and restrict our attention to the case where the conically singular metric $\omega_{sing}$ has some ``preferable'' form. By this, we mean that $\omega_{sing}$ is either of the following types.
\begin{definition} \label{csmefmcdef}
\
\begin{enumerate}
\item Let $\ocal_X (D)$ be the line bundle associated to $D$ and $s$ be a global section that defines $D$ by $\{ s=0 \}$. Giving a hermitian metric $h$ on $\ocal_X (D)$, we define $\hat{\omega} := \omega + \lambda \ai \ddbar |s|^{2 \beta}_h$ which is indeed a \kah metric if $\lambda >0$ is chosen to be sufficiently small. Metrics of such form have been studied in many papers (\cite{bre, cz, donconic, jmr} amongst others). In this paper, we call such a metric $\hat{\omega}$ a conically singular metric \textbf{of elementary form}.
\item When $\mathbb{X}$ is a projective completion $\prj (\mathcal{F} \oplus \cx)$ of a line bundle $\mathcal{F}$ over a \kah manifold $M$, with the projection map $p: \mathcal{F} \to M $, we can consider a \textbf{momentum-constructed} metric $\omega_{\varphi}$ (as we mentioned in \S \ref{intromc}; see also \S \ref{bgmc} for the details). We have an explicit description of cone singularities, as we shall see in \S  \ref{mccone}.
\end{enumerate}
\end{definition}
Throughout in what follows, we shall write $X$ to denote a projective \kah manifold, and $\mathbb{X}$ for the projective completion $\prj (\mathcal{F} \oplus \cx)$.

What is common in the above two classes of metrics is that they can be written as a sum of a globally defined smooth differential form and a term of order $O(|z_1|^{2 \beta})$, together with some more explicit estimates on the second $O(|z_1|^{2 \beta})$ term, which will be important for us in proving that these metrics enjoy some nice estimates on the Ricci (and scalar) curvature (cf.~\S \ref{mccone}, \S \ref{eleestcstf}); see also Remark \ref{gencsmacnd}.

For these types of metrics, $\hat{\omega}$ and $\omega_{\varphi}$, we first show that $\mathrm{Ric} (\hat{\omega}) \wedge \hat{\omega}^{n-1}$ and $\mathrm{Ric} (\omega_{\varphi}) \wedge \omega_{\varphi}^{n-1}$ define a current that is well-defined on the whole manifold. In fact, we can even show that they are well-defined as a current on any open subset $\Omega$ in $X$, as stated in the following. They are the main technical results that are used in what follows to compute the log Futaki invariant.


\begin{theorem} \label{scaldist}
Let $\hat{\omega}$ be a conically singular \kah metric of elementary form $\hat{\omega} = \omega + \lambda \ai \ddbar |s|^{2 \beta}_h$ with $0 < \beta <1$. Then the following equation
\begin{equation*}
\int_{\Omega} f \mathrm{Ric}(\hat{\omega}) \wedge \frac{\hat{\omega}^{n-1}}{(n-1)!} = \int_{\Omega \setminus D} f S (\hat{\omega}) \frac{\hat{\omega}^n}{n!} + 2  \pi (1 - \beta) \int_{\Omega \cap D} f \frac{\omega^{n-1}}{(n-1)!}
\end{equation*}
holds for any open set $\Omega \subset X$ and any $f \in C^{\infty} (X , \rl)$, and all the integrals are finite.
\end{theorem}

\begin{theorem} \label{scaldistmc}
Let $p: \mathcal{F} \to M$ be a holomorphic line bundle with hermitian metric $h_{\mathcal{F}}$ over a \kah manifold $(M , \omega_M)$, and $\omega_{\varphi}$ be a momentum-constructed conically singular \kah metric on $\mathbb{X} := \prj (\mathcal{F} \oplus \cx)$ with a real analytic momentum profile $\varphi$ and $0< \beta <1$. Then the following equation
\begin{equation*}
\int_{\Omega} f \mathrm{Ric} (\omega_{\varphi}) \wedge \frac{\omega_{\varphi}^{n-1}}{(n-1)!} = \int_{\Omega \setminus D} f S (\omega_{\varphi}) \frac{\omega_{\varphi}^n}{n!} + 2  \pi (1 - \beta)  \int_{\Omega \cap D} f \frac{p^* \omega_M(b)^{n-1}}{(n-1)!}
\end{equation*}
holds for any open set $\Omega \subset \mathbb{X}$ and any $f \in C^{\infty} (\mathbb{X} , \rl)$, and all the integrals are finite, where $\omega_M (b)$ is as defined in (\ref{defofomegamtau}).
\end{theorem}

See Remark \ref{remcompprsts} for the comparison to similar results in the literature.

Recalling (cf.~Theorem \ref{dfdiff}) that the log Futaki invariant $\mathrm{Fut}_{D , \beta} $ is defined as a sum of the classical Futaki invariant (cf.~Theorem \ref{cldfinvdiffform}) and a ``correction'' term, we need to ensure that the classical Futaki invariant with respect to the conically singular metrics, of elementary form and momentum-constructed, is well-defined. Theorem \ref{scaldist} enables us to make sense\footnote{In fact, there is also a subtlety involving the asymptotic behaviour of the holomorphy potential $\hat{H}$, cf.~\S \ref{futinvcmpcsmtf} and \S \ref{mccsmesohpndfut}.} of the following quantity
\begin{equation*}
\textup{Fut} (\Xi, \hat{\omega}) := \int_X \hat{H} \left( \mathrm{Ric} (\hat{\omega}) - \frac{\bar{S} (\hat{\omega})}{n} \right) \wedge \frac{\hat{\omega}^{n-1}}{(n-1)!},
\end{equation*}
where $\hat{H}$ is the holomorphy potential of $\Xi$ with respect to $\hat{\omega}$ (cf.~(\ref{defholopotf})). Similarly, Theorem \ref{scaldistmc} gives us an analogous statement for the momentum-constructed conically singular metrics. The detailed statement of these results is given in Corollary \ref{intfutakisingth}. Given all these results, we can finally compute the log Futaki invariant, as in Theorem \ref{mtinvfutthird}; a key step in the proof is that the ``distributional'' term in $\textup{Fut} (\Xi , \hat{\omega})$ (resp.~$\textup{Fut} (\Xi , \omega_{\varphi} )$) exactly cancels the ``correction'' term in the log Futaki invariant (cf.~Corollary \ref{corfutcsmtypdist} (resp.~Corollary \ref{invlfimcmpbsfoxcc})). We also prove a partial invariance result for the Futaki invariant, when it is computed with respect to these classes of conically singular metrics. For the smooth metrics, that the Futaki invariant depends only on the \kah class is a well-known theorem of Futaki \cite{futaki} (cf. Theorem \ref{cldfinvdiffform}), where the proof crucially relies on the integration by parts. When we compute it with respect to conically singular metrics, we are essentially on the noncompact manifold $X \setminus D$, and hence cannot naively apply the integration by parts. Still, we can claim the following result.

\begin{theorem} \label{mtinvfutthird}
Suppose $0 < \beta < 3/4$.
\begin{enumerate} 
\item The log Futaki invariant computed with respect to a conically singular metric of elementary form $\hat{\omega}$, evaluated against a holomorphic vector field $\Xi$ which preserves $D$ and with the holomorphy potential $\hat{H}$, is given by
\begin{equation*}
\textup{Fut}_{D , \beta} (\Xi , \hat{\omega}) = \frac{1}{2 \pi} \int_{X \setminus D} \hat{H} (S(\hat{\omega}) - \underline{S} (\hat{\omega})) \frac{\hat{\omega}^n}{n!} ,
\end{equation*}
and it is invariant under the change $\hat{\omega} \mapsto \hat{\omega} + \ai \ddbar \psi$ for any smooth function $\psi \in C^{\infty} (X , \rl)$ with $\hat{\omega} + \ai \ddbar \psi > 0$ on $X \setminus D$, i.e.~$\textup{Fut}_{D, \beta} (\Xi , \hat{\omega} + \ai \ddbar \psi) = \textup{Fut}_{D , \beta} (\Xi , \hat{\omega})$. 
In particular, if $\hat{\omega}$ is cscK, $\textup{Fut}_{D , \beta} (\Xi , \hat{\omega} + \ai \ddbar \psi) =0$ for any $\psi \in C^{\infty} (X ,\rl)$ with $\hat{\omega} + \ai \ddbar \psi > 0$ on $X \setminus D$. 

\item Suppose that the $\sigma$-constancy hypothesis (cf.~Definition \ref{defsigmaconst}) is satisfied for our data, and let $D$ be the $\infty$-section of $\mathbb{X} = \prj (\mathcal{F} \oplus \cx)$. Then the log Futaki invariant computed with respect to a momentum-constructed conically singular metric $\omega_{\varphi}$, evaluated against the generator $\Xi$ of fibrewise $\cx^*$-action, is given by
\begin{equation*}
\textup{Fut}_{D , \beta} (\Xi , \omega_{\varphi}) =\int_{\mathbb{X} \setminus D} \tau (S (\omega_{\varphi}) - \underline{S}(\omega_{\varphi})) \frac{\omega_{\varphi}^n}{n!} ,
\end{equation*}
and it is invariant under the change $\omega_{\varphi} \mapsto \omega_{\varphi} + \ai \ddbar \psi$ for any smooth function $\psi \in C^{\infty} (\mathbb{X} , \rl)$ with $\omega_{\varphi} + \ai \ddbar \psi > 0$ on $\mathbb{X} \setminus D$.
\end{enumerate}
\end{theorem}

\begin{remark}
The author conjectures that the result should be true for $0 < \beta <1$ in general.
\end{remark}


\subsection{Organisation of the paper}
We first review the basics on log $K$-stability and log Futaki invariant in \S \ref{sclgkstab}.

\S \ref{mccmcsad} discusses in detail the momentum-constructed conically singular metrics and log Futaki invariant, in particular our main result Theorem \ref{mainthmmc};  \S \ref{bgmc} is a general introduction, and \S \ref{mccone} discusses some basic properties of momentum-constructed metrics that have cone singularities. \S \ref{mainlfmcpf} is devoted to the proof of Theorem \ref{mainthmmc}.

\S \ref{scalmeas} and \S \ref{sipftcsmotf} discuss in detail the log Futaki invariant computed with respect to conically singular metrics, as presented in \S \ref{intscmsfi}. After collecting some basic estimates on conically singular metrics of elementary form in \S \ref{eleestcstf}, we prove in \S \ref{scscaldist} that the current $\mathrm{Ric} (\hat{\omega}) \wedge \hat{\omega}^{n-1}$ (and $\mathrm{Ric} (\omega_{\varphi}) \wedge \omega_{\varphi}^{n-1}$) is well-defined on the whole of $X$, as stated in Theorems \ref{scaldist} and \ref{scaldistmc}. Corollary \ref{intfutakisingth} is proved in \S \ref{lfinvcwercsm}.

\S \ref{sipftcsmotf} is concerned with the proof of Theorem \ref{mtinvfutthird}; the main result of \S \ref{invvhamcstf} is Corollary \ref{corfutcsmtypdist} (see also Remark \ref{exttermlogfutsing}), which reduces the claim (for the conically singular metrics of elementary form) to the computations that we do in \S \ref{invfutwrtcsmtf} along the line of proving the invariance of the classical Futaki invariant (i.e.~the smooth case). \S \ref{iotlficmccsmss} establishes the claim for the momentum-constructed conically singular metrics.

\section{Log Futaki invariant and log $K$-stability}  \label{sclgkstab}

\subsection{Test configurations and $K$-stability}  \label{usualksdtystatotcon} 

We first recall the ``usual'' $K$-stability. This was first introduced by Tian \cite{tian97} and made a purely algebro-geometric notion by Donaldson \cite{dontoric}.

\begin{definition} \label{intdeftestconfigusu}
A \textbf{test configuration} for a polarised \kah manifold $(X,L)$ with exponent $r \in \mathbb{N}$ is a projective scheme $\mathcal{X}$ together with a relatively ample line bundle $\mathcal{L}$ over $\mathcal{X}$ and a flat morphism $\pi : \mathcal{X} \to \cx$ with a $\cx^*$-action on $\mathcal{X}$, which covers the usual multiplication in $\cx$ and lifts to $\mathcal{L}$ in an equivariant manner, such that the fibre $\pi^{-1} (1)$ is isomorphic to $(X,L^{\otimes r})$.
\end{definition}

\begin{remark}
We recall the following important and well known observations.
\begin{enumerate}
\item By virtue of the (equivariant) $\cx^*$-action on $\mathcal{X}$, all non-central fibres $\mathcal{X}_t := \pi^{-1} (t)$ ($t \in \cx^*$) are isomorphic and the central fibre $\mathcal{X}_ 0:= \pi^{-1} (0)$ is naturally acted on by $\cx^*$.

\item Although $X$ is a smooth manifold, the central fibre $\mathcal{X}_0$ of a test configuration is usually not smooth. In fact, $\mathcal{X}_0$ is a priori just a scheme and not even a variety.

\item A test configuration $(\mathcal{X}, \mathcal{L})$ is called \textbf{product} if $\mathcal{X}$ is isomorphic $X \times \cx$. Note that this isomorphism is not necessarily equivariant, so $X$ may have a nontrivial $\cx^*$-action. $(\mathcal{X}, \mathcal{L})$ is called \textbf{trivial} if $\mathcal{X}$ is equivariantly isomorphic to $X \times \cx$, i.e. with trivial $\cx^*$-action on $X$.
\end{enumerate}
\end{remark}

\begin{remark} \label{pathsingkstble}
A well-known pathology found by Li and Xu \cite{lx} means that we may have to assume that $\mathcal{X}$ is a normal variety when $(\mathcal{X} , \mathcal{L})$ is not product or trivial. Alternatively, we may have to assume that the $L^2$-norm of the test configuration (as introduced by Donaldson \cite{donlb}) is non-zero to define the non-triviality of the test configuration, as proposed by Sz\'ekelyhidi \cite{sze, szefilt}. See also \cite{bhj, dertwisted, stoppa}.
\end{remark}

Let $(\mathcal{X}_t , \mathcal{L}_t)$ be any fibre of a test configuration $(\mathcal{X}, \mathcal{L})$ with the polarisation given by $\mathcal{L}_t := \mathcal{L} |_{\mathcal{X}_t}$. By the Riemann--Roch theorem and flatness,
\begin{equation*}
d_k := \dim H^0 (\mathcal{X}_t , \mathcal{L}_t^{\otimes k}) = a_0 k^n + a_1 k^{n-1} + O(k^{n-2})
\end{equation*}
with $a_0 , a_1 \in \rtn$. On the other hand, the $\cx^*$-action on the central fibre $(\mathcal{X}_0, \mathcal{L}_0)$ induces a representation $\cx^* \actson H^0 (\mathcal{X}_0, \mathcal{L}_0^{\otimes k})$. Let $w_k$ be the weight of the representation $\cx^* \actson \bigwedge^{\text{max}} H^0 (\mathcal{X}_0, \mathcal{L}_0^{\otimes k})$. Equivariant Riemann--Roch theorem (cf. \cite{dontoric}) shows that
\begin{equation*}
w_k = b_0 k^{n+1} + b_1 k^n + O(k^{n-1}) .
\end{equation*}
Now expand
\begin{equation*}
\frac{w_k}{k d_k} = \frac{b_0}{a_0} + \frac{a_0 b_1 - a_1 b_0}{a_0^2} k^{-1} + O(k^{-2}).
\end{equation*}

\begin{definition} \label{defdfinvag}
\textbf{Donaldson--Futaki invariant} $DF(\mathcal{X}, \mathcal{L})$ of a test configuration $(\mathcal{X}, \mathcal{L})$ is a rational number defined by $DF(\mathcal{X}, \mathcal{L}) = (a_0 b_1 - a_1 b_0) / {a_0}$.
\end{definition}

\begin{definition}
A polarised projective scheme $(X,L)$ is \textbf{$K$-semistable} if $DF (\mathcal{X}, \mathcal{L}) \ge 0$ for any test configuration $(\mathcal{X} , \mathcal{L})$ for $(X,L)$. $(X,L)$ is \textbf{$K$-polystable} if $DF (\mathcal{X}, \mathcal{L}) \ge 0$ with equality if and only if $(\mathcal{X} , \mathcal{L})$ is product, and is \textbf{$K$-stable} if $DF (\mathcal{X}, \mathcal{L}) \ge 0$ with equality if and only if $(\mathcal{X} , \mathcal{L})$ is trivial.
\end{definition}

We see that the sign of $DF (\mathcal{X}, \mathcal{L})$ is unchanged when we replace $\mathcal{L}$ by $\mathcal{L}^{\otimes r}$. Therefore, once $\mathcal{X}$ is fixed, we may assume that the exponent of the test configuration is always 1 with $L$ being very ample.

The following conjecture, usually referred to as \textbf{Yau--Tian--Donaldson conjecture}, is well-known; see Remark \ref{logdtyfanosolv} for the special case when $X$ is a Fano manifold.

\begin{conjecture} \label{usualdtyconj}
(Yau \cite{yauprob}, Tian \cite{tian97}, Donaldson \cite{dontoric})
$(X,L)$ admits a cscK metric in $c_1 (L)$ if and only if it is $K$-polystable.
\end{conjecture}

We now discuss product test configurations and the automorphism group of $(X,L)$ in detail. In this case, the Donaldson--Futaki invariant admits a differential-geometric formula as given in Theorem \ref{cldfinvdiffform}, which is called the (classical) Futaki invariant. We first briefly review the automorphism group of $(X,L)$; the reader is referred to \cite{kobayashi, lebsim} for more details on what is discussed here.

Let $\mathrm{Aut} (X)$ be the group of holomorphic transformations of $X$ which consists of diffeomorphisms of $X$ which preserve the complex structure $J$, and we write $\text{Aut}_0 (X)$ for the connected component of $\mathrm{Aut} (X)$ containing the identity. 
\begin{definition}
A vector field $v$ on $X$ is called \textbf{real holomorphic} if it preserves the complex structure, i.e.~the Lie derivative $L_v J$ of $J$ along $v$ is zero. A vector field $\Xi$ is called \textbf{holomorphic} if it is a global section of the holomorphic tangent sheaf $T_X$, i.e.~$\Xi \in H^0 (X,T_X)$. 
\end{definition}


\begin{remark} \label{rem11corbauth0}
It is well-known (cf.~\cite[Proposition 2.11, Chapter IX]{kn}) that there exists a one-to-one correspondence between the elements in $\mathfrak{aut} (X)$ and $H^0 (X ,T_X)$; the map $f^{1,0} : \mathfrak{aut} (X) \ni v \mapsto v^{1,0} \in H^0 (X ,T_X)$ defined by taking the $(1,0)$-part and the map $f^{\mathrm{Re}} : H^0 (X ,T_X) \ni \Xi \mapsto \mathrm{Re} (\Xi) \in \mathfrak{aut} (X)$ defined by taking the real part are the inverses of each other.
\end{remark}

We now write $\text{Aut} (X,L)$ for the subgroup of $\textup{Aut} (X)$ consisting of the elements whose action lifts to an automorphism of the total space of the line bundle $L$, and write $\mathrm{Aut}_0 (X,L)$ for the identity component of $\mathrm{Aut} (X,L)$. It is known that for any $v \in \mathrm{LieAut}_0 (X,L)$ and a \kah metric $\omega$ on $X$ there exists $f \in C^{\infty} (X, \cx)$ such that

\begin{equation}
	\iota (v^{1,0}) \omega = - \bar{\partial} f , \label{defholopotf}
\end{equation}
where $\iota$ denotes the interior product. Such $f$ is called the \textbf{holomorphy potential} of $v^{1,0}$ with respect to $\omega$. Conversely, if $\Xi \in H^0 (X,T_X)$ admits a holomorphy potential, then $\mathrm{Re} (\Xi) \in \mathrm{LieAut}_0 (X,L)$ (cf. \cite[Theorem 1]{lebsim} and \cite[Theorems 9.4 and 9.7]{kobayashi}).


\begin{remark} \label{1psinaut0eqprtc}
It is immediate that a (nontrivial) product test configuration for $(X,L)$ is exactly a choice of 1-parameter subgroup $\cx^*$ in $\mathrm{Aut}_0 (X,L)$, where we recall that the $\cx^*$-action has to lift to the total space of the line bundle $L$ to define a test configuration (cf.~Definition \ref{intdeftestconfigusu}). If we write $v \in \mathrm{LieAut}_0 (X,L)$ for the generator of this subgroup $\cx^* \le \mathrm{Aut}_0 (X,L)$, the above argument shows that $v^{1,0} \in H^0 (X,T_X)$ admits a holomorphy potential, and that conversely $\Xi \in H^0 (X,T_X)$ admitting a holomorphy potential defines a 1-parameter subgroup $\cx^* \le \mathrm{Aut}_0 (X,L)$ under the correspondence in Remark \ref{rem11corbauth0}. To summarise, \textit{a product test configuration is exactly a choice of $\Xi \in H^0 (X,T_X)$ which admits a holomorphy potential.}
\end{remark}

Finally, we recall the following well-known theorem.

\begin{theorem} \label{cldfinvdiffform}
\emph{(Donaldson \cite{dontoric}, Futaki \cite{futaki})}
Let $f \in C^{\infty} (X , \cx)$ be the holomorphy potential of a holomorphic vector field $\Xi_f$ on $X$ with respect to a \kah metric $\omega \in c_1 (L)$. If $(\mathcal{X} , \mathcal{L})$ is the product test configuration generated by $\Xi_f$, the Donaldson--Futaki invariant can be written as
\begin{equation*}
DF (\mathcal{X}, \mathcal{L}) = \frac{1}{4 \pi} \int_X f (S(\omega) - \bar{S}) \frac{\omega^n}{n!} ,
\end{equation*}
where $S(\omega)$ is the scalar curvature of $\omega$ and $\bar{S}$ is the average of $S(\omega)$ over $X$. The integral in the right hand side
\begin{equation*}
\mathrm{Fut} (\Xi_f , [\omega]):= \int_X f (S(\omega) - \bar{S}) \frac{\omega^n}{n!},
\end{equation*}
called the \textbf{Futaki invariant} or \textbf{classical Futaki invariant}, does not depend on the specific choice of \kah metric $\omega$, i.e. is an invariant of the cohomology class $[\omega]$.
\end{theorem}

\subsection{Log $K$-stability}


Donaldson \cite{donconic} introduced the notion of log $K$-stability, in the attempt to solve Conjecture \ref{usualdtyconj} for the Fano manifolds; see also Remark \ref{logdtyfanosolv}. This is a variant of $K$-stability that is expected to be more suited to conically singular cscK metrics. We refer to \cite{donconic, os} for a general introduction.

This purely algebro-geometric notion can be defined for an $n$-dimensional polarised normal variety $(X,L)$ together with an effective integral reduced divisor $D \subset X$, but we will throughout assume that $(X,L)$ is a polarised \kah manifold and $D \subset X$ is a smooth effective divisor as this is the case we will be exclusively interested in. We write $((X,D);L)$ for these data. 


Suppose now that we have a test configuration $(\mathcal{X} , \mathcal{L})$ for $(X,L)$. As in \S \ref{usualksdtystatotcon}, the equivariant $\cx^*$-action on $\mathcal{X}$ induces an action on the central fibre $\mathcal{X}_0$, and hence an action on $H^0 (\mathcal{X}_0 , \mathcal{L}^{\otimes k} |_{\mathcal{X}_0})$ for any $k \in \mathbb{N}$. We write $d_k$ for $\dim H^0 (\mathcal{X}_0 , \mathcal{L}^{\otimes k} |_{\mathcal{X}_0})$ and $w_k$ for the weight of the $\cx^*$-action on $\bigwedge^{\text{max}} H^0 (\mathcal{X}_0 , \mathcal{L}^{\otimes k} |_{\mathcal{X}_0})$. As we saw in \S \ref{usualksdtystatotcon}, these admit an expansion in $k \gg 1$ as 
\begin{align*}
d_k &= a_0 k^n + a_1 k^{n-1} + \cdots \\
w_k &= b_0 k^{n+1} + b_1 k^{n} + \cdots  
\end{align*}
where $a_i$, $b_i$ are some rational numbers.

The $\cx^*$-action on $\mathcal{X} $ naturally induces a test configuration $(\mathcal{D} , \mathcal{L} |_{\mathcal{D}})$ of $(D,L |_D)$ by supplementing the orbit of $D$ (under the $\cx^*$-action) with the flat limit. Similarly to the above, writing $\mathcal{D}_0$ for the central fibre, we write $\tilde{d}_k$ for $\dim H^0 (\mathcal{D}_0 , \mathcal{L}^{\otimes k} |_{\mathcal{D}_0})$ and $\tilde{w}_k$ for the weight of the $\cx^*$-action on $\bigwedge^{\text{max}} H^0 (\mathcal{D}_0 , \mathcal{L}^{\otimes k} |_{\mathcal{D}_0})$. We have the expansion
\begin{align*}
\tilde{d}_k &= \tilde{a}_0 k^{n-1} + \tilde{a}_1 k^{n-2} + \cdots  \\
\tilde{w}_k &= \tilde{b}_0 k^n + \tilde{b}_1 k^{n-1} + \cdots 
\end{align*}
exactly as above, where $\tilde{a}_i$, $\tilde{b}_i$ are some rational numbers.

Thus a test configuration $(\mathcal{X} , \mathcal{L})$ and a choice of divisor $D \subset X$ gives us two test configurations $(\mathcal{X} , \mathcal{L})$ and $(\mathcal{D} , \mathcal{L} |_{\mathcal{D}})$. We call the pair $(\mathcal{X} , \mathcal{L})$ and $(\mathcal{D} , \mathcal{L} |_{\mathcal{D}})$ constructed as above a \textbf{log test configuration} for the pair $((X,D);L)$, and write $((\mathcal{X} , \mathcal{D}); \mathcal{L})$ to denote these data. We now define the \textbf{log Donaldson--Futaki invariant}
\begin{equation} \label{defoflogdfinv}
DF (\mathcal{X} , \mathcal{D} , \mathcal{L} , \beta) := \frac{2( a_0 b_1 - a_1 b_0 )}{a_0} - (1 - \beta) \left( \tilde{b}_0 - \frac{\tilde{a}_0}{a_0} b_0\right) ,
\end{equation}
analogously to Definition \ref{defdfinvag}.

We now consider a special case where the log test configuration $((\mathcal{X} , \mathcal{D}); \mathcal{L})$ is given by a $\cx^*$-action on $X$ which lifts to $L$ and preserves $D$. We then have isomorphisms $\mathcal{X} \cong X \times \cx$ and $\mathcal{D} \cong D \times \cx$, and in particular the central fibre $\mathcal{X}_0$ (resp. $\mathcal{D}_0$) is isomorphic to $X$ (resp. $D$). Note that the above isomorphisms are not necessarily equivariant, and hence the central fibres $\mathcal{X}_0 \cong X$ and $\mathcal{D}_0 \cong D$ could have a nontrivial $\cx^*$-action. In this case the log test configuration $((\mathcal{X} , \mathcal{D}); \mathcal{L})$ is called \textbf{product}. In the more restrictive case where the above isomorphisms are equivariant, i.e. when $\cx^*$-action acts trivially on the central fibres $\mathcal{X}_0 \cong X$ and $\mathcal{D}_0 \cong D$, the log test configurations is called \textbf{trivial}.

\begin{remark} \label{1psinaut0eqprtclog}
As in Remark \ref{1psinaut0eqprtc}, a product log test configuration is exactly a choice of $\Xi \in H^0 (X,T_X)$ that admits a holomorphy potential and preserves $D$ (i.e.~is tangential to $D$).
\end{remark}

With these preparations, the log $K$-stability can now be defined as follows.

\begin{definition}
A pair $((X,D);L)$ is called \textbf{log $K$-semistable with cone angle $2 \pi \beta$} if $DF (\mathcal{X} , \mathcal{D} , \mathcal{L} , \beta) \ge 0$ for any log test configuration $((\mathcal{X} , \mathcal{D}); \mathcal{L})$ for $((X,D);L)$. It is called \textbf{log $K$-polystable with cone angle $2 \pi \beta$} if it is log $K$-semistable with cone angle $2 \pi \beta$ and $DF (\mathcal{X} , \mathcal{D} , \mathcal{L} , \beta) = 0$ if and only if $((\mathcal{X} , \mathcal{D}); \mathcal{L})$ is product. It is called \textbf{log $K$-stable with cone angle $2 \pi \beta$} if it is log $K$-semistable with cone angle $2 \pi \beta$ and $DF (\mathcal{X} , \mathcal{D} , \mathcal{L} , \beta) = 0$ if and only if $((\mathcal{X} , \mathcal{D}); \mathcal{L})$ is trivial.
\end{definition}

\begin{remark}
We need some restriction on the singularities of $\mathcal{X}$ and $\mathcal{D}$ to define log $K$-stability (cf.~Remark  \ref{pathsingkstble}), when the log test configuration is not product or trivial (cf.~\cite{os}), but we do not discuss this issue since only the product log test configurations will be important for us later.
\end{remark}

\begin{remark}
While we shall see later (cf.~Corollary \ref{corfutcsmtypdist} and Remark \ref{exttermlogfutsing} that follows) in differential-geometric context how the ``extra'' terms $(1 - \beta) \left( \tilde{b}_0 - \frac{\tilde{a}_0}{a_0} b_0\right)$ in (\ref{defoflogdfinv}) (or the corresponding terms in (\ref{dfsd})) come out, they come out naturally in the blow-up formalism in algebraic geometry (cf.~\cite[Theorem 3.7]{os}).  
\end{remark}

The following may be called the \textbf{log Yau--Tian--Donaldson conjecture}. This seems to be a folklore conjecture in the field, and is mentioned in e.g.~\cite{dervan, keller14}.
\begin{conjecture} \label{logdtyconj}
$((X,D);L)$ is log $K$-polystable with cone angle $2 \pi \beta$ if and only if $X$ admits a cscK metric in $c_1(L)$ with cone singularities along $D$ with cone angle $2 \pi \beta$.
\end{conjecture}

\begin{remark} \label{logdtyfanosolv}
	When $X$ is Fano with $L = - \lambda K_X$ (for some $\lambda \in \mathbb{N}$) and $D \in |- \lambda K_X|$, this conjecture is solved in the affirmative. Berman \cite{ber} first proved that the existence of conically singular \ke metric with cone angle $2 \pi \beta$ implies log $K$-stability of $((X, D) ; - \lambda K_X)$ with cone angle $2 \pi \beta$. Chen--Donaldson--Sun \cite{cds1, cds2, cds3} proved that the log $K$-stability with cone angle $2 \pi \beta$ implies the existence of the conically singular \ke metric with cone angle $2 \pi \beta$, in the course of proving the ``ordinary'' version of the Yau--Tian--Donaldson conjecture (Conjecture \ref{usualdtyconj}) for Fano manifolds; see also Tian \cite{tian12}. 
\end{remark}

Let $f \in C^{\infty} (X , \cx)$ be the holomorphy potential, with respect to $\omega$, of the holomorphic vector field $\Xi_f$ on $X$ which preserves $D$. Recall that we use the sign convention $\iota (\Xi_f) \omega = - \bar{\partial} f$ for the holomorphy potential. Let $((\mathcal{X} , \mathcal{D}); \mathcal{L})$ be the product log test configuration defined by $\Xi_f$ (cf.~Remark \ref{1psinaut0eqprtclog}). In this case, a straightforward adaptation of the argument in \cite[\S 2]{dontoric} shows the following.

\begin{theorem}
\emph{(Donaldson \cite{dontoric, donconic})} \label{dfdiff}
The log Donaldson--Futaki invariant reduces to the following differential-geometric formula
\begin{align}
DF (\mathcal{X} , \mathcal{D} , \mathcal{L} , \beta) &= \textup{Fut}_{D, \beta} (\Xi_f , [ \omega ])  \label{dfsd} \\
&:= \frac{1}{2 \pi} \textup{Fut} (\Xi_f , [ \omega ]) -  (1 - \beta) \left( \int_D f \frac{\omega^{n-1}}{(n-1)!} - \frac{\textup{Vol}(D , \omega)}{\textup{Vol}(X , \omega)} \int_X f \frac{\omega^n}{n!} \right) ,  \notag
\end{align}
defined for some (in fact any) smooth \kah metric $\omega \in c_1(L)$, when the log test configuration $((\mathcal{X} , \mathcal{D}); \mathcal{L})$ is product, defined by the holomorphic vector field $\Xi_f$ on $X$ which preserves $D$. In the formula above, $\textup{Vol}(D, \omega) := \int_D \frac{\omega^{n-1}}{(n-1)!}$ and $\textup{Vol}(X, \omega) := \int_X \frac{\omega^n}{n!}$ are the volumes given by the smooth \kah metric $\omega \in c_1 (L)$.
\end{theorem}

We may call the above $\textup{Fut}_{D , \beta}$ the \textbf{log Futaki invariant}, where the fact that $\textup{Fut}_{D, \beta} (\Xi_f , [ \omega ])$ depends only on the \kah class $[ \omega]$ (and not on the specific choice of the metric) can be shown exactly as the classical case; see e.g.~\cite[\S 4.2]{sze}.

\section{Momentum-constructed cscK metrics with cone singularities along a divisor} \label{mccmcsad}

\subsection{Background and overview} \label{bgmc}

Consider a \kah manifold $(M, \omega_M)$ of complex dimension $n-1$ together with a holomorphic line bundle $p: \mathcal{F} \to M$, endowed with a hermitian metric $h_{\mathcal{F}}$ with curvature form $\gamma :=- \ai \ddbar \log h_{\mathcal{F}}$. We first consider \kah metrics on the total space of $\mathcal{F}$, which can be regarded as an open dense subset of $\mathbb{X}:= \prj (\mathcal{F} \oplus \cx)$; we shall later impose some ``boundary conditions'' for these metrics to extend to $\mathbb{X}$. Consider a \kah metric on the total space of $\mathcal{F}$ of the form\footnote{We shall use the convention $d^c := \ai ( \bar{\partial} - \partial)$.} $p^* \omega_M + dd^c f (t)$, where $f$ is a function of $t$, and $t$ is the log of the fibrewise norm function defined by $h_{\mathcal{F}}$ serving as a fibrewise radial coordinate. A \kah metric of this form is said to satisfy the \textbf{Calabi ansatz}.

This setting was studied by Hwang \cite{hwang} in terms of the moment map associated to the fibrewise $U(1)$-action on the total space of $\mathcal{F}$; see also \cite{acgi, acgtfii, acgtfiii, hs}. Suppose that we write $\frac{\partial}{\partial \theta}$ for the generator of this $U(1)$-action, normalised so that $\exp (2 \pi \frac{\partial}{\partial \theta}) =1$, and $\tau$ for the corresponding moment map with respect to the \kah form $\omega_f :=p^* \omega_M + dd^c f (t)$. An observation of Hwang and Singer \cite{hs} was that the function $|| \frac{\partial}{\partial \theta} ||^2_{\omega_f}$ is constant on each level set of $\tau$, and hence we have a function $\varphi : I \to \rl_{\ge 0}$, defined on the range $I \subset \rl$ of the moment map $\tau$, given by
\begin{equation*}
\varphi (\tau) := \left| \left| \frac{\partial}{\partial \theta} \right| \right|^2_{\omega_f}
\end{equation*}
which is called the \textbf{momentum profile} in \cite{hs}. 

An important point of this theory is that we can in fact ``reverse'' the above construction as follows. We start with some interval $I \subset \rl$ (called \textbf{momentum interval} in \cite{hs}) and $\tau \in I$ such that
\begin{equation} \label{defofomegamtau}
\omega_M (\tau) := \omega_M - \tau \gamma >0,
\end{equation}
and write $\{ p: (\mathcal{F}, h_{\mathcal{F}}) \to (M ,\omega_M) ,I \}$ for this collection of data. We now consider a function $\varphi$ which is smooth on $I$ and positive on the interior of $I$. Proposition 1.4 (and also \S 2.1) of \cite{hs} shows that the \kah metric on $\mathcal{F}$ defined by
\begin{equation} \label{defofomegaphimc}
\omega_{\varphi} := p^* \omega_M - \tau p^* \gamma + \frac{1}{\varphi} d \tau \wedge d^c \tau = p^* \omega_M (\tau)  + \frac{1}{\varphi} d \tau \wedge d^c \tau
\end{equation}
is equal to $\omega_f =p^* \omega_M + dd^c f (t)$ satisfying the Calabi ansatz, where $(f,t)$ and $(\varphi , \tau)$ are related in the way as described in (2.2) and (2.3) of \cite{hs}.

We now come back to the projective completion $\mathbb{X}= \prj (\mathcal{F} \oplus \cx)$ of $\mathcal{F}$, and suppose that $\omega_f =p^* \omega_M + dd^c f (t)$ extends to a well-defined \kah metric on $\mathbb{X}$. In this case, without loss of generality we may write $I = [-b,b]$ for some $b>0$; $\tau =b$ (resp. $\tau = -b$) corresponds to the $\infty$-section (resp. $0$-section) of $\mathbb{X}= \prj (\mathcal{F} \oplus \cx)$, cf.~\cite[\S 2.1]{hs}. Hwang \cite{hwang} proved\footnote{See also \cite[Proposition 1.4 and \S 2.1]{hs}. The boundary condition of $\varphi$ at $\partial I = \{ \pm b \}$ will be discussed later in detail.} that the condition for $\omega_{\varphi}$ defined by (\ref{defofomegaphimc}) to extend to a well-defined \kah metric on $\mathbb{X}$ is given by the following boundary conditions for $\varphi$ at $\partial I$: $\varphi (\pm b)=0$ and $\varphi' (\pm b) = \mp2$. We can thus construct a \kah metric $\omega_{\varphi}$ on $\mathbb{X}$ from the data $\{ p: (\mathcal{F}, h_{\mathcal{F}}) \to (M ,\omega_M) ,I \}$, and such $\omega_{\varphi}$ is said to be \textbf{momentum-constructed}.

We recall the following notion.
\begin{definition} \label{defsigmaconst}
The data $\{ p: (\mathcal{F} , h_{\mathcal{F}} ) \to (M , \omega_M) , I  \}$ are said to be \textbf{$\sigma$-constant} if the curvature endomorphism $\omega_M^{-1} \gamma$ has constant eigenvalues on $M$, and the \kah metric $\omega_M(\tau)$ (on $M$) has constant scalar curvature for each $\tau \in I$.
\end{definition}

The advantage of assuming the $\sigma$-constancy is that the scalar curvature $S(\omega_{\varphi})$ of $\omega_{\varphi}$ can be written as
\begin{equation} \label{scalmomconsmc}
S(\omega_{\varphi}) = R(\tau)  - \frac{1}{2Q} \frac{\partial^2}{\partial \tau^2} (\varphi Q) (\tau)
\end{equation}
in terms of $\tau$, where
\begin{equation} \label{defofq}
Q(\tau) := \frac{\omega_M (\tau)^{n-1}}{\omega_M^{n-1}}
\end{equation}
and
\begin{equation} \label{defofr}
R(\tau) := \mathrm{tr}_{\omega_M (\tau)} \mathrm{Ric} (\omega_M)
\end{equation}
are both functions of $\tau$ by virtue of the $\sigma$-constancy hypothesis. Note that (\ref{scalmomconsmc}) means that the cscK equation $S(\omega_{\varphi}) = \cst$ is now a second order linear ODE.

In what follows, we assume that $(M, \omega_M)$ is a product of \ke manifolds $(M_i , \omega_i)$, and $\mathcal{F} := \bigotimes_{i=1}^r p_i^* K_i^{\otimes l_i} $, where $l_i \in \itg$, $p_i : M \surj M_i$ is the obvious projection, and $K_i$ is the canonical bundle of $M_i$ (we can in fact assume $l_i \in \rtn$ as long as $K_i^{\otimes l_i} $ is a genuine line bundle, rather than a $\rtn$-line bundle). It is easy to see that this satisfies the $\sigma$-constancy. We also assume that each $M_i$ is Fano, as in \cite{hwang}; this hypothesis is needed in the Appendix A of \cite{hwang}, which will also be used in \S \ref{construct}. 

We now recall the work of Hwang (cf.~\cite[Theorem 1]{hwang}), who constructed an extremal metric on $\mathbb{X} = \prj (\mathcal{F} \oplus \cx)$ in every \kah class.

\begin{theorem}
\emph{(Hwang \cite[Corollary 1.2 and Theorem 2]{hwang})} \label{thhwapcomplb}
The projective completion $\prj (\mathcal{F} \oplus \cx)$ of a line bundle $\mathcal{F} := \bigotimes_{i=1}^r p_i^* K_i^{\otimes l_i} $, over a product of \ke Fano manifolds, each with the second Betti number $1$, admits an extremal metric in each \kah class. 
\end{theorem}

\begin{remark}
We also recall that the scalar curvature of these extremal metrics can be written as $S (\omega_{\varphi} ) = \sigma_0 + \lambda \tau$ where $\sigma_0$ and $\lambda$ are constants (cf.~\cite[Lemma 3.2]{hwang}).
\end{remark}

Whether this extremal metric is in fact cscK depends on if the (classical) Futaki invariant vanishes (Theorem \ref{cldfinvdiffform}); see also e.g.~\cite[Corollary 4.22]{sze}. Hwang's argument, however, gives the following alternative viewpoint on this problem. The above formula $S (\omega_{\varphi} ) = \sigma_0 + \lambda \tau$ for the scalar curvature of the extremal metric implies that $\omega_{\varphi}$ is cscK if and only if $\lambda =0$, and hence the question reduces to whether there exists a well-defined extremal \kah metric $\omega_{\varphi}$ such that $S(\omega_{\varphi})$ has $\lambda = 0$. As Hwang \cite{hwang} shows, the obstruction for achieving this is the following boundary conditions for $\varphi$ at $\partial I = \{-b , +b \}$: $\varphi(\pm b) =0$ and $\varphi ' (\pm b) = \mp 2$. They are the conditions that must be satisfied for $\omega_{\varphi}$ to be a well-defined smooth metric on $\mathbb{X}$; $\varphi(\pm b) =0$ means that the fibres ``close up'', and $\varphi' (\pm b) = \mp 2$ means that the metric is smooth along the $\infty$-section (resp.~0-section). 

It is not possible to achieve $\lambda= 0$, $\varphi(\pm b) =0$, $\varphi ' (\pm b) = \mp 2$ all at the same time if the Futaki invariant is not zero. On the other hand, however, we can brutally set $\lambda =0$ and try to see what happens to $\varphi (\pm b)$ and $\varphi '(\pm b)$. In fact, it is possible to set $\lambda =0$, $\varphi (\pm b) =0$, and $\varphi (-b) = 2$ all at the same time\footnote{It is possible to set $\varphi (b) = - 2$ instead of $\varphi (-b) = 2$ in here, and in this case $\omega_{\varphi}$ will be smooth along the $\infty$-section with cone singularities along the $0$-section; this is purely a matter of convention. However, just to simplify the argument, we will assume henceforth that $\omega_{\varphi}$ is always smooth along the $0$-section with the cone singularities forming along the $\infty$-section.}, as discussed in \cite[\S 3.2]{hwang} and recalled in \S \ref{construct} below. Thus, we should have $\varphi ' (b) \neq - 2$ if the Futaki invariant is not zero. A crucially important point for us is that the value $- \pi \varphi' ( b)  = 2 \pi \beta$ is the angle of the \textit{cone singularities} that the metric develops along the $\infty$-section, if $\varphi$ is real analytic on $I$. This point is briefly mentioned in \cite[p2299]{hs} and seems to be well-known to the experts (cf.~\cite[Lemma 2.3]{lih12}). However, as the author could not find an explicitly written proof in the literature, the proof of this fact is provided in Lemma  \ref{lemestgmc}, \S \ref{mccone}. 

What we prove in \S \ref{construct} is that it is indeed possible to run the argument as above, namely it is indeed possible to have a cscK metric on $\mathbb{X}$ in each \kah class, at the cost of introducing cone singularities along the $\infty$-section. An important point here is that the cone angle $2 \pi \beta$ is \textit{uniquely determined in each K\"ahler class}; we can even obtain an explicit formula (equation (\ref{formcangle})) for the cone angle. 

We compute in \S \ref{logfutmc} the log Futaki invariant. The point is that the computation becomes straightforward by using the extremal metric, afforded by Theorem \ref{thhwapcomplb}. It turns out that the vanishing of the log Futaki invariant gives an equation for $\beta$ to satisfy (equation (\ref{futformcangle})); in other words, there is a unique value of $\beta$ for which the log Futaki invariant vanishes. The content of our main result, Theorem \ref{mainthmmc}, is that this value of $\beta$ agrees with the one for which there exists a momentum-constructed conically singular cscK metric with cone angle $2 \pi \beta$ (equation (\ref{formcangle})).  

\begin{remark} \label{kclmcmcvol}
The hypothesis $b_2(M_i)=1$ in Theorem \ref{mainthmmc} is to ensure that each \kah class of $\mathbb{X}$ can be represented by a momentum-constructed metric, as we now explain. Observe first that $b_2(M_i)=1$ implies $H^2(M, \rl) = \bigoplus_i \rl [ p_i^* \omega_i ]$, by recalling that every Fano manifold is simply connected (cf. \cite{debarre}). Thus recalling the Leray--Hirsch theorem, we have 
\begin{equation*}
H^2 (\mathbb{X}, \rl) = p^* H^2(M, \rl) \oplus \rl c_1(\xi) = p^* \left(  \bigoplus_i \rl [  p_i^*  \omega_i ] \right) \oplus \rl c_1(\xi) ,
\end{equation*}
i.e.~each \kah class on $\mathbb{X}$ can be written as $\sum_{i=1}^r \alpha_i p^* [ p_i^* \omega_i ] + \alpha_{r+1} c_1 (\xi)$ for some $\alpha_i >0$, where $\xi$ is the dual of the tautological bundle on $\mathbb{X}$. We can now prove (cf.~\cite[Lemma 4.2]{hwang}) that each \kah class can be represented by a momentum-constructed metric $\omega_{\phi} = p^* \omega_M - \tau p^* \gamma + \frac{1}{\phi} d \tau \wedge d^c \tau $ as follows. Observe now that the form $- \tau p^* \gamma + \frac{1}{\phi} d \tau \wedge d^c \tau$ is closed. Thus its cohomology class can be written as
\begin{equation*}
\left[ - \tau p^* \gamma + \frac{1}{\phi} d \tau \wedge d^c \tau \right] = \sum_{i=1}^r \alpha'_i p^* [ p_i^* \omega_i ] + \alpha'_{r+1} c_1 (\xi)
\end{equation*}
for some $\alpha'_i >0$. We shall prove in Lemma \ref{lemfbwv4pbps} that any momentum-constructed metric with the momentum interval $I = [-b,b]$ has fibrewise volume $4 \pi b$. This proves $\alpha'_{r+1} = 4 \pi b$. Thus, writing $\omega_{M} = \sum_{i=1}^r \tilde{\alpha}_i \omega_i$, we see that $[\omega_{\phi} ] = \sum_{i=1}^r (\alpha'_i + \tilde{\alpha}_i) p^* [ p_i^* \omega_i ] + 4 \pi b c_1 (\xi) $. Thus, given any \kah class in $\kappa \in H^2 (\mathbb{X}, \rl)$, we can choose $\tilde{\alpha}_i$ and $b$ appropriately so that $[\omega_{\phi}] = \kappa$.

\end{remark}

\begin{remark} \label{remcangle01}
We do not necessarily have $0 < \beta <1$ in Theorem \ref{mainthmmc}; although $\beta \ge 0$ always holds, as we prove in \S \ref{construct}, there are examples\footnote{Results with $\beta >1$ are also given in \cite{mp}.} where $\beta  > 1$. Indeed, when we take $M = \prj^1 \times \prj^1$, $\omega_M = p_1^* \omega_{KE} + p_2^* \omega_{KE}$ for the \ke metric $\omega_{KE} \in 2 \pi c_1 (-K_{\prj^1})$ and $\mathcal{F} = p_1^* (-K_{\prj^1}) \otimes p_2^* (2 K_{\prj^1})$, we always have $\beta >1$ as shown in Figure \ref{graph2}, by noting that $0 < b < 0.5$ gives a well-defined momentum interval.

\begin{figure}
\begin{center}
\includegraphics[clip,width=10.0cm]{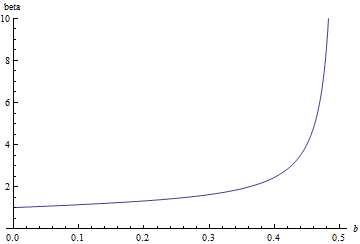}
\caption{Graph of $\beta$ as a function of $b$ for $\mathcal{F} = p_1^* (K^{-1}_{\prj^1}) \otimes p_2^* (K^2_{\prj^1})$ on $M = \prj^1 \times \prj^1$.}
\label{graph2}
\end{center}
\end{figure}

On the other hand, as shown in Figure \ref{graph1}, $\mathcal{F} = p_1^* (-2K_{\prj^1}) \otimes p_2^* (K_{\prj^1})$ with $M$ and $\omega_M$ as above, $0< b< 0.5$ implies $0.3 \lesssim \beta <1$; in particular Theorem \ref{mainthmmc} is not vacuous even if we impose an extra condition $0 < \beta <1$.

\begin{figure}
\begin{center}
\includegraphics[clip,width=10.0cm]{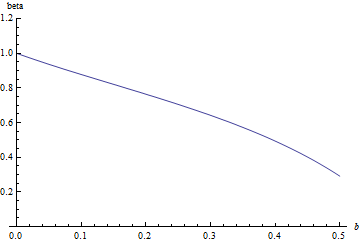}
\caption{Graph of $\beta$ as a function of $b$ for $\mathcal{F} = p_1^* (K^{-2}_{\prj^1}) \otimes p_2^* (K_{\prj^1})$ on $M = \prj^1 \times \prj^1$.}
\label{graph1}
\end{center}
\end{figure}

The author could not find an example where $\beta =0$ is achieved, which (at least heuristically) corresponds to the cuspidal singularity (cf.~\cite{guecusp}).
\end{remark}

\subsection{Some properties of momentum-constructed metrics with $\varphi' (b) = -2 \beta$} \label{mccone}
We do \textit{not} assume in this section that the $\sigma$-constancy hypothesis (cf.~Definition \ref{defsigmaconst}) is necessarily satisfied, but \textit{do} assume that $\varphi$ is real analytic.

We first prove that $ \varphi ' (b) = -2 \beta$ does indeed define a \kah metric that is conically singular along the $\infty$-section. The author thanks Michael Singer for the advice on the proof of the following lemma.
\begin{lemma} \label{lemestgmc}
	\emph{(Singer \cite{test1}, Li \cite[Lemma 2.3]{lih12})}
Suppose that $\omega_{\varphi}$ is a momentum-constructed \kah metric on $\mathbb{X} = \prj (\mathcal{F} \oplus \cx)$ with the momentum interval $I = [-b,b]$ and the momentum profile $\varphi$ that is real analytic on $I$ with $\varphi (\pm b)=0$, $\varphi ' (-b) = 2$, and $\varphi '(b) = -2 \beta$. Then $\omega_{\varphi}$ is smooth on $\mathbb{X} \setminus D$, where $D = \{ \tau = b \}$ is the $\infty$-section, and has cone singularities along $D$ with cone angle $2 \pi \beta$. Moreover, choosing the local coordinate system $(z_1 , \dots , z_{n})$ on $\mathbb{X}$ so that $D= \{ z_1 =0 \}$ and that $(z_2 , \dots , z_{n})$ defines a local coordinate system on the base $M$, $b- \tau$ can be written as a locally uniformly convergent power series
\begin{equation*}
b- \tau = A_0 |z_1|^{2 \beta} \left( 1 + \sum_{i=1}^{\infty} A_i |z_1|^{2 \beta i} \right)
\end{equation*}
around $D = \{ \tau=b \} = \{ z_1 =0  \}$, where $A_i$'s are smooth functions which depend only on the local coordinates $(z_2 , \dots , z_n)$ on $M$, and $A_0 >0$ is in addition bounded away from 0.

Thus $\varphi (\tau)$ can be written as a locally uniformly convergent power series around $D$
\begin{equation} \label{expnsofvarphibetalem}
\varphi (\tau) = 2 \beta A'_1 |z_1|^{2\beta}  + \sum_{i=2}^{\infty} A'_i |z_1|^{2 \beta i},
\end{equation}
where $A'_i$'s are smooth functions which depend only on the local coordinates $(z_2 , \dots , z_n)$ on $M$, and $A'_1 >0$ is in addition bounded away from 0. This means that the metric $g_{\varphi}$ corresponding to $\omega_{\varphi}$ satisfies the following estimates around $D$:
\begin{enumerate}
\item $(g_{\varphi})_{1 \bar{1}} = O(|z_1|^{2 \beta -2})$, 
\item $(g_{\varphi})_{1 \bar{j}} = O(|z_1|^{2 \beta -1})$ ($j \neq 1$), 
\item $(g_{\varphi})_{i \bar{j}} = O(1)$ ($i,j \neq 1$),
\end{enumerate}
i.e.~$\omega_{\varphi}$ is a \kah metric with cone singularities along $D$ with cone angle $2 \pi \beta$ (cf.~Definition \ref{defcsmcgen}).
\end{lemma}

\begin{remark}
Since we can expand $\tau$ and $\varphi (\tau)$ in the powers of $|z_1|^{2 \beta}$, we see from the estimates for $g_{\varphi}$ that the \kah potential for $\omega_{\varphi}$	 is an element of $C^{4 , \alpha, \beta}$ as defined e.g.~in \cite{kellerzheng,lizheng}.
\end{remark}

\begin{proof}
Since Lemma 2.5 and Proposition 2.1 in \cite{hwang} imply that $\omega_{\varphi}$ is smooth on $\mathbb{X} \setminus D$, we only have to check that the condition $\varphi '(b) = -2 \beta$ implies that $\omega_{\varphi}$ has cone singularities along $D$ with cone angle $2 \pi \beta$.

Writing $t$ for the log of the fibrewise length measured by $h_{\mathcal{F}}$, we have
\begin{equation} \label{dtdtauphi}
dt = \frac{d \tau}{ \varphi ( \tau)} ,
\end{equation}
by recalling the equation (2.2) in \cite{hs}. We now write $\varphi$ as a convergent power series in $b- \tau$ around $\tau = b$ as
\begin{equation} \label{expnsofvarphibeta}
\varphi (\tau) = 2 \beta (b - \tau) + \sum_{i=2}^{\infty} a'_i (b - \tau)^i,
\end{equation}
since we assumed that $\varphi$ is real analytic, where $a'_i$'s are real numbers. Note that the coefficient of the first term is fixed by the boundary condition $\varphi' (b) = - 2 \beta$. This gives
\begin{equation*}
t = \frac{1}{2} \log h_{\mathcal{F}} (\zeta , \zeta) = - \frac{1}{2 \beta} \log (b- \tau) + \sum_{i=2}^{\infty} a''_i (b - \tau)^{i-1} + \cst
\end{equation*}
with some real numbers $a''_i$, where $\zeta$ is a fibrewise coordinate on $\mathcal{F} \to M$.

On the other hand, since $\zeta$ is a fibrewise coordinate on $\mathcal{F} \to M$, it gives a fibrewise local coordinate of $\prj (\mathcal{F} \oplus \cx) \to M$ around the 0-section; in other words, at each point $p \in M$, $\zeta$ gives a local coordinate on each fibre $\prj^1$ in the neighbourhood containing $ 0 = [0:1] \in \prj^1$. Since $\tau =b$ defines the $\infty$-section of $\prj (\mathcal{F} \oplus \cx) \to M$, it is better to pass to the local coordinates on $\prj^1$ in the neighbourhood containing $ \infty = [1:0] \in \prj^1$ in order to evaluate the asymptotics as $\tau \to b$. The coordinate change is of course given by $\zeta \mapsto 1 / \zeta =: z_1$, and hence we have
\begin{equation*}
\frac{1}{2} \log h_{\mathcal{F}} (\zeta , \zeta) = \frac{1}{2} \phi_{\mathcal{F}} - \frac{1}{2} \log |z_1|^2 = - \frac{1}{2 \beta} \log (b- \tau) + \sum_{i=2}^{\infty} a''_i (b - \tau)^{i-1} + \cst
\end{equation*}
by writing $h_{\mathcal{F}} = e^{\phi_{\mathcal{F}}}$ locally around a point $p \in M$. This means that there exists a smooth function $A = A (z_2 , \dots , z_n)$ which is bounded away from 0 and depends only on the coordinates $(z_2 , \dots , z_n)$ on $M$ such that
\begin{equation*}
|z_1|^2 = A (b-\tau)^{\frac{1}{\beta}} \left( 1 + \sum_{i=1}^{\infty} a'''_i (b-\tau)  \right) ,
\end{equation*}
with some real numbers $a'''_i$ and hence, by raising both sides of the equation to the power of $\beta$ and applying the inverse function theorem, we have
\begin{equation} \label{expnsofbtauinz}
b- \tau = A_0 |z_1|^{2 \beta} \left( 1 + \sum_{i=1}^{\infty} A_i |z_1|^{2 \beta i} \right)
\end{equation}
as a locally uniformly convergent power series around $D = \{ \tau=b \} = \{ z_1 =0  \}$, where each $A_i = A_i(z_2 , \dots , z_n)$ is a smooth function which depends only on the coordinates $(z_2 , \dots , z_n)$ on $M$, and $A_0>0$ is in addition bounded away from 0. In particular, we have $b-\tau = O(|z_1|^{2 \beta})$, and combined with the equation (\ref{expnsofvarphibeta}), we thus get the result (\ref{expnsofvarphibetalem}) that we claimed.


We now evaluate $\frac{1}{\varphi} d \tau \wedge d^c \tau$ in $\omega_{\varphi} = p^* \omega_M - \tau p^* \gamma + \frac{1}{\varphi} d \tau \wedge d^c \tau$. The above equation (\ref{expnsofbtauinz}) means 
\begin{equation*}
\partial (b - \tau) = A_0 \beta |z_1|^{2 \beta-2} \bar{z}_1 B_1 d z_1 + |z_1|^{2 \beta}   \sum_{i=2}^{n} B_{2,i} dz_i 
\end{equation*}
and 
\begin{equation*}
\bar{\partial} (b - \tau) = A_0 \beta |z_1|^{2 \beta-2} {z_1} B_1 d \bar{z}_1 + |z_1|^{2 \beta} \sum_{i=2}^{n} \overline{B_{2,i}}  d \bar{z}_i, 
\end{equation*}
where we defined $B_1$ and $B_2$ as $B_1 : = 1 + \sum_{i=1}^{\infty} i A_i |z_1|^{2 \beta i}$ and $B_{2,i} := \frac{\partial}{\partial z_i} \left( A_0 + A_0 \sum_{j=1}^{\infty} A_i |z_1|^{2 \beta j} \right)$. We thus have
\begin{align}
d \tau \wedge d^c \tau =& d (b - \tau) \wedge d^c (b - \tau) \notag \\
=& 2 A_0^2 B_1^2 \beta^2 |z_1|^{4 \beta-2}  \ai d z_1 \wedge d \bar{z}_1 \notag \\
&+ 2 \beta |z_1|^{4 \beta-2} \bar{z}_1  A_0B_1 \sum_{i=2}^{n} \overline{B_{2,i}}   \ai d z_1 \wedge d \bar{z}_i + c.c. + O(|z_1|^{4 \beta}) \label{estdtau}.
\end{align}
where $O(|z_1|^{4 \beta})$ stands for a term of the form
\begin{align*}
|z_1|^{4 \beta} &\times (\text{smooth function in } (z_2 , \dots , z_n)) \\
&\times (\text{locally uniformly convergent power series in } |z_1|^{2 \beta}) ,
\end{align*}
and $c.c.$ stands for complex conjugate of the preceding terms.

We now estimate the behaviour of each component $(g_{\varphi})_{i \bar{j}} $ of the \kah metric $\omega_{\varphi} =  \sum_{i,j=1}^{n}(g_{\varphi})_{i \bar{j}} \ai dz_i \wedge d\bar{z}_j$ in terms of the local holomorphic coordinates $ (z_1 , z_2 , \dots , z_{n})$ on $\mathbb{X}$. The above computation with $\varphi (\tau) = O(|z_1|^{2 \beta})$ means that $(g_{\varphi})_{1 \bar{1}} = O(|z_1|^{2 \beta -2})$, $(g_{\varphi})_{1 \bar{j}} = O(|z_1|^{2 \beta -1})$ ($j \neq 1$), $(g_{\varphi})_{i \bar{j}} = O(1)$ ($i,j \neq 1$) as it approaches the $\infty$-section, proving that $\omega_{\varphi}$ has cone singularities of cone angle $2 \pi \beta$ along $D$. 

\end{proof}

We also see that the above means that the inverse matrix $(g_{\varphi})^{\ijbar}$ satisfies the following estimates. 

\begin{lemma} \label{lemginvmc}
Suppose that $g_{\varphi}$ is a momentum-constructed conically singular \kah metric with cone angle $2 \pi \beta$ along $D = \{ z_1 =0 \}$, with the real analytic momentum profile $\varphi$. Then, around $D$,
\begin{enumerate}
\item $(g_{\varphi})^{1 \bar{1}} = O(|z_1|^{2 -2 \beta})$ ,
\item $(g_{\varphi})^{1 \bar{j}} = O(|z_1|)$ if $j \neq 1$ ,
\item $(g_{\varphi})^{i \bar{j}} = O(1)$ if $i, j \neq 1$.
\end{enumerate}
Thus, $\Delta_{\omega_{\varphi}} f = \sum_{i,j=1}^n (g_{\varphi})^{ i \bar{j}} \frac{\partial^2}{ \partial z_i \partial \bar{z}_j} f$ is bounded if $f$ is a smooth function on $\mathbb{X}$. Also, if $f'$ is a smooth function on $\mathbb{X} \setminus D$ that is of order $|z_1|^{2 \beta}$ around $D$, then $\Delta_{\omega_{\varphi}} f'= O(1) + O(|z_1|^{2 \beta}) $. In particular, $\Delta_{\omega_{\varphi}} f'$ remains bounded on $\mathbb{X} \setminus D$.


\end{lemma}

We now prove the following estimates on the Ricci curvature and the scalar curvature of $\omega_{\varphi}$ around the $\infty$-section, i.e. when $\tau \to b$. 

\begin{lemma} \label{lemricmc}
Choosing a local coordinate system $(z_1 , \dots , z_{n})$ on $\mathbb{X}$ so that $ z_1$ is the fibrewise coordinate which locally defines the $\infty$-section $D$ by $z_1 =0$ and that $(z_2 , \dots , z_{n})$ defines a local coordinate system on the base $M$, we have, around $D$,
\begin{enumerate}
\item $\textup{Ric} (\omega_{\varphi})_{1 \bar{1}} = O(1) + O(|z_1|^{2 \beta -2})$, 
\item $\textup{Ric} (\omega_{\varphi})_{1 \bar{j}} = O(1) + O(|z_1|^{2 \beta -1})$ ($j \neq 1$), 
\item $\textup{Ric} (\omega_{\varphi})_{i \bar{j}} = O(1) + O(|z_1|^{2 \beta })$ ($i,j \neq 1$),
\end{enumerate}
for a momentum-constructed metric $\omega_{\varphi}$ with smooth $\varphi$ and $\varphi '(b) = -2 \beta$. In particular, combined with Lemma \ref{lemginvmc}, we see that $S(\omega_{\varphi})$ is bounded on $\mathbb{X} \setminus D$ if $0 < \beta <1$.
\end{lemma}

\begin{proof}
First note that (cf.~Lemma \ref {lemestgmc}, the equation (\ref{defofomegaphimc}), and \cite[p2296]{hs}) $\omega_{\varphi}^{n} = \frac{n}{\varphi} p^* \omega_M (\tau)^{n-1} \wedge d \tau \wedge d^c \tau$ is of order
\begin{equation*} 
\omega_{\varphi}^{n}  = |z_1|^{2 \beta-2} F p^* \omega_M(\tau)^{n-1} \wedge \ai d z_1  \wedge d \bar{z}_1,
\end{equation*}
where $F$ stands for some locally uniformly convergent power series in $|z_1|^{2 \beta}$ that is bounded from above and away from 0 on $\mathbb{X} \setminus D$ (this follows from Lemma \ref{lemestgmc}).


Writing $\omega_0 := p^* \omega_M + \delta \omega_{FS}$ for a reference \kah form on $\mathbb{X} = \prj(\mathcal{F} \oplus \cx)$, where $\omega_{FS}$ is a fibrewise Fubini-Study metric and $\delta>0$ is chosen to be small enough so that $\omega_0 >0$, we thus have
\begin{equation*} 
\frac{\omega_{\varphi}^{n}}{\omega_0^{n}} = \frac{p^* \omega_M (\tau)^{n-1}}{p^* \omega_M^{n-1}} |z_1|^{2 \beta -2} F'
\end{equation*}
with another locally uniformly convergent power series $F'$ in $|z_1|^{2 \beta}$ on $\mathbb{X} \setminus D$, which is bounded from above and away from 0 (note also that the derivatives of $F'$ in the $z_1$-direction are not necessarily bounded on $\mathbb{X} \setminus D$ due to the dependence on $|z_1|^{2 \beta}$; they may have a pole of fractional order along $D$). Recalling (\ref{defofomegamtau}), we see that $p^* \omega_M (\tau)^{n-1} / \omega_M^{n-1}$ depends polynomially on $\tau$. We thus have a locally uniformly convergent power series
\begin{equation} \label{volfmc}
\frac{\omega_{\varphi}^{n}}{\omega_0^{n}} =  |z_1|^{2 \beta -2} \left( F_0 + \sum_{j=1}^{\infty} F_j |z_1|^{2 \beta j} \right)
\end{equation}
with some smooth functions $F_j$ depending only on the coordinates $(z_2 , \dots , z_n)$ on $M$, where $F_0$ is also bounded away from 0.

Choosing a local coordinate system $(z_1 , \dots , z_{n})$ on $\mathbb{X}$ so that $D = \{ z_1 =0 \}$ and that $(z_2 , \dots , z_{n})$ defines a local coordinate system on the base $M$, we evaluate the order of each component of the Ricci curvature $\textup{Ric} (\omega_{\varphi}) = - \ai \ddbar \log \left( \frac{\omega_{\varphi}^{n}}{\omega_0^{n}} \right)$ around the $\infty$-section, i.e. as $\tau \to b$. Writing $\textup{Ric} (\omega_{\varphi})_{i \bar{j}} = - \frac{\partial^2}{\partial z_i \partial \bar{z}_j} \log \left( \frac{\omega_{\varphi}^{n}}{\omega_0^{n}} \right)$ and noting $\frac{\partial^2}{\partial z_i \partial \bar{z}_j} \log |z_1|^2=0$ on $X \setminus D$ for all $i,j$, we see that $\textup{Ric} (\omega_{\varphi})_{1 \bar{1}} = O(1) + O(|z_1|^{2 \beta -2})$, $\textup{Ric} (\omega_{\varphi})_{1 \bar{j}} = O(1) + O(|z_1|^{2 \beta -1})$ ($j \neq 1$), and $\textup{Ric} (\omega_{\varphi})_{i \bar{j}} = O(1) + O(|z_1|^{2 \beta })$ ($i,j \neq 1$). In particular, we see that $S(\omega_{\varphi})$ is bounded if $0 < \beta <1$.

\end{proof}

\subsection{Proof of Theorem \ref{mainthmmc}} \label{mainlfmcpf}

\subsubsection{Construction of conically singular cscK metrics on $\mathbb{X}= \prj (\mathcal{F} \oplus \cx)$} \label{construct}


We start from recalling the materials in \S 3.2 of \cite{hwang}, particularly \cite[Propositions 3.1 and 3.2]{hwang}. We first define a function
\begin{equation} \label{defofphiextmcmc}
\phi (\tau) := \frac{1}{Q (\tau)} \left( 2 (\tau+b) Q(-b) -2 \int_{-b}^\tau (\sigma_0 + \lambda x - R(x))(\tau-x) Q(x) dx \right)
\end{equation}
where $Q (\tau)$, $R(\tau)$ are defined as in (\ref{defofq}) and (\ref{defofr}). These being functions of $\tau$ follows from $\sigma$-constancy (Definition \ref{defsigmaconst}). We re-write this as
\begin{equation} \label{equalphiqr}
(\phi Q)(\tau) = 2 (\tau+b) Q(-b) -2 \int_{-b}^\tau (\sigma_0 + \lambda x - R(x))(\tau-x) Q(x) dx ,
\end{equation}
and differentiate both sides of (\ref{equalphiqr}) twice, to get
\begin{equation} \label{difftwequalphiqr}
R (\tau) - \frac{1}{2Q} \frac{\partial^2}{\partial \tau^2}(\phi Q) (\tau) = \sigma_0 + \lambda \tau .
\end{equation}
We can show, as in \cite[Proposition 3.1]{hwang}, that there exist constants $\sigma_0$ and $\lambda$ such that $\phi$ satisfies $\phi (\pm b)=0$, $\phi ' (\pm b) = \mp 2$, and $\phi (\tau) >0$ if $\tau \in (-b,b)$; namely that $\phi$ defines a smooth momentum-constructed metric $\omega_{\phi}$. We thus have $S(\omega_{\phi}) =  \sigma_0 + \lambda \tau$, by recalling (\ref{scalmomconsmc}) and (\ref{difftwequalphiqr}), so that $\omega_{\phi}$ is extremal. 

Roughly speaking, our strategy is to ``brutally substitute $\lambda=0$'' in the above to get a cscK metric with cone singularities along the $\infty$-section. More precisely, we aim to solve the equation
\begin{equation} \label{equalphiqrpsi}
R (\tau) - \frac{1}{2Q} \frac{\partial^2}{\partial \tau^2}(\varphi Q) (\tau) = \sigma'_0 
\end{equation}
with some constant $\sigma'_0$, for a profile $\varphi$ that is strictly positive on the interior $(-b,b)$ of $I$ with boundary conditions $\varphi (b) = \varphi (-b) =0$ and $\varphi' (-b)=-2$. The value $\varphi' (b)$ has more to do with the cone singularities of the metric $\omega_{\varphi}$, and we shall see at the end that the metric $\omega_{\varphi}$ associated to such $\varphi$ defines a \kah metric with cone singularities along the $\infty$-section with cone angle $-  \pi \varphi'(b) = 2 \pi \beta$.

Since
\begin{equation*}
\varphi (\tau) := \frac{1}{Q (\tau)} \left( 2 (\tau+b) Q(-b) -2 \int_{-b}^\tau (\sigma'_0  - R(x))(\tau-x) Q(x) dx \right)
\end{equation*}
certainly satisfies the equation (\ref{equalphiqrpsi}), we are reduced to checking the boundary conditions at $\partial I$ and the positivity of $\varphi$ on the interior of $I$. Note first that the equality
\begin{equation} \label{equalphiqrpsisolq}
(\varphi Q) (\tau) = 2 (\tau+b) Q(-b) -2 \int_{-b}^\tau (\sigma'_0  - R(x))(\tau-x) Q(x) dx 
\end{equation}
immediately implies that $\varphi (-b)=0$ and $\varphi '(-b) =2$ are always satisfied. Imposing $\varphi (b)=0$, we get
\begin{equation} \label{bcond}
0=2bQ(-b) - \int_{-b}^b (\sigma'_0  - R(x)) (b-x) Q(x) dx 
\end{equation}
from (\ref{equalphiqrpsisolq}), which in turn determines $\sigma'_0$. Differentiating both sides of (\ref{equalphiqrpsisolq}) and evaluating at $b$, we also get
\begin{equation} \label{bbcond}
\varphi '(b) Q(b) = 2 Q(-b) -2 \int_{-b}^b (\sigma'_0  -R(x))Q(x) dx .
\end{equation}
Writing $A:= \int_{-b}^b Q(x)dx$ and $B:= \int_{-b}^b xQ(x)dx$ we can re-write (\ref{bcond}), (\ref{bbcond}) as
\begin{equation} \label{lambdasigma}
\begin{pmatrix}
A \sigma'_0 \\
B \sigma'_0
\end{pmatrix}
=
\begin{pmatrix}
Q(-b) - \varphi' (b) Q(b) /2 + \int_{-b}^b R(x) Q(x) dx \\
-b Q(-b) - b \varphi' (b) Q(b) /2 + \int_{-b}^bx  R(x) Q(x) dx
\end{pmatrix} ,
\end{equation}
which can be regarded as an analogue of the equations (26) and (27) in \cite{hwang}. The consistency condition $B (A \sigma'_0) = A (B \sigma'_0)$ gives an equation for $\varphi '(b)$, which can be written as
\begin{equation}
- \frac{\varphi '(b)}{2} =\frac{Q(-b)(bA +B) - A \int_{-b}^b x R(x) Q(x)dx +B \int_{-b}^b R(x)Q(x)dx}{Q(b)(bA-B)}. \label{formcangle}
\end{equation}
%

Summarising the above argument, we have now obtained a profile function $\varphi$ which solves  (\ref{equalphiqrpsi}) with boundary conditions $\varphi (b) = \varphi (-b) =0$, $\varphi' (-b)=-2$, and $\varphi' (b)$ as specified by (\ref{formcangle}). Now, Hwang's argument \cite[Appendix A]{hwang} applies word by word to show that $\varphi$ is strictly positive on the interior of $I$, and hence it now remains to show that the \kah metric $\omega_{\varphi}$ has cone singularities along the $\infty$-section. Since $Q(\tau)$ is a polynomial in $\tau$ and $R(\tau)$ is a rational function in $\tau$ (with no poles when $\tau \in [-b,b]$), we see from (\ref{equalphiqrpsi}) that $\varphi$ is real analytic on $I = [-b,b]$ by the standard ODE theory. Thus the value $-  \pi \varphi'(b) = 2 \pi \beta$ is the angle of the cone singularities that $\omega_{\varphi}$ develops along the $\infty$-section of $\mathbb{X} = \prj (\mathcal{F} \oplus \cx)$, by Lemma  \ref{lemestgmc}. This completes the construction of the momentum-constructed conically singular metric $\omega_{\varphi}$, with cone angle $-  \pi \varphi'(b) = 2 \pi \beta$ as specified by (\ref{formcangle}).

We also see $\varphi' (b) \le 0$ since otherwise $\varphi ' (-b) > 0$, $\varphi' (b) > 0$, and $\varphi (\pm b)=0$ imply that $\varphi$ has to have a zero in $(-b,b)$, contradicting the positivity $\varphi >0$ on $(-b,b)$. Hence $\beta \ge 0$.

Finally, we identify the \kah class $[\omega_{\varphi}] \in H^2 (\mathbb{X} , \rl)$ of the momentum-constructed conically singular cscK metric $\omega_{\varphi}$. We first show that the restriction $\omega_{\varphi} |_{\text{fibre}}$  of $\omega_{\varphi}$ to each fibre has (fibrewise) volume $4 \pi b$. This is well-known when the metric is smooth, but we reproduce the proof here to demonstrate that the same argument works even when $\omega_{\varphi}$ has cone singularities. Related discussions can also be found in \S \ref{iotlficmccsmss} (see Lemma \ref{summcmcssmkcviavhi} in particular).

\begin{lemma} \emph{(\cite[\S 4]{hwang} or \cite[\S 2.1]{hs})} \label{lemfbwv4pbps}
Suppose that $\omega_{\varphi}$ is a (possibly conically singular) momentum-constructed metric with the momentum profile $\varphi : [-b , b] \to \rl_{\ge 0}$. Then the fibrewise volume of $\omega_{\varphi}$ is given by $4 \pi b$.
\end{lemma}

\begin{proof}
The equation (\ref{dtdtauphi}) means that the restriction of $\omega_{\varphi}$ at each fibre (which is isomorphic to $\prj^1$) is given by (cf.~equation (2.5) in \cite{hs})
\begin{equation*}
\omega_{\varphi} |_{\text{fibre}} = \frac{1}{2} \varphi (\tau) |\zeta|^{-2} \ai d \zeta \wedge d \bar{\zeta} =\varphi (\tau) r^{-2} rdr \wedge d \theta
\end{equation*}
where $\zeta = r e^{ \ai \theta}$ is a holomorphic coordinate on each fibre ($| \cdot |$ denotes the fibrewise Euclidean norm defined by $h_{\mathcal{F}}$; see \cite[\S 2.1]{hs} for more details). By using (\ref{dtdtauphi}), we can re-write this as
\begin{equation} \label{omvpmcfbwvol}
\omega_{\varphi} |_{\text{fibre}} =  \frac{d \tau}{d t} r^{-1} dr \wedge d \theta = \frac{d \tau}{d r}  dr \wedge d \theta
\end{equation}
since $t = \log r$. Integrating this over the fibre, we get
\begin{equation*}
\int_{\text{fibre}} \omega_{\varphi} = 2 \pi \int_{0}^{\infty} \frac{d \tau}{d r}  dr =  2 \pi \int_{-b}^b d \tau = 4 \pi b
\end{equation*}
since $\tau =b$ corresponds to $\infty \in \prj^1$ and $\tau = -b$ to $0 \in \prj^1$.
\end{proof}

Thus we can write $[\omega_{\varphi}] = \sum_{i=1}^r \alpha_i p^* [ p_i^* \omega_i ]+ 4 \pi b c_1 (\xi)$ for some $\alpha_i >0$, in the notation used in Remark \ref{kclmcmcvol}. Since the same proof applies to the smooth metric $\omega_{\phi}$, we also have $[\omega_{\phi}] = \sum_{i=1}^r \tilde{\alpha_i} p^* [ p_i^* \omega_i ] + 4 \pi b c_1 (\xi)$ for some $\tilde{\alpha}_i >0$. On the other hand, since $\omega_{\varphi} |_{M} = \omega_M(b) = \omega_{\phi} |_{M}$ (where $M$ is identified with the $0$-section), it immediately follows that $\alpha_i = \tilde{\alpha}_i$ for all $i$, i.e. $[ \omega_{\varphi} ] = [\omega_{\phi}]$.

\subsubsection{Computation of the log Futaki invariant} \label{logfutmc}


We again take the (smooth) momentum-constructed extremal metric $\omega_{\phi}$, with $\phi$ defined as in (\ref{defofphiextmcmc}), and write $S(\omega_{\phi}) = \sigma_0 + \lambda \tau$ for its scalar curvature. 


Recall that the generator $v_f$ of the fibrewise $U(1)$-action has $a \tau$ as its Hamiltonian function with respect to $\omega_{\phi}$ (cf.~\cite[\S 2.1]{hs}), with some $a \in \rl$ up to an additive constant which does not change $v_f$. This means that $a \tau$ (up to an additive constant) is the holomorphy potential for the holomorphic vector field $\Xi_f := v^{1,0}_f$ (cf. Remark \ref{rem11corbauth0}) which generates the complexification of the fibrewise $U(1)$-action, i.e. the fibrewise $\cx^*$-action. Thus we can take $f  = a (\tau - \bar{\tau})$, with $\bar{\tau}$ being the average of $\tau$ over $\mathbb{X}$ with respect to $\omega_{\phi}$, for the holomorphy potential $f$ in the formula (\ref{dfsd}). Then, noting that $S(\omega_{\phi}) - \bar{S} = \lambda (\tau - \bar{\tau})$, we compute the (classical) Futaki invariant as 
\begin{equation*}
\textup{Fut} (\Xi_f , [ \omega_{\phi} ]) = \int_{\mathbb{X}} a \lambda (\tau - \bar{\tau})^2 \frac{\omega_{\phi}^{n}}{n!} = 2 \pi a \lambda \textup{Vol}(M , \omega_M) \int_{-b}^b (\tau - \bar{\tau})^2 Q(\tau) d \tau
\end{equation*}
with $\textup{Vol}(M, \omega_{M}) := \int_M\frac{ \omega_M^{n-1}}{(n-1)!}$, by \cite[Lemma 2.8]{hwang}. Recalling $D = \{ \tau =b \}$, the second term in the log Futaki invariant can be obtained by computing
\begin{align*}
\int_D f \frac{\omega_{\phi}^{n-1}}{(n-1)!} = \int_D a (\tau - \bar{\tau}) \frac{\omega_{\phi}^{n-1}}{(n-1)!} &= \int_D a (b - \bar{\tau}) \frac{p^* \omega_M (b)^{n-1}}{(n-1)!} \\
&= a (b -\bar{\tau}) Q(b) \int_M \frac{ \omega_M^{n-1}}{(n-1)!} \\
&= a (b -\bar{\tau}) Q(b) \textup{Vol}(M, \omega_M)
\end{align*}
where we used
\begin{equation}
\omega_{\phi}^{n-1} = p^* \omega_M (\tau)^{n-1} + \frac{n-1}{\phi} p^* \omega_M (\tau)^{n-2} d \tau \wedge d^c \tau \label{estmcond}
\end{equation}
which was proved in \cite[p2296]{hs}, and the definitional $Q(b) = \omega_M (b)^{n-1} / \omega_M^{n-1}$ (cf. equation (\ref{dtdtauphi})). We also note the trivial equality $\int_{\mathbb{X}} f \frac{\omega_{\phi}^{n}}{n!} = \int_{\mathbb{X}} \lambda (\tau - \bar{\tau}) \frac{\omega_{\phi}^{n}}{n!} =0 $ to see that the third term of the log Futaki invariant is 0. Collecting these calculations together, the log Futaki invariant evaluated against $\Xi_f$ is given by
\begin{equation*}
\textup{Fut}_{D , \beta} (\Xi_f , [ \omega_{\phi}]) =  a \lambda \textup{Vol}(M, \omega_M) \int_{-b}^b (\tau - \bar{\tau})^2 Q(\tau) d \tau -  (1-\beta) a (b -\bar{\tau}) Q(b) \textup{Vol}(M , \omega_M) .
\end{equation*}
Thus, writing $A := \int_{-b}^b Q(\tau)d\tau$, $B := \int_{-b}^b \tau Q(\tau)d\tau$, and $C := \int_{-b}^b \tau^2 Q(\tau)d\tau$ and noting $\bar{\tau} = B/A$, setting $\textup{Fut}_{D , \beta} (\Xi_f , [\omega_{\phi}])=0$ gives an equation for the cone angle $\beta$ as 
\begin{align}
\beta &= 1- \frac{ \lambda  \int_{-b}^b (\tau - \bar{\tau})^2 Q(\tau) d \tau}{ (b -\bar{\tau}) Q(b) } \notag \\
&= \frac{Q(b)(bA-B) -\lambda \left( AC-B^2 \right)}{Q(b)(bA-B)}  \label{lfutcangledeninv}
\end{align}
Applying (\ref{bcond}) and (\ref{bbcond}) to the case of smooth extremal metric $\omega_{\phi}$, i.e. with $\phi'(b) = -2$, we get the equations (26) and (27) in \cite{hwang} which can be re-written as
\begin{equation*} 
\begin{pmatrix}
A & B \\
B & C
\end{pmatrix}
\begin{pmatrix}
\sigma_0 \\
\lambda
\end{pmatrix}
=
\begin{pmatrix}
Q(-b) + Q(b)  + \int_{-b}^b R(x) Q(x) dx \\
-b Q(-b) + b  Q(b)  + \int_{-b}^bx  R(x) Q(x) dx
\end{pmatrix} ,
\end{equation*}
and hence, noting $AC-B^2 >0$ by Cauchy--Schwarz (where we regard $Q(\tau)d \tau$ as a measure on $I = [-b,b]$), we get $\lambda$ as
\begin{align*}
	(AC-B^2)^{-1} &\left[ -B \left( Q(-b)+ Q(b) + \int_{-b}^b R(x)Q(x)dx \right) \right. \\
	&\left. + A \left( -b Q(-b) +bQ(b) +\int_{-b}^b x R(x)Q(x)dx \right) \right],
\end{align*}
and hence
\begin{align}
\beta &= \frac{Q(b)(bA-B) -\lambda \left( AC-B^2 \right)}{Q(b)(bA-B)} \notag \\
&= \frac{Q(-b)(bA+B)  + B \int_{-b}^b R(x)Q(x)dx  - A \int_{-b}^b x R(x)Q(x)dx }{Q(b)(bA-B)} \label{futformcangle}
\end{align}
which agrees with (\ref{formcangle}). This is precisely what was claimed in Theorem \ref{mainthmmc}.

\begin{remark}
In fact, in the above proof we did not need the \kah class $[ \omega_{\varphi} ]$ to be rational, since the log Futaki invariant can be defined for a nonrational \kah class as well. It seems interesting to speculate connections to the recent works on nonrational \kah classes \cite{dr16,sd16}; see also \cite[\S 4.4]{rt}. 
\end{remark}

\section{Futaki invariant computed with respect to the conically singular metrics} \label{scalmeas}

\subsection{Some estimates for the conically singular metrics of elementary form} \label{eleestcstf}

We now consider conically singular metrics of elementary form $\hat{\omega} = \omega+ \lambda \ai \ddbar |s|_h^{2 \beta}$, as defined in Definition \ref{csmefmcdef}. We collect here some estimates that we need later.


\begin{remark}
What we discuss in here is just a review of well-known results, and in fact for the most part, they are contained in \S 2 of the paper of Jeffres--Mazzeo--Rubinstein \cite{jmr} or \S 3 of the paper of Brendle \cite{bre}.
\end{remark}

Pick a local coordinate system $(z_1 , \dots , z_n)$ around a point in $X$ so that $D$ is locally given by $\{ z_1 = 0 \}$. We then write
\begin{equation*}
\hat{\omega} = \sum_{i,j} \hat{g}_{\ijbar} \ai d z_i \wedge d \bar{z}_j  = \sum_{i,j} g_{\ijbar} \ai d z_i \wedge d \bar{z}_j + \lambda \sum_{i,j} \frac{\partial^2 |s|_h^{2 \beta}}{\partial z_i \partial \bar{z}_j}  \ai d z_i \wedge d \bar{z}_j
\end{equation*}
which means
\begin{equation*}
(\hat{g}_{\ijbar})_{\ijbar}
=
\begin{pmatrix}
g_{1 \bar{1}} + O(|z_1|^{2 \beta -2}) & g_{1 \bar{2}} + O(|z_1|^{2 \beta -1}) & \hdots &  g_{1 \bar{n}} + O(|z_1|^{2 \beta -1}) \\
g_{2 \bar{1}} + O(|z_1|^{2 \beta -1}) & g_{2 \bar{2}} + O(|z_1|^{2 \beta}) & \hdots &  g_{2 \bar{n}} + O(|z_1|^{2 \beta}) \\
\vdots & \vdots & \ddots & \vdots \\
g_{n \bar{1}} + O(|z_1|^{2 \beta -1}) & g_{n \bar{2}} + O(|z_1|^{2 \beta}) & \hdots &  g_{n \bar{n}} + O(|z_1|^{2 \beta}) 
\end{pmatrix} .
\end{equation*}
Thus, writing $\hat{g}$ for the metric corresponding to $\hat{\omega}$, we have (cf.~Definition \ref{defcsmcgen})
\begin{enumerate}
\item $\hat{g}_{1 \bar{1}} = O(|z_1|^{2 \beta -2})$ ,
\item $\hat{g}_{1 \bar{j}} = O(|z_1|^{2 \beta -1})$ if $j \neq 1$ ,
\item $\hat{g}_{i \bar{j}} = O(1)$ if $i, j \neq 1$.
\end{enumerate}
The above also means that the volume form $\hat{\omega}^n$ can be estimated as (cf.~\cite[p10]{bre})
\begin{equation*}
\hat{\omega}^n = \left( |z_1|^{2 \beta -2} \sum_{j=0}^{n-1} a_j |z_1|^{2 \beta j} + \sum_{j=0}^n b_j |z_1|^{2 \beta j} \right) \omega_0^n
\end{equation*}
where $\omega_0$ is a smooth reference \kah form on $X$, $a_j$'s and $b_j$'s being smooth functions on $X$, and $a_0$ is also strictly positive. Thus we immediately have the following lemma.
\begin{lemma} \label{lemvoltyp}
We may write $\hat{\omega}^n = |z_1|^{2  - 2 \beta} \alpha$ with some $(n,n)$-form $\alpha$, which is smooth on $X \setminus D$ and bounded as we approach $D = \{ z_1 =0\} $, but whose derivatives (in $z_1$-direction) may not be bounded around $D$ due to the dependence on the fractional power $|z_1|^{2 \beta}$.
\end{lemma}

We also see, analogously to Lemma \ref{lemginvmc}, that the above means that the inverse matrix $\hat{g}^{\ijbar}$ satisfies the following estimates.

\begin{lemma} \label{lemginv}
Suppose that $\hat{g}$ is a conically singular \kah metric of elementary form with cone angle $2 \pi \beta$ along $D = \{ z_1 =0 \}$. Then, around $D$,
\begin{enumerate}
\item $\hat{g}^{1 \hat{1}} = O(|z_1|^{2 -2 \beta})$ ,
\item $\hat{g}^{1 \bar{j}} = O(|z_1|)$ if $j \neq 1$ ,
\item $\hat{g}^{i \bar{j}} = O(1)$ if $i, j \neq 1$.
\end{enumerate}
Thus, $\Delta_{\hat{\omega}}f = \sum_{i,j=1}^n \hat{g}^{ i \bar{j}} \frac{\partial^2}{ \partial z_i \partial \bar{z}_j} f$ is bounded if $f$ is a smooth function on $X$. Also, if $f'$ is a smooth function on $X \setminus D$ that is of order $|z_1|^{2 \beta}$ around $D$, then $\Delta_{\hat{\omega}} f' =  O(1) + O(|z_1|^{2 \beta}) $. In particular, $\Delta_{\hat{\omega}} f'$ remains bounded on $X \setminus D$.


\end{lemma}


We now evaluate the Ricci curvature of $\hat{\omega}$. In terms of the local coordinate system $(z_1 , \dots , z_n)$ as above, we have
\begin{align*}
\text{Ric} (\hat{\omega})_{\ijbar} &= - \frac{\partial^2}{ \partial z_i \partial \bar{z}_j} \log \left( \frac{\hat{\omega}^n}{\omega_0^n} \right) \\
&= - \frac{\partial^2}{ \partial z_i \partial \bar{z}_j} \log \left( |z_1|^{2 \beta -2} \sum_{j=0}^{n-1} a_j |z_1|^{2 \beta j} + \sum_{j=0}^n b_j |z_1|^{2 \beta j} \right)   .
\end{align*}

Since $\ddbar \log |z_1|^2 =0$ on $X \setminus D$, we have 
\begin{equation*}
\text{Ric} (\hat{\omega})_{\ijbar} =  - \frac{\partial^2}{ \partial z_i \partial \bar{z}_j} \log \left( \sum_{j=0}^{n-1} a_j |z_1|^{ 2 \beta j} + \sum_{j=0}^n b_j |z_1|^{2 - 2 \beta + 2 \beta j} \right).
\end{equation*}
Note now that we can write
\begin{align}
&\log \left( \sum_{j=0}^{n-1} a_j |z_1|^{ 2 \beta j} + \sum_{j=0}^n b_j |z_1|^{2 - 2 \beta + 2 \beta j} \right) \notag \\
&= F_0 + \log \left( O(1) + O(|z_1|^{2 - 2 \beta}) + O(|z_1|^{2 \beta}) \right) \notag \\
&= O(1) + O(|z_1|^{2 - 2 \beta}) + O(|z_1|^{2 \beta}) \label{tepoovpmcmf}
\end{align}
with some smooth function $F_0$, around the divisor $D$. Thus, we eventually get $\textup{Ric}(\hat{\omega})_{1 \bar{1}} = O(1) + O(|z_1|^{-2 \beta}) + O(|z_1|^{2 \beta -2})$, $\textup{Ric}(\hat{\omega})_{1 \bar{j}} = O(1) + O(|z_1|^{1 -2 \beta}) + O(|z_1|^{2 \beta -1})$ ($j \neq 1$), and $\textup{Ric}(\hat{\omega})_{j \bar{k}} = O(1)$ ($j, k \neq 1$). Together with Lemma \ref{lemginv}, this means the following.

\begin{lemma} \label{lemrictyf}
Suppose that $\hat{g}$ is a conically singular \kah metric of elementary form with cone angle $2 \pi \beta$ along $D$ locally defined by $z_1 = 0$. Then
\begin{enumerate}
\item $\textup{Ric}(\hat{\omega})_{1 \bar{1}} = O(1) + O(|z_1|^{-2 \beta}) + O(|z_1|^{2 \beta -2})$, 
\item $\textup{Ric}(\hat{\omega})_{1 \bar{j}} = O(1) + O(|z_1|^{1 -2 \beta}) + O(|z_1|^{2 \beta -1})$ ($j \neq 1$),
\item $\textup{Ric}(\hat{\omega})_{j \bar{k}} = O(1)$ ($j,k \neq 1$).
\end{enumerate}
In particular, combined with Lemma \ref{lemginv}, we see that the scalar curvature $S(\hat{\omega})$ can be estimated as $S(\hat{\omega}) = O(1) + O(|z_1|^{2 - 4 \beta})$.
\end{lemma}

\begin{remark} \label{csscall1}
We observe that the above estimate implies
\begin{align*}
\left| \int_{\Omega \setminus D} S(\hat{\omega}) \frac{\hat{\omega}^n}{n!} \right| &< \cst . \int_{\text{unit disk in } \cx} (1+ |z_1|^{2-4 \beta}) |z_1|^{2 \beta -2} \ai dz_1 \wedge d \bar{z}_1 \\
&< \cst . \int_0^1 (r^{2 \beta -1} + r^{-2 \beta +1} )dr < \infty 
\end{align*}
for any open set $\Omega \subset X$ with $\Omega \cap D \neq \emptyset$, as $0 < \beta <1 $.
\end{remark}


\subsection{Scalar curvature as a current} \label{scscaldist}
In order to compute the log Futaki invariant with respect to a conically singular metric $\omega_{sing}$, we need to make sense of $\mathrm{Ric}(\omega_{sing}) \wedge \omega_{sing}^{n-1}$ globally on $X$. However, this is not well-defined for a general conically singular metric $\omega_{sing}$, as we discuss in Remark \ref{gencsmacnd}. We thus restrict our attention to the case of conically singular metrics of elementary form $\hat{\omega}$ or the momentum-constructed conically singular metrics $\omega_{\varphi}$. Theorems \ref{scaldist} and \ref{scaldistmc} state that in these cases it is indeed possible to have a well-defined current $\mathrm{Ric} (\hat{\omega}) \wedge \hat{\omega}^{n-1}$ or $\mathrm{Ric} (\omega_{\varphi}) \wedge \omega_{\varphi}^{n-1}$ on the whole manifold, and this section is devoted to the proof of these results.


\begin{remark} \label{remcompprsts}
	There are some results in the literature that are similar to Theorems \ref{scaldist} and \ref{scaldistmc}. We discuss their similarities and differences below.
	\begin{enumerate}
		\item Li \cite[Proposition 2.16]{li15} also worked on the distributional meaning of the scalar curvature. An important difference to the above results is that Li considered the distributional term $[D] \wedge \omega^{n-1}_{sing}$ as a nonpluripolar product \cite[Lemma 2.14]{li15}, which is zero. Theorems \ref{scaldist} and \ref{scaldistmc} clarify the non-zero contribution from the distributional term $[D]$, by assuming that the conically singular metrics have the preferable forms as in Definition \ref{csmefmcdef}; Li assumed, on the other hand, a condition on the Ricci curvature as in \cite[Definition 2.7]{li15}.
		\item Theorems \ref{scaldist} and \ref{scaldistmc} also bear some similarities to the equation (4.60) in Proposition 4.2 of the paper \cite{sw} by Song and Wang. The main difference is that our theorems show that $\mathrm{Ric}(\hat{\omega}) \wedge \hat{\omega}^{n-1}$ (resp. $\mathrm{Ric} (\omega_{\varphi}) \wedge \omega_{\varphi}^{n-1}$) is a current well-defined over any open subset $\Omega$ in $X$, as opposed to just computing $ \int_X \mathrm{Ric} (\hat{\omega}) \wedge \hat{\omega}^{n-1}$ (resp. $\int_X \mathrm{Ric} (\omega_{\varphi}) \wedge \omega_{\varphi}^{n-1}$); indeed our proof is quite different to theirs, although we have in common the basic strategy of doing the integration by parts ``correctly''. 
	\end{enumerate}
\end{remark}

\begin{remark}
We decide to present the argument for the conically singular metric of elementary form $\hat{\omega}$ in parallel with the one for the momentum-constructed conically singular metric $\omega_{\varphi}$, as they have much in common.

To distinguish these two cases, we continue to denote $\omega_{\varphi}$ for a momentum-constructed conically singular metric on $\mathbb{X} = \prj (\mathcal{F} \oplus \cx)$ over a base \kah manifold $(M ,\omega_M)$, with the projection $p: (\mathcal{F} , h_{\mathcal{F}}) \to (M, \omega_M)$. We do \textit{not} necessarily assume that $p: (\mathcal{F} , h_{\mathcal{F}}) \to (M, \omega_M)$ satisfies $\sigma$-constancy (cf.~Definition  \ref{defsigmaconst}), but \textit{do} need to assume that $\varphi$ is real analytic; we will only rely on the results proved in \S \ref{mccone}, in which we did not assume $\sigma$-constancy but assumed that $\varphi$ is real analytic.


On the other hand, when we consider the conically singular metrics of elementary form $\hat{\omega} = \omega + \lambda \ai \ddbar |s|^{2 \beta}_h$, $X$ can be any (projective) \kah manifold with some smooth effective divisor $D \subset X$.
\end{remark}


\begin{remark} \label{averscaldxd}
Suppose that we write, for a conically singular metric of elementary form $\hat{\omega}$,
\begin{equation*}
\bar{S}(\hat{\omega}) := \frac{1}{\text{Vol}(X, \hat{\omega})} \int_X \mathrm{Ric} (\hat{\omega}) \wedge \frac{\hat{\omega}^{n-1}}{(n-1)!}
\end{equation*}
for the ``average of $S(\hat{\omega})$ on the whole of $X$'', where we note $\text{Vol}(X, \hat{\omega}) :=\int_X \hat{\omega}^n /n! = \int_{X \setminus D} \hat{\omega}^n/n! < \infty$ (by recalling Remark \ref{remvolcs}). We then have, from Theorem \ref{scaldist},
\begin{equation*}
\bar{S} (\hat{\omega}) = \underline{S} (\hat{\omega}) + 2  \pi (1 - \beta) \frac{\text{Vol} (D, \omega)}{\text{Vol}(X, \hat{\omega})},
\end{equation*}
where $\underline{S} (\hat{\omega}) := \int_{X \setminus D}  S (\hat{\omega}) \frac{\hat{\omega}^n}{n!} / \text{Vol} (X , \hat{\omega})$ is the average of $S(\hat{\omega})$ over $X \setminus D$, which makes sense by Remark \ref{csscall1}. Similarly, for a momentum-constructed conically singular metric $\omega_{\varphi}$, we have (by recalling Theorem \ref{scaldistmc} and Lemma  \ref{lemricmc})
\begin{align*}
\bar{S} (\omega_{\varphi}) &= \underline{S} (\omega_{\varphi}) + 2  \pi (1 - \beta) \frac{ \text{Vol} (D, p^* \omega_M(b) )}{\text{Vol}(\mathbb{X}, \omega_{\varphi})} \\
&=  \underline{S} (\omega_{\varphi}) + 2  \pi (1 - \beta) \frac{ \text{Vol} (M,  \omega_M(b) )}{\text{Vol}(\mathbb{X}, \omega_{\varphi})}.
\end{align*}

The reader is warned that the average of the scalar curvature $\bar{S} (\hat{\omega})$ computed with respect to the conically singular metrics may \textit{not} be a cohomological invariant since $\mathrm{Ric} (\hat{\omega})$ is not necessarily a de Rham representative of $c_1 (L)$ due to the cone singularities of $\hat{\omega}$, whereas $\text{Vol} (D , \omega) = \int_D c_1 (L)^{n-1} /{(n-1)!}$ certainly is. Exactly the same remark of course applies to the momentum-constructed conically singular metric $\omega_{\varphi}$. On the other hand, we can show $ \text{Vol}(X, \hat{\omega}) = \int_X c_1 (L)^n / n!$ (cf.~Lemma \ref{volinvcstf}), and $\text{Vol} (\mathbb{X} , \omega_{\varphi}) = 4 \pi b \text{Vol} (M , \omega_M)$ (cf.~Remark \ref{kclmcmcvol}) for $\mathbb{X} = \prj( \mathcal{F} \oplus \cx)$.
\end{remark}

\begin{remark} \label{gencsmacnd}
	We will use in the proof the estimates established in \S \ref{mccone} and \S \ref{eleestcstf}, and our proof will not apply to conically singular metrics in full generality. Most importantly, we do not know what the ``distributional'' component (i.e.~the second term in Theorems \ref{scaldist} and \ref{scaldistmc}) should be for a general conically singular metric $\omega_{sing}$; the proof below shows that it should be equal to $[D] \wedge \omega_{sing}^{n-1}$, $[D]$ being a current of integration over $D$, but it is far from obvious that it is well-defined (particularly so since $\omega_{sing}$ is singular along $D$). Indeed, even for the case of conically singular metrics of elementary form $\hat{\omega}$, $[D] \wedge \hat{\omega}^{n-1}$ being well-defined as a current with nontrivial contribution from $[D]$ (Lemma \ref{termtwo}) seems to be a new result (cf.~Remark \ref{remcompprsts}). 



\end{remark}

\begin{proof}[Proof of Theorems \ref{scaldist} and \ref{scaldistmc}]
The proof is essentially a repetition of the usual proof of the Poincar\'e--Lelong formula (cf.~\cite{dem}), with some modifications needed to take care of the cone singularities of $\hat{\omega}$ and $\omega_{\varphi}$.

We first consider the case of the conically singular metric of elementary form $\hat{\omega}$. We first pick a $C^{\infty}$-tubular neighbourhood $D_{0}$ around $D$ with (small but fixed) radius $\epsilon_0$, meaning that points in $D_{0}$ have distance less than $\epsilon_0$ from $D$ measured in the metric $\omega$. We then write
\begin{equation*}
\int_{\Omega} f \mathrm{Ric} (\hat{\omega}) \wedge \frac{\hat{\omega}^{n-1}}{(n-1)!} = \int_{\Omega \setminus D_{0}} f \mathrm{Ric} (\hat{\omega}) \wedge \frac{\hat{\omega}^{n-1}}{(n-1)!} + \int_{\Omega \cap D_{0}} f \mathrm{Ric} (\hat{\omega}) \wedge \frac{\hat{\omega}^{n-1}}{(n-1)!}
\end{equation*}
and apply the partition of unity on the compact manifold $\overline{\Omega \cap {D}_{0}}$ (i.e.~the closure of $\Omega \cap D_0$) to reduce to the local computation in a small open set $U \subset \Omega \cap D_0$ around the divisor $D$. Confusing $U \subset \Omega \cap D_0$ with an open set in $\cx^n$, this means that we take an open set $U$ in $\cx^n$ (by abuse of notation) endowed with the \kah metric $\omega$, where we may also assume that $U$ is biholomorphic to the polydisk $\{ (z_1 , \dots,  z_n) \mid |z_1|_{\omega} < \epsilon_0 /2 , |z_2|_{\omega} < \epsilon_0 /2 , \dots , |z_n|_{\omega} < \epsilon_0 /2 \}$, in which the divisor $D$ is given by the local equation $z_1 =0$. Thus our aim now is to show
\begin{equation*}
 \int_{U} f \mathrm{Ric} (\hat{\omega}) \wedge \frac{\hat{\omega}^{n-1}}{(n-1)!} = \int_{U \setminus \{ z_1 =0 \}} f S (\hat{\omega}) \frac{\hat{\omega}^n}{n!} + 2  \pi (1 - \beta) \int_{ \{ z_1 =0\} } f \frac{\omega^{n-1}}{(n-1)!} ,
\end{equation*}
where we recall that the partition of unity allows us to assume that $f$ is smooth and compactly supported on $U$.

Note that exactly the same argument applies to the momentum-constructed conically singular metric $\omega_{\varphi}$, by using some reference smooth metric $\omega_0$ on $X$ (in place of $\omega$) to define $D_0$. Hence our aim for the momentum-constructed conically singular metric $\omega_{\varphi}$ is to show
\begin{equation*}
 \int_{U} f \mathrm{Ric} (\omega_{\varphi}) \wedge \frac{\omega_{\varphi}^{n-1}}{(n-1)!} 
 = \int_{U \setminus \{ z_1 =0 \}} f S (\omega_{\varphi}) \frac{\omega_{\varphi}^n}{n!} + 2  \pi (1 - \beta)  \int_{ \{ z_1 =0\} }  f \frac{p^* \omega_M(b)^{n-1}}{(n-1)!} ,
\end{equation*}
for a smooth and compactly supported $f$.

For the conically singular metrics of elementary form $\hat{\omega}$, we recall Lemma \ref{lemvoltyp} and write $\hat{\omega}^n = |z_1|^{2 \beta -2} \alpha$ with some smooth bounded $(n,n)$-form $\alpha$ on $X \setminus D$,  and hence have $\ddbar \log \det (\hat{\omega}) = (\beta -1) \ddbar \log |z_1|^2 + R$ where $R$ is a 2-form which is smooth on $U \setminus \{ z_1 =0\}$ but may have a pole (of fractional order) along $\{ z_1 =0 \}$. We thus write
\begin{align}
\mathrm{Ric} (\hat{\omega}) \wedge \hat{\omega}^{n-1} &= -  \ai \ddbar \log \det (\hat{\omega}) \wedge \hat{\omega}^{n-1} \notag \\
&=  (1 - \beta) \ai \ddbar \log |z_1|^2 \wedge \hat{\omega}^{n-1} -  \ai R \wedge \hat{\omega}^{n-1} . \label{sctyfmform}
\end{align}

On the other hand, we can argue in exactly the same way, by using (\ref{volfmc}) in place of Lemma \ref{lemvoltyp}, to see that for a momentum-constructed conically singular metric $\omega_{\varphi}$, we can write
\begin{align}
\mathrm{Ric} (\omega_{\varphi}) \wedge \omega^{n-1}_{\varphi} &= - \ai \ddbar \log \det (\omega_{\varphi}) \wedge \omega_{\varphi}^{n-1} \notag \\
&=  (1 - \beta) \ai \ddbar \log |z_1|^2 \wedge \omega_{\varphi}^{n-1} -  \ai R_{\varphi} \wedge \omega_{\varphi}^{n-1}  \label{scmcform}
\end{align}
for some 2-form $R_{\varphi}$ that is smooth on $U \setminus \{ z_1 =0\}$ but may have a pole (of fractional order) along $\{ z_1 =0 \}$. 

We aim to show that these formulae (\ref{sctyfmform}) and (\ref{scmcform}) are well-defined in the weak sense. This means that we aim to show that
\begin{equation*}
\int_U f \mathrm{Ric} (\hat{\omega}) \wedge \frac{\hat{\omega}^{n-1}}{(n-1)!} 
= \int_{U} f \ai R \wedge \frac{\hat{\omega}^{n-1}}{(n-1)!} +  (1 - \beta) \int_{U}f   \ai \ddbar \log |z_1|^2 \wedge \frac{\hat{\omega}^{n-1}}{(n-1)!}
\end{equation*}
is well-defined and is equal to
\begin{equation*}
\int_{U \setminus \{ z_1 =0 \}} f S (\hat{\omega}) \frac{\hat{\omega}^n}{n!} +  2  \pi (1 - \beta) \int_{ \{ z_1 =0\} } f \frac{\omega^{n-1}}{(n-1)!}
\end{equation*}
for any smooth function $f$ with compact support in $U$. Theorem \ref{scaldist} obviously follows from this, and exactly the same argument applies to $\omega_{\varphi}$ to prove Theorem \ref{scaldistmc}.


We prove these claims as follows. Let $U_{\epsilon}$ be a subset of $U$ defined for sufficiently small $\epsilon \ll \epsilon_0$ by $U_{\epsilon} := \{ (z_1 , \dots,  z_n) \in U \mid 0< \epsilon < |z_1| \}$ (the norm in the inequality $\epsilon < |z_1|$ is given by the Euclidean metric on $\cx^n$). In Lemma \ref{termone}, we shall prove that
\begin{equation*}
- n  \int_{U} f \ai R \wedge \hat{\omega}^{n-1} = - n  \lim_{\epsilon \to 0}  \int_{U_{\epsilon}} f \ai R \wedge \hat{\omega}^{n-1} = \int_{U \setminus \{ z_1 =0 \}} f S (\hat{\omega}) \hat{\omega}^n 
\end{equation*}
for a conically singular metric of elementary form $\hat{\omega}$, and
\begin{equation*}
- n  \int_{U} f \ai R_{\varphi} \wedge \omega_{\varphi}^{n-1} = - n  \lim_{\epsilon \to 0}  \int_{U_{\epsilon}} f \ai R_{\varphi} \wedge \omega_{\varphi}^{n-1} = \int_{U \setminus \{ z_1 =0 \}} f S (\omega_{\varphi}) \omega_{\varphi}^n 
\end{equation*}
for a momentum-constructed conically singular metric $\omega_{\varphi}$, and that both of these terms are finite if $f$ is compactly supported on $U$, 

In Lemma \ref{termtwo} we shall prove
\begin{equation*}
 \int_{U}f   \ai \ddbar \log |z_1|^2 \wedge \hat{\omega}^{n-1} = 2  \pi  \int_{ \{ z_1 =0\} } f \omega^{n-1} ,
\end{equation*}
and in Lemma \ref{termtwomc} we shall prove
\begin{equation*}
\int_{U}f   \ai \ddbar \log |z_1|^2 \wedge \omega_{\varphi}^{n-1} = 2 \pi  \int_{\{ z_1 =0\}} f p^* \omega_M(b)^{n-1} ,
\end{equation*}
if $f$ is smooth. Granted these lemmas, we complete the proof of Theorems \ref{scaldist} and \ref{scaldistmc}.
\end{proof}

\begin{lemma} \label{termone}
For a conically singular metric of elementary form $\hat{\omega}$, we have
\begin{equation*}
- n  \int_{U} f \ai R \wedge \hat{\omega}^{n-1} = - n  \lim_{\epsilon \to 0}  \int_{U_{\epsilon}} f \ai R \wedge \hat{\omega}^{n-1} = \int_{U \setminus \{ z_1 =0 \}} f S (\hat{\omega}) \hat{\omega}^n 
\end{equation*}
and the integral is well-defined for any smooth function $f$ compactly supported on $U$, i.e. $ \left|  \int_U f \ai R \wedge \hat{\omega}^{n-1} \right|$ is finite.

For a momentum-constructed conically singular metric $\omega_{\varphi}$, we have
\begin{equation*}
- n  \int_{U} f \ai R_{\varphi} \wedge \omega_{\varphi}^{n-1} = - n  \lim_{\epsilon \to 0}  \int_{U_{\epsilon}} f \ai R_{\varphi} \wedge \omega_{\varphi}^{n-1} = \int_{U \setminus \{ z_1 =0 \}} f S (\omega_{\varphi}) \omega_{\varphi}^n 
\end{equation*}
and the integral is well-defined for any smooth function $f$ compactly supported on $U$.
\end{lemma}

\begin{proof}

We first consider the case of the conically singular metric of elementary form $\hat{\omega}$. Although $R$ is not bounded on the whole of $U \setminus \{ z_1 =0 \}$, Lemma \ref{lemrictyf} shows that the metric contraction of $R$ with $\hat{\omega}$ (which is equal to $S(\hat{\omega})/n$ on $X \setminus D$) satisfies
\begin{equation}
|\Lambda_{\hat{\omega}} R | < \cst .(1+ |z_1|^{2- 4\beta} ) . \label{estscal}
\end{equation}
on $U \setminus \{ z_1 =0 \}$, thus
\begin{equation*}
|R \wedge \hat{\omega}^{n-1} |_{{\omega}} \le \cst . (|z_1|^{2\beta -2} + |z_1|^{2- 4\beta + 2 \beta -2})=  \cst . (|z_1|^{2\beta -2} + |z_1|^{- 2\beta})
\end{equation*}
on $U \setminus \{ z_1 =0 \}$. Since $f$ is bounded on the whole of $U$, we see, by writing $r:=|z_1| $ and choosing a large but fixed number $A$ which depends only on $U$ and $\omega$, that
\begin{align}
\lim_{\epsilon \to 0} \left|  \int_{U_{\epsilon}} f \ai R \wedge \hat{\omega}^{n-1} \right| & \le \cst . \lim_{\epsilon \to 0} \int_{U_{\epsilon}} (|z_1|^{2\beta -2} + |z_1|^{- 2\beta}) \omega^n \notag \\
&\le \cst . \lim_{\epsilon \to 0} \int_{\epsilon < |z_1| < A } (|z_1|^{2\beta -2} + |z_1|^{- 2\beta}) \ai dz_1 \wedge d\bar{z}_1 \notag \\
&\le \cst . \lim_{\epsilon \to 0} \int_{\epsilon}^{A} (r^{2 \beta -2} + r^{-2 \beta}) rdr < \infty \notag 
\end{align}
since $0 < \beta < 1$. In other words, the above shows that the signed measure defined by $\ai R \wedge \hat{\omega}^{n-1}$ on $U$ is well-defined. Observe also
\begin{align}
\left|  \int_{U \setminus U_{\epsilon}} f \ai R \wedge \hat{\omega}^{n-1} \right| &\le \cst . \int_{U \setminus U_{\epsilon}} |f \ai R \wedge \hat{\omega}^{n-1}|_{\omega} \omega^n \notag \\
&\le \cst . \int_0^{\epsilon} \sup_{|z_1|=r} |f \ai R \wedge \hat{\omega}^{n-1}|_{\omega} rdr \label{eestscal} \\
& \le \cst . \int_{0}^{\epsilon} (r^{2 \beta -1} + r^{1 -2 \beta})dr \to 0\notag 
\end{align}
as $\epsilon \to 0$, where we used the elementary $\int_0^{\epsilon} = \int_{[0 , \epsilon]} = \int_{(0, \epsilon]}$ in (\ref{eestscal}) to apply (\ref{estscal}), by noting that $\sup_{|z_1|=r} |f \ai R \wedge \hat{\omega}^{n-1}|_{\omega}$ is continuous in $ r \in (0, \epsilon]$ and its only singularity is the pole of fractional order at $r =0$. We thus have
\begin{equation*}
\int_{U } f \ai R \wedge \hat{\omega}^{n-1} = \lim_{\epsilon \to 0} \int_{U_{\epsilon} } f \ai R \wedge \hat{\omega}^{n-1} =  \int_{U \setminus \{z_1 =0 \} } f \ai R \wedge \hat{\omega}^{n-1} 
\end{equation*}
and the above integrals are all finite. 

On the other hand, we know that $\ddbar \log |z_1|^2 =0$ on $U \setminus \{ z_1 =0 \}$, and hence, recalling (\ref{sctyfmform}), $S (\hat{\omega}) \hat{\omega}^n = -n \ai R \wedge \hat{\omega}^{n-1}$ on $U \setminus \{ z_1 =0 \}$. Thus we can write
\begin{equation*}
-n \int_{U } f \ai R \wedge \hat{\omega}^{n-1} = - n  \lim_{\epsilon \to 0}  \int_{U_{\epsilon}} f \ai R \wedge \hat{\omega}^{n-1} = \int_{U \setminus \{ z_1 =0 \}} f S(\hat{\omega}) \hat{\omega}^n
\end{equation*}
as claimed.

For the case of momentum-constructed conically singular metric $\omega_{\varphi}$, Lemma \ref{lemricmc} shows that $|\Lambda_{\omega_{\varphi}} R_{\varphi} | $ is bounded on $U \setminus \{ z_1 =0 \}$. Since this is better than the estimate (\ref{estscal}), all the following argument applies word by word. We thus establish the claim for the momentum-constructed conically singular metric.
\end{proof}


\begin{lemma} \label{termtwo}
For a conically singular metric of elementary form $\hat{\omega}$,
\begin{equation*}
 \int_{U}f   \ai \ddbar \log |z_1|^2 \wedge \hat{\omega}^{n-1} = 2  \pi  \int_{ \{ z_1 =0\} } f \omega^{n-1} ,
\end{equation*}
if $f$ is smooth and compactly supported in $U$.
\end{lemma}

\begin{remark}
Note that we cannot naively apply the usual Poincar\'e--Lelong formula, since the metric $\hat{\omega}$ is singular along $\{ z_1 =0 \}$. Note also that the integral $ \int_{ \{ z_1 =0\} } f \omega^{n-1}$ is manifestly finite.
\end{remark}

\begin{proof}
We start by re-writing
\begin{align}
&\int_{U } \ai \ddbar \log |z_1|^2 \wedge f \hat{\omega}^{n-1} \notag \\
&=   \frac{1}{2} \lim_{\epsilon \to 0} \int_{U \setminus U_{\epsilon}} dd^c \log |z_1|^2 \wedge f \hat{\omega}^{n-1} \notag \\
&= \frac{1}{2} \lim_{\epsilon \to 0}   \int_{U \setminus U_{\epsilon}} d \left( d^c \log |z_1|^2 \wedge f \hat{\omega}^{n-1} \right) + \frac{1}{2} \lim_{\epsilon \to 0}  \int_{U \setminus U_{\epsilon}}  d^c \log |z_1|^2 \wedge df \wedge \hat{\omega}^{n-1} \label{distpart}
\end{align}
since $\ddbar \log |z_1|^2 =0$ if $ |z_1| \neq 0$, where we used $d = \partial + \bar{\partial}$ and $d^c = \ai (\bar{\partial} - \partial)$.


We first claim $\lim_{\epsilon \to 0}  \int_{U \setminus U_{\epsilon}}  d^c \log |z_1|^2 \wedge df \wedge \hat{\omega}^{n-1} =0$. We start by observing that $\hat{\omega}^{n-1}$ cannot contain the term proportionate to $dz_1 \wedge d\bar{z}_1$ when we take the wedge product of it with $d^c \log |z_1|^2$ or $d \log|z_1|^2$, since it will be cancelled by them. Namely, writing $|s|^{2 \beta}_{h} = e^{\phi}|z_1|^{2 \beta}$ and defining
\begin{align}
\tilde{\omega} &:= \hat{\omega} - \lambda \ai \frac{\partial^2}{\partial z_1 \partial \bar{z}_1} (e^{\phi} |z_1|^{2 \beta}) dz_1 \wedge d \bar{z}_1 \notag \\
&= \omega  + \lambda \ai \left(  \sum_{j =2}^n \beta |z_1|^{2 \beta -2} z_1(\partial_{\bar{j}} e^{\phi}) dz_1 \wedge d \bar{z}_j + c.c. + |z_1|^{2 \beta} \eta ' \right)  \label{omegatilde}
\end{align}
where $\eta' := \ddbar e^{\phi} -  \frac{\partial^2 e^{\phi}}{\partial z_1 \partial \bar{z}_1}  dz_1 \wedge d \bar{z}_1$ is a smooth 2-form, we have $d^c \log |z_1|^2 \wedge \hat{\omega}^{n-1} = d^c \log |z_1|^2 \wedge \tilde{\omega}^{n-1}$ and $d \log |z_1|^2 \wedge \hat{\omega}^{n-1} = d \log |z_1|^2 \wedge \tilde{\omega}^{n-1}$. It should be stressed that $\tilde{\omega}$ is not necessarily closed; indeed we observe $d \tilde{\omega} = - \lambda \ai d\left(  \frac{\partial^2}{\partial z_1 \partial \bar{z}_1} (e^{\phi} |z_1|^{2 \beta}) dz_1 \wedge d \bar{z}_1 \right)$. Note also that $\tilde{\omega} \le \cst . \hat{\omega}$.

Combined with the well-known equality $d^c \log |z_1|^2 \wedge df \wedge \hat{\omega}^{n-1} = -   d \log |z_1|^2 \wedge d^c f \wedge \hat{\omega}^{n-1}$, we find
\begin{align}
 \int_{V_{\epsilon}}  d^c \log |z_1|^2 \wedge df \wedge \hat{\omega}^{n-1} \ \ \ \ \ \ \ \ \ \ \ \ \ \ \ \ & \notag \\
 = - \int_{V_{\epsilon}}  d \left (\log |z_1|^2  d^c f \wedge \tilde{\omega}^{n-1} \right) &+  \int_{V_{\epsilon}}   \log |z_1|^2 dd^c f \wedge \tilde{\omega}^{n-1} \label{bounv} \\
&-\int_{V_{\epsilon}}   \log |z_1|^2  d^c f \wedge d \tilde{\omega}^{n-1} \notag
\end{align}
where we decide to write $V_{\epsilon} := U \setminus U_{\epsilon}$.

We evaluate each term separately and show that all of them go to 0 as $\epsilon \to 0$. To evaluate the first term of (\ref{bounv}), we write $\int_{V_{\epsilon}}  d \left (\log |z_1|^2 d^c f \wedge \tilde{\omega}^{n-1} \right)  = \int_{\partial V_{\epsilon}} \log |z_1|^2 d^c f \wedge \tilde{\omega}^{n-1} $. Observe now that 
\begin{equation}
\tilde{\omega} |_{\partial V_{\epsilon}} = \omega |_{\partial V_{\epsilon}} + \lambda \ai \left(  \sum \epsilon^{2 \beta } (\partial_{\bar{j}} e^{\phi}) \ai e^{\ai \theta} d \theta \wedge d \bar{z}_j + c.c. + \epsilon^{2 \beta} \eta' |_{\partial V_{\epsilon}} \right) \label{estomtilde}
\end{equation}
where we wrote $z_1 = \epsilon e^{\ai \theta}$ on $\partial V_{\epsilon} = \{ |z_1| = \epsilon \}$. This means that
\begin{align}
&\left| \int_{\partial V_{\epsilon}} \log |z_1|^2 d^c f \wedge \tilde{\omega}^{n-1} \right| \label{probfone} \\
&\le  \cst . \log \epsilon \left| \int_{\partial V_{\epsilon}} (\epsilon^{2\beta} + \epsilon) d \theta \wedge dz_2 \wedge d \bar{z}_2 \wedge \dots dz_n \wedge d\bar{z}_n \right| \notag \\
&\le \cst . \epsilon^{2 \beta} \log \epsilon \to 0 \notag
\end{align}
as $\epsilon \to 0$, by noting that $dz_1 = \epsilon \ai e^{\ai \theta} d\theta$ on $\partial V_{\epsilon}$ and $f$ is smooth on $U$.

The second term of (\ref{bounv}) can be evaluated as
\begin{align}
\left| \int_{V_{\epsilon}}   \log |z_1|^2  dd^c f \wedge \tilde{\omega}^{n-1} \right| 
&\le \cst . \left| \int_{V_{\epsilon}}   \log r^2  \Delta_{\hat{\omega}} f   \hat{\omega}^{n} \right| \notag \\
&\le \cst . \left| \int_{V_{\epsilon}}   \log r^2  \hat{\omega}^{n} \right|  \label{probftwo} \\
&\le \cst . \left| \int_0^{\epsilon}  r^{2 \beta -1} \log r dr  \right| \to 0 \notag 
\end{align}
as $\epsilon \to 0$, by noting that $\Delta_{\hat{\omega}} f $ is bounded since $f$ is smooth on $U$ (cf.~Lemma \ref{lemginv}).

In order to evaluate the third term of (\ref{bounv}), we start by re-writing it as
\begin{align}
 &\int_{V_{\epsilon}}   \log |z_1|^2  d^c f \wedge d \tilde{\omega}^{n-1} \notag  \\
 &= - \lambda (n-1)  \int_{V_{\epsilon}}   \log |z_1|^2  d^c f \wedge  d\left(  \frac{\partial^2}{\partial z_1 \partial \bar{z}_1} (e^{\phi} |z_1|^{2 \beta}) \ai dz_1 \wedge d \bar{z}_1 \right) \wedge \tilde{\omega}^{n-2}  \label{probftthree}.
\end{align}
We have
\begin{align*}
&d \left( \frac{\partial^2}{\partial z_1 \partial \bar{z}_1} (e^{\phi} |z_1|^{2 \beta}) \ai dz_1 \wedge d \bar{z}_1 \right) \\
&= \sum_{j=2}^n \Big( \beta^2 (\partial_j e^{\phi})|z_1|^{2\beta -2} + \beta \partial_j ((\partial_1 \phi) e^{\phi})|z_1|^{2 \beta -2} z_1 \\
&\ \ + \beta \partial_j ((\partial_{\bar{1}} \phi) e^{\phi})|z_1|^{2 \beta -2} \bar{z}_1+ |z_1|^{2 \beta} \frac{\partial^3 e^{\phi}}{\partial z_1 \partial \bar{z}_1\partial z_j } \Big) \ai dz_1 \wedge d \bar{z}_1 \wedge dz_j + c.c.
\end{align*}
Since $\tilde{\omega}$ does not have any term proportionate to $dz_1$ or $d \bar{z}_1$ when wedged with $d \left( \frac{\partial^2}{\partial z_1 \partial \bar{z}_1} (e^{\phi} |z_1|^{2 \beta}) dz_1 \wedge d \bar{z}_1 \right)$, we have, from (\ref{omegatilde}),
\begin{align*}
&\left| d^c f \wedge  d\left(  \frac{\partial^2}{\partial z_1 \partial \bar{z}_1} (e^{\phi} |z_1|^{2 \beta}) dz_1 \wedge d \bar{z}_1 \right) \wedge \tilde{\omega}^{n-2} \right|_{\omega} \\
&\le \cst . \left| d^c f \wedge  d\left(  \frac{\partial^2}{\partial z_1 \partial \bar{z}_1} (e^{\phi} |z_1|^{2 \beta}) dz_1 \wedge d \bar{z}_1 \right) \wedge {\omega}^{n-2} \right|_{\omega} 
\end{align*}
and noting that $f$ is smooth on $U$, we have
\begin{equation}
\left| d^c f \wedge  d\left(  \frac{\partial^2}{\partial z_1 \partial \bar{z}_1} (e^{\phi} |z_1|^{2 \beta}) dz_1 \wedge d \bar{z}_1 \right) \wedge {\omega}^{n-2} \right|_{\omega} \le \cst . |z_1|^{2 \beta -2} . \label{probfthree}
\end{equation}
Thus
\begin{align}
 &\left| \int_{V_{\epsilon}}   \log |z_1|^2  d^c f \wedge d \tilde{\omega}^{n-1}  \right| \notag \\
 &= \left| \lambda (n-1)  \int_{V_{\epsilon}}   \log |z_1|^2  d^c f \wedge  d\left(  \frac{\partial^2}{\partial z_1 \partial \bar{z}_1} (e^{\phi} |z_1|^{2 \beta}) dz_1 \wedge d \bar{z}_1 \right) \wedge \tilde{\omega}^{n-2} \right| \notag \\
 &\le \cst . \left| \int_{V_{\epsilon}} r^{2 \beta -2} \log r \omega^n \right| \le \cst . \left| \int_0^{\epsilon} r^{2\beta -1} \log r dr \right| \to 0 \label{probffthree}
\end{align}
as $\epsilon \to 0$, finally establishing $ \int_{V_{\epsilon}}  d^c \log |z_1|^2 \wedge df \wedge \hat{\omega}^{n-1} \to 0$ as $\epsilon \to 0$.

Going back to (\ref{distpart}), we have thus shown $\lim_{\epsilon \to 0} \int_{V_{\epsilon}} \ai \ddbar \log |z_1|^2 \wedge f \hat{\omega}^{n-1} = \frac{1}{2} \lim_{\epsilon \to 0} \int_{V_{\epsilon}} d \left( d^c \log |z_1|^2 \wedge f \hat{\omega}^{n-1} \right) $, and hence are reduced to evaluating
\begin{align*}
\lim_{\epsilon \to 0} \int_{V_{\epsilon}} d \left( d^c \log |z_1|^2 \wedge f \hat{\omega}^{n-1} \right) &= \lim_{\epsilon \to 0} \int_{\partial V_{\epsilon}} d^c \log |z_1|^2 \wedge f \hat{\omega}^{n-1} \\
&= \lim_{\epsilon \to 0} \int_{\partial V_{\epsilon}} d^c \log |z_1|^2 \wedge f \tilde{\omega}^{n-1} .
\end{align*}
Recall that $d^c \log |z_1|^2 = 2 d \theta$ on $\{ |z_1| = \epsilon \}$, and also that $\lim_{\epsilon \to 0} \tilde{\omega} |_{\partial V_{\epsilon}} = \omega |_{\{z_1 = 0 \}}$, which follows from (\ref{omegatilde}). We thus have
\begin{align*}
&\lim_{\epsilon \to 0} \int_{\partial V_{\epsilon}} d^c \log |z_1|^2 \wedge f \tilde{\omega}^{n-1} \\
&=  \lim_{\epsilon \to 0} \int_{\partial V_{\epsilon}} 2 d\theta \wedge f \tilde{\omega}^{n-1} = \int_{0}^{2 \pi} 2 d \theta \int_{\{z_1 =0 \}} f \omega^{n-1} =4\pi \int_{\{ z_1=0 \}} f \omega^{n-1}.
\end{align*}
This means that
\begin{align*}
\lim_{\epsilon \to 0} \int_{V_{\epsilon}} \ai \ddbar \log |z_1|^2 \wedge f \hat{\omega}^{n-1} &= \frac{1}{2} \lim_{\epsilon \to 0} \int_{V_{\epsilon}} dd^c \log |z_1|^2 \wedge f \hat{\omega}^{n-1} \\
&= 2\pi \int_{\{ z_1 =0 \}} f \omega^{n-1}
\end{align*}
as claimed.

\end{proof}

\begin{lemma} \label{termtwomc}
For a momentum-constructed conically singular metric $\omega_{\varphi}$,
\begin{equation*}
 \int_{U}f   \ai \ddbar \log |z_1|^2 \wedge \omega_{\varphi}^{n-1} = 2  \pi   \int_{ \{ z_1 =0\} } f p^* \omega_M(b)^{n-1} ,
\end{equation*}
if $f$ is smooth and compactly supported in $U$.
\end{lemma}

\begin{proof}
The proof is essentially the same as the one for Lemma \ref{termtwo}. We note that we can proceed almost word by word, except for the places where we used the explicit description of $\hat{\omega}$ and $\tilde{\omega}$: the estimates (\ref{probfone}), (\ref{probftwo}), and in estimating (\ref{probftthree}). 

We certainly need to define a differential form, say $\tilde{\omega}_{\varphi}$, which replaces $\tilde{\omega}$ in the proof of Lemma \ref{termtwo}. We define it as $\tilde{\omega}_{\varphi} := \omega_{\varphi}  -  \frac{2 A_0^2 B_1^2 \beta^2 |z_1|^{4 \beta-2}  }{\varphi}  \ai d z_1 \wedge d \bar{z}_1$, by recalling the estimate (\ref{estdtau}).

Note again that this is not necessarily closed, and also that $\tilde{\omega}_{\varphi}$ does not even define a metric, since it is degenerate in the $dz_1 \wedge d \bar{z}_1$-component, whereas we certainly have $\tilde{\omega}_{\varphi} \le \cst . \omega_{\varphi}$. Observe that (\ref{estdtau}) and $\varphi = O(|z_1|^{2 \beta}) $ (as proved in Lemma \ref{lemestgmc}) imply that
\begin{align}
&\tilde{\omega}_{\varphi} |_{\partial V_{\epsilon}} \notag \\
&= \omega_{\varphi} |_{\partial V_{\epsilon}} + \left. \frac{1}{\varphi} \left( 2 \beta |z_1|^{4 \beta-2} \bar{z_1}  A_0 B_1 \sum_{i=2}^{n} \overline{B_{2,i}}  \ai d z_1 \wedge d \bar{z}_i  + c.c. + O(|z_1|^{4 \beta}) \right) \right|_{\partial V_{\epsilon}} \notag \\
&=  \omega_{\varphi} |_{\partial V_{\epsilon}} + O(\epsilon^{2 \beta}) , \label{esttilommc1}
\end{align}
which replaces (\ref{estomtilde}) in the proof of Lemma \ref{termtwo}. Note also that, by recalling (\ref{estdtau}),
\begin{align}
&\omega_{\varphi} |_{\partial V_{\epsilon}} = \left. \left( p^* \omega_M (\tau) + \frac{1}{\varphi} d \tau \wedge d^c \tau \right) \right|_{\partial V_{\epsilon}} \notag \\
&=p^* \omega_M (\tau) |_{\partial V_{\epsilon}} + \left. \frac{1}{\varphi} \left( - 2 \epsilon^{4 \beta} \beta   A_0 B_1 \sum_{i=2}^{n} \overline{B_{2,i}} d \theta \wedge d \bar{z}_i + c.c. + O(\epsilon^{4 \beta}) \right) \right|_{\partial V_{\epsilon}}  \label{esttilommc2}
\end{align}
where we wrote $z_1 = \epsilon e^{\ai \theta}$ on $\partial V_{\epsilon} = \{ |z_1| = \epsilon \}$ and used $d z_1 = \epsilon \ai e^{\ai \theta} d \theta$. Thus, recalling $\varphi = O(|z_1|^{2 \beta})$, $\omega_M (\tau) \le \cst . \omega_M$, and that $\omega_M$ depends only on $(z_2 , \dots , z_n)$, i.e.~the coordinates on the base $M$, we have the estimate
\begin{equation}
\omega_{\varphi} |_{\partial V_{\epsilon}} \le \cst . \left. \left( \sum_{i,j \neq 1} \ai dz_i \wedge d \bar{z}_j + \epsilon^{2 \beta} \sum_{j=2}^n \ai d \theta \wedge d z_j + c.c. \right) \right|_{\partial V_{\epsilon}} \label{estmcboun}
\end{equation}
from which it follows that
\begin{align}
&\left| \int_{\partial V_{\epsilon}} \log |z_1|^2 d^c f \wedge \omega_{\varphi}^{n-1} \right| \notag \\
&\le  \cst . \log \epsilon \left| \int_{\partial V_{\epsilon}} (\epsilon^{2\beta} + \epsilon) d \theta \wedge dz_2 \wedge d \bar{z}_2 \wedge \dots dz_n \wedge d\bar{z}_n \right|  \notag \\
&\le \cst . \epsilon^{2 \beta} \log \epsilon \to 0 \label{estbtmcttwo}
\end{align}
as $\epsilon \to 0$, for any smooth $f \in C^{\infty} (\mathbb{X} , \rl)$. This means that the estimate (\ref{probfone}) in the proof of Lemma \ref{termtwo} is still valid for momentum-constructed metrics $\omega_{\varphi}$.


Also, Lemma \ref{lemginvmc} and the estimate (\ref{volfmc}) (and also $\tilde{\omega}_{\varphi} \le \cst . \omega_{\varphi}$) means that the estimate in (\ref{probftwo}) in the proof of Lemma \ref{termtwo} is still valid for momentum-constructed metrics $\omega_{\varphi}$.

We are thus reduced to estimating (\ref{probftthree}), which is the third term of (\ref{bounv}) in the proof of Lemma \ref{termtwo}. We first note
\begin{align*}
&d \left( \frac{2 A_0^2 B_1^2 \beta^2 |z_1|^{4 \beta-2} }{\varphi}  \ai d z_1 \wedge d \bar{z}_1 \right) \\
&= \sum_{i=2}^n \frac{\partial}{\partial z_i} \left( \frac{2 A_0^2 B_1^2 \beta^2 |z_1|^{4 \beta-2} }{\varphi} \right) \ai dz_1 \wedge d \bar{z}_1 \wedge d z_i + c.c.
\end{align*}
Recalling the estimate (\ref{estmcboun}) and $\tilde{\omega}_{\varphi} \le \cst . \omega_{\varphi}$, we thus have, by using a smooth reference metric $\omega_0$ on $X$,
\begin{equation}
\left| d^c f \wedge d \left( \frac{2 A_0^2 B_1^2 \beta^2 |z_1|^{4 \beta-2} }{\varphi}  \ai d z_1 \wedge d \bar{z}_1 \right) \wedge \tilde{\omega}_{\varphi}^{n-2} \right|_{\omega_0}  \le \cst . |z_1|^{2 \beta -2}, \label{estfprimexc}
\end{equation}
where in the last estimate we used the fact that $f$ is smooth and that $\varphi$ is of order $O(|z_1|^{2 \beta})$ (cf.~Lemma \ref{lemestgmc}). This replaces (\ref{probfthree}) in the proof of Lemma \ref{termtwo}, and hence we see that the estimate (\ref{probffthree}) is still valid for the momentum-constructed metrics, establishing that the third term of (\ref{bounv}) in the proof of Lemma \ref{termtwo} goes to 0 as $\epsilon \to 0$. Since all the other arguments in the proof of Lemma \ref{termtwo} do not need the estimates that use the specific properties of $\hat{\omega}$, and hence applies word by word to the momentum-constructed case, we finally have
\begin{align*}
&\lim_{\epsilon \to 0} \int_{\partial V_{\epsilon}} d^c \log |z_1|^2 \wedge f \tilde{\omega}_{\varphi}^{n-1} =  \lim_{\epsilon \to 0} \int_{\partial V_{\epsilon}} 2 d\theta \wedge f \tilde{\omega}_{\varphi}^{n-1} \\
&= \int_{0}^{2 \pi} 2 d \theta \int_{\{z_1 =0 \}} f p^* \omega_M (b)^{n-1} =4\pi \int_{\{ z_1=0 \}} f p^* \omega_M (b)^{n-1},
\end{align*}
where we used $\tilde{\omega}_{\varphi}^{n-1} |_D = \omega^{n-1}_{\varphi} |_D = p^* \omega_M (b)^{n-1}$ by recalling (\ref{esttilommc1}), (\ref{esttilommc2}) and $D =\{ z_1=0 \} = \{ \tau =b \}$. We can thus conclude, as in Lemma \ref{termtwo}, that
\begin{align*}
 &\int_{U} \ai \ddbar \log |z_1|^2 \wedge f \omega_{\varphi}^{n-1} \\
 &= \lim_{\epsilon \to 0} \int_{V_{\epsilon}} \ai \ddbar \log |z_1|^2 \wedge f \omega_{\varphi}^{n-1}  = 2\pi \int_{\{ z_1 =0 \}} f p^* \omega_M(b)^{n-1},
\end{align*}
to get the claimed result.
\end{proof}


\subsection{Computation of Futaki invariant with respect to the conically singular metrics} \label{lfinvcwercsm}

\subsubsection{Statement of the result}

The aim of this section is to prove the following corollary of Theorems \ref{scaldist} and \ref{scaldistmc}, which computes the Futaki invariant of conically singular metrics on the whole manifold. This result will be used in \S \ref{sipftcsmotf} to prove Theorem \ref{mtinvfutthird}.

\begin{corollary} \label{intfutakisingth}
\
\begin{enumerate}
\item Suppose that $\Xi$ is a holomorphic vector field on $X$ which preserves $D$. Write $H$ for the holomorphy potential of $\Xi$ with respect to $\omega$, and $\hat{H} $ for the one with respect to a conically singular metric of elementary form $\hat{\omega}$ with $0 < \beta <1$. Then we have
\begin{align*}
\textup{Fut} (\Xi , \hat{\omega}) = &\int_{X \setminus D} \hat{H} (S (\hat{\omega}) - \underline{S}(\hat{\omega})) \frac{\hat{\omega}^n}{n!}  \notag \\
&+ 2 \pi (1 - \beta) \left( \int_D H \frac{\omega^{n-1}}{(n-1)!} -   \frac{\textup{Vol} (D , \omega)}{\textup{Vol} (X , \hat{\omega})}\int_X \hat{H} \frac{\hat{\omega}^n}{n!} \right), 
\end{align*}

where $\underline{S} (\hat{\omega})$ is the average of $S(\hat{\omega})$ over $X \setminus D$ and all the integrals are finite.
\item Writing $\Xi$ for the generator of the fibrewise $\cx^*$-action on $\mathbb{X} = \prj (\mathcal{F} \oplus \cx)$, and $\tau$ for the holomorphy potential with respect to a momentum-constructed conically singular metric $\omega_{\varphi}$ with $0 < \beta <1$, we have
\begin{align*}
\textup{Fut} (\Xi , \omega_{\varphi} ) = &\int_{\mathbb{X} \setminus D} \tau (S (\omega_{\varphi}) - \underline{S}(\omega_{\varphi})) \frac{\omega_{\varphi}^n}{n!} \\
&+ 2 \pi (1 - \beta) \left(    b \mathrm{Vol} (M , \omega_M (b))- \frac{\textup{Vol} (M ,  \omega_M (b))}{\textup{Vol} (\mathbb{X} , \omega_{\varphi})}\int_{\mathbb{X}} \tau \frac{\omega_{\varphi}^n}{n!} \right) ,
\end{align*}
where $D$ is the $\infty$-section defined by $\tau =b$, and $\omega_M (b)$ is as defined in (\ref{defofomegamtau}); see \S \ref{bgmc}. All the integrals in the above are finite.
\end{enumerate}
\end{corollary}

\subsubsection{Proof of the first item of Corollary \ref{intfutakisingth}} \label{futinvcmpcsmtf}
We first consider the conically singular metric of elementary form $\hat{\omega} = \omega + \lambda \ai \ddbar |s|^{2 \beta}_h$. Suppose now that $\Xi$ is a holomorphic vector field with the holomorphy potential $H \in C^{\infty} (X ,\cx)$, with respect to $\omega$, so that $\iota(\Xi) \omega = - \bar{\partial} H$. The holomorphy potential of $\Xi$ with respect to $\hat{\omega}$ is given by $H - \lambda \ai \Xi(|s|_h^{2 \beta})$, since, writing $\Xi = \sum_{i=1}^n v^i \frac{\partial}{\partial z_i}$ with $\bar{\partial} v^i=0$ in terms of local holomorphic coordinates $(z_1 , \dots , z_n)$, we have (cf.~\cite[Lemma 4.10]{sze})
\begin{equation}
 \iota \left( v^i \frac{\partial}{\partial z_i} \right) \ai \ddbar |s|_h^{2 \beta} = \ai v^i \frac{\partial^2 |s|_h^{2 \beta}}{\partial z^i \partial \bar{z}^j} d \bar{z}^j = \bar{\partial} \left( \ai  v^i \frac{\partial |s|_h^{2 \beta}}{\partial z^i} \right) . \label{varhamhomega}
\end{equation}

Suppose we write $|s|_h^{2 \beta} = e^{\beta \phi} |z_1|^{2 \beta}$ in local coordinates on $U$, where $h = e^{\phi}$ for some function $\phi$ that is smooth on the closure of $U$. We now wish to evaluate $\Xi (e^{\beta \phi} |z_1|^{2 \beta})$. If we assume that $\Xi$ preserves the divisor $D = \{ z_1 =0 \} $, we need to have $\Xi |_D =  \sum_{i=2}^n v^i \frac{\partial}{\partial z_i}$, and so $v^1$ has to be a holomorphic function that vanishes on $\{ z_1 =0 \}$. This means that we can write $v^1 = z_1 v'$ for another holomorphic function $v'$. We thus see that $\Xi (e^{ \beta \phi} |z_1|^{2 \beta}) = \sum_{i=1}^n v^i \partial_i (e^{\beta \phi} |z_1|^{2 \beta})$ is of order $|z_1|^{2 \beta}$ near $D$. We thus obtain that, for a holomorphic vector field $\Xi$ preserving $D$, there exists a ($\cx$-valued) function $H'$ that is smooth on $X \setminus D$ and is of order $|z_1|^{2 \beta}$ near $D$ and satisfies
\begin{equation} \label{holpotcsmtf}
\iota(\Xi) \hat{\omega} = - \bar{\partial} (H + H') ,
\end{equation}
i.e. $\hat{H}: = H+H'$ is the holomorphy potential of $\Xi$ with respect to $\hat{\omega}$.

We wish to extend Theorem \ref{scaldist} to the case when $f$ is replaced by the holomorphy potential $\hat{H}$ of a holomorphic vector field $\Xi$ with respect to $\hat{\omega}$. This means that we need to extend Theorem \ref{scaldist} to functions $f'$ that are not necessarily smooth on the whole of $X$ but merely smooth on $X \setminus D$ and are asymptotically of order $O( |z_1|^{2 \beta})$ near $D$. Note that most of the proof carries over word by word when we replace $f$ by such $f'$, except for the place where we showed $ \lim_{\epsilon \to 0}  \int_{U \setminus U_{\epsilon}}  d^c \log |z_1|^2 \wedge df \wedge \hat{\omega}^{n-1} =0$ in the equation (\ref{distpart}) when we proved Lemma \ref{termtwo}. More specifically, the smoothness of $f$ was crucial in the estimates (\ref{probfone}), (\ref{probftwo}), and (\ref{probfthree}) but not anywhere else. Thus the Lemma \ref{termtwo} still applies to $f'$ if we can prove the estimates used in (\ref{probfone}), (\ref{probftwo}), and (\ref{probfthree}) for $f'$. Note that we may still assume that $f'$ is compactly supported on $U$, since this is the property coming from applying the partition of unity.

For (\ref{probfone}), note first that on $\partial V_{\epsilon}$, 
\begin{equation*}
| d^c f' |_{\omega} \le \cst . \left| \epsilon (\partial_1 f ) d \theta + \sum_{i=2}^n (\partial_i f') + c.c. \right|_{\omega} = O(\epsilon^{2\beta})
\end{equation*}
by noting that $dz_1 = \ai \epsilon e^{\ai \theta} d \theta$ on $\partial V_{\epsilon}$. Thus we have
\begin{align}
&\left| \int_{\partial V_{\epsilon}} \log |z_1|^2 d^c f' \wedge \tilde{\omega}^{n-1} \right| \notag \\
&\le  \cst . \epsilon^{2 \beta} \log \epsilon \left| \int_{\partial V_{\epsilon}} (\epsilon^{2\beta} + \epsilon) d \theta \wedge dz_2 \wedge d \bar{z}_2 \wedge \dots dz_n \wedge d\bar{z}_n \right| \notag \\
&\le \cst . \epsilon^{4 \beta} \log \epsilon \to 0 \label{estbtmcttwo1}
\end{align}
in place of (\ref{probfone}).

For (\ref{probftwo}), we need to estimate $\Delta_{\hat{\omega}} f'$, but we simply recall Lemma \ref{lemginv} and see that $\Delta_{\hat{\omega}} f'$ is bounded on the whole of $U$. Thus the estimate established in (\ref{probftwo})
\begin{equation} \label{probftwoprm}
\left| \int_{V_{\epsilon}}   \log |z_1|^2  dd^c f' \wedge \tilde{\omega}^{n-1} \right| \le \cst . \left| \int_{V_{\epsilon}}   \log r^2  \hat{\omega}^{n} \right|
\end{equation}
still holds for $f'$.

We are left to verify that the estimate (\ref{probfthree}) holds for $f'$. We remark that, in computing (\ref{probfthree}), we may replace $d^cf$ with $\ai \sum_{j=2}^n ( \partial_{\bar{j}} f d \bar{z}_j- \partial_j f dz_j )$, since any term proportionate to $dz_1$ or $d \bar{z}_1$ will vanish when wedged with $d \left( \frac{\partial^2}{\partial z_1 \partial \bar{z}_1} (e^{\phi} |z_1|^{2 \beta}) dz_1 
\wedge d \bar{z}_1 \right)$. Thus, since $\partial_{\bar{j}} f'$ and $\partial_j f' $ ($2 \le j \le n$) are of order $O(r^{2 \beta})$, we have
\begin{equation} \label{estfprimexc1}
\left| d^c f' \wedge  d\left(  \frac{\partial^2}{\partial z_1 \partial \bar{z}_1} (e^{\phi} |z_1|^{2 \beta}) dz_1 \wedge d \bar{z}_1 \right) \wedge {\omega}^{n-2} \right|_{\omega} \le \cst . |z_1|^{4 \beta -2}
\end{equation}
in place of (\ref{probfthree}), so that the conclusion (\ref{probffthree}) still holds.

Thus the proof of Lemma \ref{termtwo} carries over to $f'$. Noting that $f'$ vanishes on $D$, we have $ \int_{U}f'   \ai \ddbar \log |z_1|^2 \wedge \hat{\omega}^{n-1} = 0$. In particular, if $\Xi$ is a holomorphic vector field on $X$ that preserves $D$ whose holomorphy potential with respect to $\omega$ (resp.~$\hat{\omega}$) is $H$ (resp.~$\hat{H}:=H + H'$), we get
\begin{equation*}
\int_X \hat{H} \mathrm{Ric} (\hat{\omega}) \wedge \frac{\hat{\omega}^{n-1}}{(n-1)!} = \int_{X \setminus D} \hat{H} S (\hat{\omega}) \frac{\hat{\omega}^n}{n!} + 2  \pi (1 - \beta) \int_D H \frac{\omega^{n-1}}{(n-1)!}.
\end{equation*}
Combined with Remark \ref{averscaldxd}, we thus get the first item of Corollary \ref{intfutakisingth}.

\subsubsection{Proof of the second item of Corollary \ref{intfutakisingth}} \label{mccsmesohpndfut}

We now consider the momentum-constructed conically singular metrics $\omega_{\varphi}$ and the generator $\Xi$ of the fibrewise $\cx^*$-action that has $\tau$ as its holomorphy potential (see the argument at the beginning of \S \ref{logfutmc}). Recalling that $\tau - b$ is of order $O(|z_1|^{2 \beta})$, as we proved in Lemma \ref{lemestgmc}, we are thus reduced to establishing the analogue for $\omega_{\varphi}$ of the statement that we proved in \S \ref{futinvcmpcsmtf} for the conically singular metric of elementary form $\hat{\omega}$. In fact, the proof carries over word by word, where we only have to replace $\tilde{\omega}$ by $\tilde{\omega}_{\varphi}$ (cf.~the proof of Lemma \ref{termtwomc}); (\ref{estbtmcttwo}) is replaced by the analogue of (\ref{estbtmcttwo1}), $\Delta_{\omega_{\varphi}} f'$ is bounded by Lemma \ref{lemginvmc} to establish the analogue of (\ref{probftwoprm}), and (\ref{estfprimexc}) can be established by observing that we can replace $d^c f'$ by $\ai \sum_{j=2}^n (\partial_{\bar{j}} f' d \bar{z}_{j} - \partial_j f' dz_j)$, as we did in (\ref{estfprimexc1}).

Thus, arguing exactly as in \S \ref{futinvcmpcsmtf}, we get the second item of Corollary \ref{intfutakisingth}.

\section{Some invariance properties for the log Futaki invariant} \label{sipftcsmotf}

\subsection{Invariance of volume and the average of holomorphy potential for conically singular metrics of elementary form} \label{invvhamcstf}

We first specialise to the conically singular metric of elementary form $\hat{\omega}$. Momentum-constructed conically singular metrics will be discussed in \S \ref{iotlficmccsmss}.

We recall that the volume $\text{Vol} (X ,\hat{\omega})$ or the average of the integral $\int_X \hat{H} \frac{\hat{\omega}^n}{n!} $ is not necessarily a invariant of the \kah class, unlike in the smooth case. This is because, as we mentioned in  \S \ref{intscmsfi}, the singularities of $\hat{\omega}$ mean that we have to work on the noncompact manifold $X \setminus D$, on which we cannot naively use the integration by parts. The aim of this section is to find some conditions under which the boundary integrals vanish, as in the smooth case. We first prove the following lemma.  

\begin{lemma} \label{volinvcstf}
The volume $\textup{Vol}(X, \hat{\omega})$ of $X$ measured by a conically singular metric with cone angle $ 2 \pi \beta$ of elementary form $\hat{\omega} = \omega + \lambda \ai \ddbar |s|_h^{2 \beta}$ with $\omega \in c_1 (L)$ is equal to the cohomological $\int_X c_1 (L)^n /n!$ if $\beta > 0$.
\end{lemma}

\begin{proof}
Consider a path of metrics $\{ \hat{\omega}_t := \omega + t \ai \ddbar |s|^{2 \beta}_h \}$ defined for $0 \le t \le \lambda$ for sufficiently small $\lambda >0$, and write $\hat{g}_t$ for the metric corresponding to $\hat{\omega}_t$, with $g := \hat{g}_0$. Then we have $\left. \frac{d}{dt} \right|_{t = T} \hat{\omega}^n_t =  n \ai \ddbar |s|^{2 \beta}_h \wedge \hat{\omega}_T^{n-1} =\Delta_{T} |s|^{2 \beta}_h \hat{\omega}_T^n  $, where $\Delta_T$ is the (negative $\bar{\partial}$) Laplacian with respect to $\hat{\omega}_T$. If we show that $\left. \frac{d}{dt} \right|_{t = T} \int_X  \hat{\omega}^n_t= \left. \frac{d}{dt} \right|_{t = T} \int_{X \setminus D}  \hat{\omega}^n_t = \int_{X \setminus D} \left. \frac{d}{dt} \right|_{t = T}  \hat{\omega}^n_t=0$ for any $0 \le T \le \lambda \ll 1$ (where we used the Lebesgue convergence theorem in the second equality), then we will have proved $\text{Vol} (X , \hat{\omega}_T) = \text{Vol} (X , \omega) = \int_X c_1(L)^n /n!$. We thus compute $\int_{X \setminus D} \Delta_T |s|^{2 \beta}_h \hat{\omega}_T^n $ for any $0 \le T \le \lambda$. We treat the case $T =0$ and $T \neq 0$ separately. Note that in both cases, we may reduce to a local computation on $U \subset X$ by applying the partition of unity as we did in the proof of Theorems \ref{scaldist} and \ref{scaldistmc}.

First assume $T=0$. We now choose local holomorphic coordinates $(z_1 , \dots , z_n)$ on $U$ so that $D = \{ z_1 = 0 \}$. Writing $z_1 = re^{\ai \theta}$, we define a local $C^{\infty}$-tubular neighbourhood $D_{\epsilon}$ around $D = \{ z_1 = 0 \}$ by $D_{\epsilon} := \{ x \in X \mid |s|_h (x) \le \epsilon \}$. 
Then we have
\begin{align*}
\int_{U \setminus D} \Delta_{\omega} |s|^{2 \beta}_h \omega^n &= \int_{U \setminus D_{\epsilon}} \Delta_{\omega} |s|^{2 \beta}_h \omega^n + \int_{D_{\epsilon} \setminus D} \Delta_{\omega} |s|^{2 \beta}_h \omega^n  \\
&= \int_{U \setminus D_{\epsilon}} \Delta_{\omega} |s|^{2 \beta}_h \omega^n + \int_{D_{\epsilon} \setminus D}  \sum_{i,j} g^{ i \bar{j}} \frac{\partial^2 }{\partial z_i \partial \bar{z}_j} |s|_h^{2 \beta} \omega^n.
\end{align*}
Writing $r = |z_1|$ and noting that $ |s|_h = fr$ for some locally defined smooth bounded function $f$, we can evaluate $\left| \sum_{i,j} g^{ i \bar{j}} \frac{\partial^2 }{\partial z_i \partial \bar{z}_j} |s|_h^{2 \beta} \right| \le \cst . (r^{2 \beta -2} + r^{2 \beta -1} + r^{2\beta}) $. Thus
\begin{equation*}
\left| \int_{D_{\epsilon} \setminus D} \sum_{i,j} g^{ i \bar{j}} \frac{\partial^2 }{\partial z_i \partial \bar{z}_j} |s|_h^{2 \beta} \omega^n \right| \le \cst . \int_0^{\epsilon} (r^{2 \beta -2} + r^{2 \beta -1} + r^{2\beta})rdr  \to 0
\end{equation*}
as $\epsilon \to 0$, if $\beta > 0$.

We thus have to show that $\int_{U \setminus D_{\epsilon}}$ goes to 0 as $\epsilon \to 0$. Note that this is reduced to the boundary integral on $\partial D_{\epsilon}$ by the Stokes theorem (by recalling that we have been assuming $|s|_h^{2 \beta}$ is compactly supported in $U$ as a consequence of applying the partition of unity) as $\int_{U \setminus D_{\epsilon}} \Delta_{\omega} |s|^{2 \beta}_h \omega^n  = \int_{\partial D_{\epsilon}} n \ai \bar{\partial} |s|^{2 \beta}_h  \wedge \omega^{n-1}$. 
Recalling $d \bar{z}_1 \mid_{|z_1| = r} = (- \ai \cos \theta - \sin \theta) r d \theta$, we may write 
\begin{equation*}
\left. \bar{\partial} |s|^{2 \beta}_h  \wedge \omega^{n-1} \right|_{\partial D_{\epsilon}} = \frac{\partial |s|^{2 \beta }_h}{ \partial \bar{z}_1} F \epsilon d \theta \wedge \ai d z_2 \wedge d\bar{z}_2 \wedge \dots \wedge \ai d z_n \wedge d \bar{z}_n
\end{equation*}
with some smooth function $F$, in the local coordinates $(z_1 , \dots , z_n)$. We thus have $\left|  \int_{\partial D_{\epsilon}} n \ai \bar{\partial} |s|^{2 \beta}_h  \wedge \omega^{n-1} \right| \le \cst . \epsilon^{2 \beta -1} \epsilon \to 0$ as $\epsilon \to 0$, if $\beta >0$.

When $T >0$, note that $\Delta_T |s|^{2 \beta}_h = O(1)$ by Lemma \ref{lemginv}. By Lemma \ref{lemvoltyp}, we have $\hat{\omega}_T^n = O(r^{2 \beta -1})$, which shows that $\left| \int_{D_{\epsilon} \setminus D}  \sum_{i,j} \hat{g}_{T}^{ i \bar{j}} \frac{\partial^2 }{\partial z_i \partial \bar{z}_j} |s|_h^{2 \beta} \hat{\omega}_T^n \right| \le \cst . \int_0^{\epsilon} r^{2 \beta -1} dr \to 0$ as $\epsilon \to 0$. We are thus reduced to showing that the boundary integral $\int_{U \setminus D_{\epsilon}} \Delta_{T} |s|^{2 \beta}_h \hat{\omega}_T^n = \int_{\partial D_{\epsilon}} n \ai \bar{\partial} |s|^{2 \beta}_h  \wedge \hat{\omega}_T^{n-1}$ 
goes to 0 as $\epsilon \to 0$. We first evaluate $\int_{\partial D_{\epsilon}} n \ai \frac{\partial |s|^{2 \beta}_h}{\partial \bar{z}_1} d \bar{z}_1 \wedge \hat{\omega}_T^{n-1}$. By noting $d z_1 \wedge d \bar{z}_1=0$ on $\partial D_{\epsilon}$, we observe that $d \bar{z}_1 \wedge \hat{\omega}_T^{n-1} |_{\partial D_{\epsilon}}= F \epsilon d \theta \wedge \ai d z_2 \wedge d \bar{z}_2 \wedge \cdots \wedge \ai d z_n \wedge d \bar{z}_n$ for some function $F$, bounded as $\epsilon \to 0$, on $\partial D_{\epsilon}$. Thus $\left| \int_{\partial D_{\epsilon}} n \ai \frac{\partial |s|^{2 \beta}_h}{\partial \bar{z}_1} d \bar{z}_1 \wedge \hat{\omega}_T^{n-1} \right| = O(\epsilon^{2 \beta -1} \epsilon) \to 0$ as $\epsilon \to 0$ if $\beta >0$.

Again by noting $d z_1 \wedge d \bar{z}_1=0$ on $\partial D_{\epsilon}$, we observe that $d \bar{z}_i \wedge \hat{\omega}_T^{n-1} |_{\partial D_{\epsilon}}= F \epsilon d \theta \wedge \ai d z_2 \wedge d \bar{z}_2 \wedge \cdots \wedge \ai d z_n \wedge d \bar{z}_n$ for some function $F = O(\epsilon^{2 \beta -1})$ on $\partial D_{\epsilon}$. Thus $\left| \int_{\partial D_{\epsilon}} n \ai \frac{\partial |s|^{2 \beta}_h}{\partial \bar{z}_i} d \bar{z}_i \wedge \hat{\omega}_T^{n-1} \right| = O(\epsilon^{2 \beta} \epsilon^{2 \beta -1} \epsilon) \to 0$ as $\epsilon \to 0$ if $\beta >0$.
\end{proof}


\begin{lemma} \label{holinvcstf}
The average of the holomorphy potential $\int_X \hat{H} \frac{\hat{\omega}^n}{n!} $ in terms of the conically singular metric with cone angle $ 2 \pi \beta$ of elementary form $\hat{\omega} = \omega + \lambda \ai \ddbar |s|_h^{2 \beta}$ with $\omega \in c_1 (L)$ is equal to the one $\int_X H \frac{\omega^n} { n!}$ measured in terms of the smooth \kah metric $ \omega$, if $\beta >0$.
\end{lemma}

In particular, it is equal to $b_0$ (in \S \ref{usualksdtystatotcon}) of the product test configuration for $(X,L)$ defined by the holomorphic vector field on $X$ generated by $H$ (cf.~\cite[\S 2]{dontoric}), if $\beta >0$.

\begin{proof}

Recall that the holomorphy potential varies as (cf.~(\ref{varhamhomega})) $\left. \frac{d}{dt} \right|_{t = T} \hat{H}_t = \hat{g}_T^{i \bar{j}} \left( \frac{\partial}{\partial \bar{z}_j} \hat{H}_T \right) \left( \frac{\partial}{\partial z_i} |s|^{2 \beta}_h \right)$. Thus, using the Lebesgue convergence theorem (as in the proof of Lemma \ref{volinvcstf}), we get
\begin{align*}
\left. \frac{d}{dt} \right|_{t = T} \int_{X } \hat{H}_t \hat{\omega}^n_t &= \int_{X \setminus D} \ai n \left(  \partial |s|^{2 \beta}_h \wedge  \bar{\partial} \hat{H}_T + \hat{H}_T  \ddbar |s|^{2 \beta}_h \right) \wedge \hat{\omega}_T^{n-1} \\
&=  - \ai n \int_{X \setminus D}  d \left( \hat{H}_T \partial |s|^{2 \beta}_h \right) \wedge \hat{\omega}_T^{n-1}.
\end{align*}

We proceed as we did above in proving Lemma \ref{volinvcstf}. When $T=0$ we evaluate
\begin{align*}
 &\int_{U \setminus D}  d \left( H \partial |s|^{2 \beta}_h \right) \wedge {\omega}^{n-1} \\
 &=  \lim_{\epsilon \to 0} \int_{U \setminus D_{\epsilon}} d \left( H \partial |s|^{2 \beta}_h \right) \wedge {\omega}^{n-1} +   \lim_{\epsilon \to 0} \int_{D_{\epsilon} \setminus D} d \left( H \partial |s|^{2 \beta}_h \right) \wedge {\omega}^{n-1}
\end{align*}
Noting that $H$ is a smooth function defined globally on the whole of $X$, we apply exactly the same argument that we used in proving Lemma \ref{volinvcstf} to see that both these terms go to 0 as $\epsilon \to 0$.

When $T >0$, we evaluate
\begin{align*}
 &\int_{U \setminus D} d \left( \hat{H}_T \partial |s|^{2 \beta}_h \right) \wedge \hat{\omega}_T^{n-1} \\
 &=  \int_{U \setminus D_{\epsilon}} d \left( \hat{H}_T \partial |s|^{2 \beta}_h \right) \wedge \hat{\omega}_T^{n-1} +  \int_{D_{\epsilon} \setminus D} d \left( \hat{H}_T \partial |s|^{2 \beta}_h \right) \wedge \hat{\omega}_T^{n-1} .
\end{align*}
Recalling that $|\hat{H}_T | < \cst . (1 + r^{2 \beta})$, we can apply exactly the same argument as we used in the proof of Lemma \ref{volinvcstf}. Hence we finally get $\left. \frac{d}{dt} \right|_{t = T} \int_{X \setminus D} \hat{H}_t \hat{\omega}^n_t =0$ for all $0 \le T \ll 1$ if $\beta > 0$.
\end{proof}

As a consequence of Corollary \ref{intfutakisingth} and Lemmas \ref{volinvcstf}, \ref{holinvcstf}, we have the following.

\begin{corollary} \label{corfutcsmtypdist}
If $0 < \beta  <1$, we have
\begin{align*}
\textup{Fut} (\Xi , \hat{\omega}) =& \int_X \hat{H} ( S(\hat{\omega}) - \bar{S} (\hat{\omega}))\frac{\hat{\omega}^n}{n!} \\
=& \int_{X \setminus D} \hat{H} (S (\hat{\omega}) - \underline{S}(\hat{\omega})) \frac{\hat{\omega}^n}{n!} \\
&+ 2 \pi (1 - \beta) \left( \int_D H \frac{\omega^{n-1}}{(n-1)!} -   \frac{\mathrm{Vol} (D , \omega)}{\mathrm{Vol} (X , {\omega})}\int_X {H} \frac{{\omega}^n}{n!} \right),
\end{align*}
where we note that the last two terms are invariant under changing the \kah metric $\omega \mapsto \omega + \ai \ddbar \phi$ by $\phi \in C^{\infty} (X, \rl)$ (cf.~Theorem \ref{dfdiff}). 
\end{corollary}

\begin{remark} \label{exttermlogfutsing}
Note that the ``distributional'' term
\begin{equation*}
2 \pi (1 - \beta) \left( \int_D H \frac{\omega^{n-1}}{(n-1)!} -   \frac{\text{Vol} (D , \omega)}{\text{Vol} (X , {\omega})}\int_X {H} \frac{{\omega}^n}{n!} \right)
\end{equation*}
in the above formula is precisely the term that appears in the definition of the log Futaki invariant (up to the factor of $2 \pi$). Note also that $\textup{Vol} (D , \hat{\omega}) = \int_X [D] \wedge \frac{\hat{\omega}^{n-1}}{(n-1)!} = \int_D \frac{\omega^{n-1}}{(n-1)!} = \mathrm{Vol} (D , \omega)$ and $\int_D \hat{H} \frac{\hat{\omega}^{n-1}}{(n-1)!} = \int_D H \frac{\omega^{n-1}}{(n-1)!}$ by Lemma \ref{termtwo} (and its extension given in \S \ref{futinvcmpcsmtf}), where $[D]$ is the current of integration over $D$. This means that, combined with Lemmas \ref{volinvcstf} and \ref{holinvcstf}, we get
\begin{equation*}
\int_D \hat{H} \frac{\hat{\omega}^{n-1}}{(n-1)!} -   \frac{\text{Vol} (D , \hat{\omega})}{\text{Vol} (X , \hat{\omega})}\int_X \hat{H} \frac{\hat{\omega}^n}{n!}  = \int_D H \frac{\omega^{n-1}}{(n-1)!} -   \frac{\text{Vol} (D , \omega)}{\text{Vol} (X , {\omega})}\int_X {H} \frac{{\omega}^n}{n!} .
\end{equation*}

Thus, if we compute the log Futaki invariant $\textup{Fut}_{D , \beta}$ in terms of the conically singular metrics of elementary form $\hat{\omega}$, we get $\textup{Fut}_{D, \beta} (\Xi , \hat{\omega}) = \frac{1}{2 \pi} \int_{X \setminus D} \hat{H} ( S(\hat{\omega}) - \underline{S} (\hat{\omega}))\frac{\hat{\omega}^n}{n!}$, which will certainly be 0 if $\hat{\omega}$ satisfies $S(\hat{\omega}) = \underline{S} (\hat{\omega})$ on $X \setminus D$, i.e. $\hat{\omega}$ is cscK as defined in Definition \ref{defofcsckncpt}.
\end{remark}

\subsection{Invariance of the Futaki invariant computed with respect to the conically singular metrics of elementary form} \label{invfutwrtcsmtf}

We first recall how we prove the invariance of the Futaki invariant in the smooth case, following the exposition given in \S 4.2 of Sz\'ekelyhidi's textbook \cite{sze}. Write $\omega$ for an arbitrarily chosen reference metric in $c_1 (L)$ and write $\omega_t := \omega + t \ai \ddbar \psi$ with some $\psi \in C^{\infty} (X , \rl)$. Defining $\textup{Fut}_t (\Xi) := \int_X H_t (S(\omega_t) - \bar{S} ) \frac{\omega^n_t}{n!}$, where $H_t$ is the holomorphy potential of $\Xi$ with respect to $\omega_t$, we need to show $\frac{d}{dt} |_{t=0} \textup{Fut}_t (\Xi) =0$. 

Arguing as in \cite[\S 4.2]{sze}, we get
\begin{align*}
\left. \frac{d}{dt} \right|_{t=0}  \textup{Fut}_t (\Xi) 
=& \int_X  \ai n \left(  (S(\omega) - \bar{S})  \partial \psi  \wedge \bar{\partial} H  \right.  \\
&\left. - H ( \overline{\mathfrak{D}_{\omega}^* \mathfrak{D}_{\omega} \psi} -  \partial \psi  \wedge \bar{\partial} S(\omega) )   + H  (S(\omega) - \bar{S})  \ddbar \psi   \right) \wedge \omega^{n-1}
\end{align*}
where $\mathfrak{D}_{\omega}^* \mathfrak{D}_{\omega}$ is a fourth order elliptic self-adjoint linear operator defined as
\begin{equation*}
\mathfrak{D}_{\omega}^* \mathfrak{D}_{\omega} \phi := \Delta_{{\omega}}^2 \phi + \nabla_j (\textup{Ric}({\omega})^{\bar{k} j} \partial_{\bar{k}} \phi ).
\end{equation*}
We now perform the following integration by parts
\begin{align*}
\int_X &(S(\omega) - \bar{S})  \partial \psi  \wedge \bar{\partial} H \wedge \omega^{n-1} \\
=& - \int_X d \left(  H (S(\omega) - \bar{S})  \partial \psi  \wedge \omega^{n-1} \right) + \int_X  H    \bar{\partial} S(\omega) \wedge \partial \psi \wedge \omega^{n-1}  \\
&- \int_X ( H (S(\omega) -\bar{S})  \ddbar \psi \wedge \omega^{n-1}  \\
=&\int_X  H    \bar{\partial} S(\omega) \wedge \partial \psi \wedge \omega^{n-1}  - \int_X ( H (S(\omega) - \bar{S})  \ddbar \psi \wedge \omega^{n-1}  
\end{align*}
by using Stokes' theorem. This means
\begin{equation*}
\left. \frac{d}{dt} \right|_{t=0} \textup{Fut}_t (\Xi)= - \int_X H \overline{\mathfrak{D}_{\omega}^* \mathfrak{D}_{\omega} \psi} \omega^n = - \int_X \psi \mathfrak{D}_{\omega}^* \mathfrak{D}_{\omega} H \omega^n =0
\end{equation*}
as required, again integrating by parts.

We now wish to perform the above calculations when the \kah metric $\hat{\omega}$ has cone singularities along $D$. An important point is that, since we are on the noncompact manifold $X \setminus D$, we have to evaluate the boundary integral when we apply Stokes' theorem, and that the remaining integrals may not be finite.

As we did in the proof of Lemma \ref{volinvcstf}, we apply the partition of unity and reduce to a local computation around an open set $U$ on which the integrand is compactly supported. Writing $\hat{H} = H + H'$ for the holomorphy potential of $\Xi$ with respect to $\hat{\omega}$, as we did in (\ref{holpotcsmtf}), we first evaluate
\begin{align*}
&\int_{U \setminus D} d \left(  \hat{H} (S(\hat{\omega}) - \underline{S}(\hat{\omega}))  \partial  \psi \wedge \hat{\omega}^{n-1} \right)  \\
&= \lim_{\epsilon \to 0} \int_{U \setminus D_{\epsilon}} d \left(  \hat{H} (S(\hat{\omega}) - \underline{S}(\hat{\omega}))  \partial  \psi \wedge \hat{\omega}^{n-1} \right) \\
&=\lim_{\epsilon \to 0} \int_{\partial D_{\epsilon}} \hat{H} (S(\hat{\omega}) - \underline{S}(\hat{\omega}))  \partial  \psi \wedge \hat{\omega}^{n-1} .
\end{align*}

Note $dz _1 \wedge d \bar{z}_1 =0$ on $\partial D_{\epsilon}$, which implies
\begin{align}
\partial  \psi \wedge \hat{\omega}^{n-1}  |_{\partial D_{\epsilon}}  =&  \frac{\partial \psi}{\partial z_1} F_1 \epsilon d \theta \wedge  \ai dz_2 \wedge d \bar{z}_2 \wedge \cdots \wedge \ai dz_n \wedge d\bar{z}_n \notag \\
&+ \sum_{i \neq 1} \frac{\partial \psi}{\partial z_i} F_i \epsilon d \theta \wedge F_1 \ai dz_2 \wedge d \bar{z}_2 \wedge \cdots \wedge \ai dz_n \wedge d\bar{z}_n \label{estdpsiohnm1eps}
\end{align}
where $F_1$ is bounded as $\epsilon \to 0$ and $F_i$ ($i \neq 1$) is at most of order $\epsilon^{2 \beta -1}$, we see that $\partial  \psi \wedge \hat{\omega}^{n-1}  |_{\partial D_{\epsilon}} = O(\epsilon) + O(\epsilon^{2 \beta})$. Recalling $\hat{H}   = O(1) + O(|z_1|^{2 \beta})$ and $S(\hat{\omega}) = O(1) + O(|z_1|^{ 2-4 \beta})$, we see that the integrand of the above is at most of order $O(\epsilon^{1+2 - 4\beta})$. Thus we need $\beta < 3/4$ for the boundary integral to be 0.


We now evaluate $\int_X  \hat{H}    \bar{\partial} S(\hat{\omega}) \wedge \partial \psi \wedge \hat{\omega}^{n-1} $. Writing
\begin{equation*}
\bar{\partial} S(\hat{\omega}) \wedge \partial \psi \wedge \hat{\omega}^{n-1} = \frac{\partial S (\hat{\omega})}{\partial \bar{z}_1} d \bar{z}_1 \wedge \partial \psi \wedge \hat{\omega}^{n-1} + \sum_{i \neq 1} \frac{\partial S (\hat{\omega})}{\partial \bar{z}_i} d \bar{z}_i \wedge \partial \psi \wedge \hat{\omega}^{n-1},
\end{equation*}
we see that the order of the first term is at most $O(|z_1|^{2 - 4 \beta -1} |z_1|^{2 \beta -1+1})= O(|z_1|^{1 - 2 \beta})$, and the second term is at most of order $O(|z_1|^{2 - 4 \beta} |z_1|^{2 \beta -1}) = O(|z_1|^{1 - 2 \beta})$, and hence we need $1 - 2 \beta > -1$, i.e. $\beta < 1$ for the integral to be finite, by recalling $\hat{H} = O(1) + O(|z_1|^{2 \beta})$. Since the second term $\int_X (\Delta_{\hat{\omega}} \psi ) \hat{H} (S(\hat{\omega}) - \underline{S}(\hat{\omega})) \hat{\omega}^n$ is manifestly finite (by Lemma \ref{lemginv} and Remark \ref{csscall1}), we can perform the integration by parts to have $\left. \frac{d}{dt} \right|_{t=0} \textup{Fut}_t (\Xi) = - \int_{X \setminus D} \hat{H} \overline{\mathfrak{D}_{\hat{\omega}}^* \mathfrak{D}_{\hat{\omega}} \psi} \hat{\omega}^n $ if $0 < \beta < 3/4$. It remains to prove $ \int_{X \setminus D} \hat{H} \overline{\mathfrak{D}_{\hat{\omega}}^* \mathfrak{D}_{\hat{\omega}} \psi} \hat{\omega}^n =  \int_{X \setminus D} \psi \mathfrak{D}_{\hat{\omega}}^* \mathfrak{D}_{\hat{\omega}} \hat{H} \hat{\omega}^n =0 $. Recalling $\overline{ \mathfrak{D}_{\hat{\omega}}^* \mathfrak{D}_{\hat{\omega}} \psi} = \Delta_{\hat{\omega}}^2 \psi + \hat{\nabla}_{\bar{j}} (\textup{Ric}(\hat{\omega})^{k \bar{j}} \partial_{k} \psi )$ (by noting $\bar{\psi} = \psi$ as $\psi$ is a real function), where $\hat{\nabla}$ is the covariant derivative on $TX$ defined by the Levi-Civita connection of $\hat{\omega}$, we first consider
\begin{align*}
&\int_{U \setminus D} \hat{H} \hat{\nabla}_{\bar{j}} (\textup{Ric}(\hat{\omega})^{k \bar{j} } \partial_{{k}} \psi ) \hat{\omega}^n \\
&=\int_{U \setminus D} \hat{H} (\hat{\nabla}_{\bar{j}} \textup{Ric}(\hat{\omega})^{k \bar{j} }) \partial_{{k}} \psi  \hat{\omega}^n + \int_{U \setminus D} \hat{H}  \textup{Ric}(\hat{\omega})^{k \bar{j} } \bar{\partial}_j \partial_{{k}} \psi  \hat{\omega}^n\\
&=\ai n \int_{U \setminus D} \hat{H} \partial \psi \wedge \bar{\partial} S(\hat{\omega}) \wedge \hat{\omega}^{n-1} + \int_{U \setminus D}  \hat{H} S(\hat{\omega}) \Delta_{\hat{\omega}} \psi \hat{\omega}^n \\
&\ \ \ \ \ \ - \ai n (n-1) \int_{U \setminus D} \hat{H}  \textup{Ric}(\hat{\omega}) \wedge \ddbar \psi \wedge \hat{\omega}^{n-2}
\end{align*}
where we used the Bianchi identity $\hat{\nabla}_{\bar{j}} \textup{Ric}(\hat{\omega})^{k \bar{j} } = \hat{g}^{k \bar{j}} \partial_{\bar{j}} S(\hat{\omega})$ and the identity in \cite[Lemma 4.7]{sze}. We perform the integration by parts for the second and the third term. We re-write the second term as

\begin{align*}
\int_{U \setminus D}  \hat{H} S(\hat{\omega}) \Delta_{\hat{\omega}} \psi \hat{\omega}^n 
= \ai n  &\left( -  \int_{U \setminus D} d( \hat{H} S(\hat{\omega}) {\partial} \psi \wedge \hat{\omega}^{n-1} ) - \int_{U \setminus D} d (S(\hat{\omega}) \bar{\partial} \hat{H}   \psi \wedge \hat{\omega}^{n-1} ) \right. \\
&\ \ \ \ +  \int_{U \setminus D} {\partial} S(\hat{\omega}) \wedge \bar{\partial} \hat{H}   \wedge \psi \hat{\omega}^{n-1} +  \int_{U \setminus D} S(\hat{\omega}) \ddbar \hat{H}   \wedge \psi \hat{\omega}^{n-1}  \\
&\ \ \ \ + \left. \int_{U \setminus D} \hat{H} \bar{\partial} S(\hat{\omega})  \wedge {\partial} \psi \wedge \hat{\omega}^{n-1}  \right)
\end{align*}
and the third term as
\begin{align*}
&\int_{U \setminus D} \hat{H}  \textup{Ric}(\hat{\omega}) \wedge \ddbar \psi \wedge \hat{\omega}^{n-2} \\
&=\int_{U \setminus D} d (\hat{H}  \textup{Ric}(\hat{\omega}) \wedge \bar{\partial} \psi \wedge \hat{\omega}^{n-2} ) + \int_{U \setminus D} d( \partial \hat{H} \wedge \textup{Ric}(\hat{\omega}) \wedge \psi  \hat{\omega}^{n-2}) \\
&\ \ \ \ \  -\int_{U \setminus D} \psi \bar{\partial} \partial \hat{H} \wedge \textup{Ric}(\hat{\omega})  \wedge \hat{\omega}^{n-2} .
\end{align*}

We thus have $\int_{U \setminus D} \hat{H} \hat{\nabla}_{\bar{j}} (\textup{Ric}(\hat{\omega})^{k \bar{j} } \partial_{{k}} \psi ) \hat{\omega}^n = \int_{U \setminus D} \psi \hat{\nabla}_{j} (\textup{Ric}(\hat{\omega})^{\bar{k} {j} } \partial_{\bar{k}} \hat{H} ) \hat{\omega}^n - \ai n (n-1) (B_1 + B_2) - \ai n ( B_3 + B_4)$, 
where the $B_i$'s stand for the boundary integrals 
\begin{align*}
B_1 &:= \lim_{\epsilon \to 0} \int_{\partial D_{\epsilon}} \hat{H}  \textup{Ric}(\hat{\omega}) \wedge \bar{\partial} \psi \wedge \hat{\omega}^{n-2}, \\
B_2 &:=  \lim_{\epsilon \to 0} \int_{\partial D_{\epsilon}}  \psi \partial \hat{H} \wedge \textup{Ric}(\hat{\omega}) \wedge  \hat{\omega}^{n-2}, \\
B_3 &:= \lim_{\epsilon \to 0} \int_{\partial D_{\epsilon}} \hat{H} S(\hat{\omega}) {\partial} \psi \wedge \hat{\omega}^{n-1}, \\
B_4 &:= \lim_{\epsilon \to 0} \int_{\partial D_{\epsilon}} \psi S(\hat{\omega}) \bar{\partial} \hat{H}   \wedge \hat{\omega}^{n-1},
\end{align*}
which we now evaluate.

We first evaluate $\int_{\partial D_{\epsilon}} \hat{H}  \textup{Ric}(\hat{\omega}) \wedge \bar{\partial} \psi \wedge \hat{\omega}^{n-2} $ in terms of $\epsilon$. Since $dz_1 \wedge d\bar{z}_1 =0$ on $\partial D_{\epsilon}$, we can see that this converges to 0 ($\epsilon \to 0$) as long as $0< \beta <1$, by recalling Lemma \ref{lemrictyf}. We thus get $B_1 =0$.


We then evaluate $ \int_{\partial D_{\epsilon}}  \psi \partial \hat{H} \wedge \textup{Ric}(\hat{\omega}) \wedge  \hat{\omega}^{n-2}$. We see that this converges to 0 ($\epsilon \to 0$) as long as $0< \beta <1$, exactly as we did before. We thus get $B_2 =0$.



Now we see that $\int_{\partial D_{\epsilon}} \hat{H} S(\hat{\omega}) {\partial} \psi \wedge \hat{\omega}^{n-1}$ is at most of order $\epsilon^{3 - 4 \beta}$, since $S(\hat{\omega})$ is at most of order $\epsilon^{2 - 4 \beta}$ and ${\partial} \psi \wedge \hat{\omega}^{n-1}$ is of order $O(\epsilon) + O(\epsilon^{2 \beta})$ (cf.~(\ref{estdpsiohnm1eps})), and hence converges to 0 (as $\epsilon \to 0$) if $\beta < 3/4$. Similarly, we can show that $\int_{\partial D_{\epsilon}} \psi S(\hat{\omega}) \bar{\partial} \hat{H}   \wedge \hat{\omega}^{n-1}$ converges to 0 if $\beta < 3/4$. Thus, we get $ B_3 = B_4 = 0$.

Note that $ \int_{U \setminus D} \psi \hat{\nabla}_k (  \textup{Ric}(\hat{\omega})^{\bar{j} k} (\partial_{\bar{j}} \hat{H}))  \hat{\omega}^n$ converges if $0 < \beta <1$, since Lemma \ref{lemrictyf}, combined with Lemma \ref{lemginv}, implies $\mathrm{Ric} (\hat{\omega})^{1 \bar{1}} = O(|z_1|^{2 - 2 \beta}) + O(|z_1|^{4 - 4 \beta})$, $\mathrm{Ric} (\hat{\omega})^{1 \bar{j}} = O(|z_1|) + O(|z_1|^{3 - 4 \beta}) + O(|z_1|^{2 - 2 \beta})$ ($j \neq 1$), and $\mathrm{Ric} (\hat{\omega})^{i \bar{j}} = O(1) + O(|z_1|^{2 \beta}) + O(|z_1|^{2 - 2 \beta})$ ($i,j \neq 1$). We thus see that we can perform the integration by parts in the above computation if we have $0 < \beta < 3/4$.


We are now left to prove $\int_{X \setminus D} \psi \Delta^2_{\hat{\omega}} \hat{H} \hat{\omega}^n = \int_{X \setminus D} \hat{H} \Delta^2_{\hat{\omega}} \psi \hat{\omega}^n$. We write
\begin{align*}
&\int_{X \setminus D} \hat{H} \Delta^2_{\hat{\omega}} \psi \hat{\omega}^n = \ai n \int_{X \setminus D} \hat{H} \ddbar (\Delta_{\hat{\omega}} \psi) \wedge \hat{\omega}^{n-1} \\
&= \ai n \int_{X \setminus D} d (\hat{H} \bar{\partial} (\Delta_{\hat{\omega}} \psi) \wedge \hat{\omega}^{n-1}) + \ai n \int_{X \setminus D} d (\partial \hat{H} \wedge (\Delta_{\hat{\omega}} \psi)  \hat{\omega}^{n-1} ) \\
&\ \ \ \ \ + \int_{X \setminus D} (\Delta_{\hat{\omega}} \hat{H}) (\Delta_{\hat{\omega}} \psi)  \hat{\omega}^{n-1} 
\end{align*}
and evaluate the boundary integrals $\lim_{\epsilon \to 0} \int_{\partial D_{\epsilon}} \hat{H} \bar{\partial} (\Delta_{\hat{\omega}} \psi) \wedge \hat{\omega}^{n-1}$ and $\lim_{\epsilon \to 0} \int_{\partial D_{\epsilon}} \partial \hat{H} \wedge (\Delta_{\hat{\omega}} \psi)  \hat{\omega}^{n-1}$ which, as before, can be shown to converge to zero as long as $\beta >0$.

We finally evaluate $\int_{U \setminus D} (\Delta_{\hat{\omega}} \hat{H}) (\Delta_{\hat{\omega}} \psi) \hat{\omega}^n$, where we recall from Lemma \ref{lemginv} that $\Delta_{\hat{\omega}}  \hat{H} = O(1) + O(|z_1|^{2 - 2 \beta})+ O(|z_1|^{2 \beta})$. Thus, computing as we did above, we see that this is finite.

Summarising the above argument, together with the results in \S \ref{invvhamcstf}, we have the following. Suppose that we compute the log Futaki invariant $\textup{Fut}_{D , \beta} (\Xi , \hat{\omega})$ defined as
\begin{equation*}
\frac{1}{2 \pi} \int_X \hat{H} (S(\hat{\omega}) - \bar{S} (\hat{\omega})) \frac{\hat{\omega}^n}{n!} -  (1- \beta ) \left( \int_D \hat{H} \frac{\hat{\omega}^{n-1}}{(n-1)!} - \frac{\textup{Vol} (D , \hat{\omega})}{\textup{Vol} (X ,\hat{\omega}) } \int_X \hat{H} \frac{\hat{\omega}^n}{n!} \right),
\end{equation*}
with respect to the conically singular metric of elementary form $\hat{\omega}$ for a holomorphic vector field $v$ that preserves the divisor $D$, with $\hat{H}$ as its holomorphy potential. As we mentioned in Remark \ref{exttermlogfutsing}, Lemmas \ref{volinvcstf}, \ref{holinvcstf}, Corollary \ref{corfutcsmtypdist}, combined with Lemma \ref{termtwo} (and its extension given in \S \ref{futinvcmpcsmtf}), show $\textup{Fut}_{D , \beta} (\Xi , \hat{\omega}) = \frac{1}{2 \pi} \int_{X \setminus D} \hat{H} (S(\hat{\omega}) - \underline{S} (\hat{\omega})) \frac{\hat{\omega}^n}{n!}$, and the calculations that we did above prove the first item of Theorem \ref{mtinvfutthird}.

\subsection{Invariance of the log Futaki invariant computed with respect to the momentum-constructed conically singular metrics} \label{iotlficmccsmss}
Now consider the case of momentum-constructed metrics on $\mathbb{X} := \prj(\mathcal{F} \oplus \cx)$ with the $\prj^1$-fibration structure $p : \prj(\mathcal{F} \oplus \cx) \to M$ over a \kah manifold $(M , \omega_M)$. In this section, we shall assume that the $\sigma$-constancy hypothesis (Definition \ref{defsigmaconst}) is satisfied for our data $\{ p: ( \mathcal{F}  , h_{\mathcal{F}})\to (M , \omega_M) , I  \}$. Let $D \subset \prj(\mathcal{F} \oplus \cx) =\mathbb{X}$ be the $\infty$-section, as before.

We first prove some lemmas that are well-known for smooth momentum-constructed metrics; the point is that they hold also for conically singular momentum-constructed metrics, since, as we shall see below, the proof applies word by word. We start with the following consequence of Lemma \ref{lemfbwv4pbps}.

\begin{lemma} \label{csqolemfbwv4pbps}
\emph{(\cite[Lemma 2.8]{hwang})}
Suppose that the $\sigma$-constancy hypothesis (Definition \ref{defsigmaconst}) is satisfied for our data. For any function $f (\tau)$ of $\tau$, we have
\begin{equation*}
\int_{\mathbb{X}} f(\tau) \frac{\omega_{\varphi}^n}{n!} = 2 \pi \mathrm{Vol}(M , \omega_M) \int_{-b}^b f(\tau) Q(\tau) d \tau ,
\end{equation*}
where $Q(\tau)$ is as defined in (\ref{defofq}). In particular, $\int_{\mathbb{X}} f(\tau) \frac{\omega_{\varphi}^n}{n!} $ does not depend on the choice of $\varphi$ or the boundary value $\varphi' (\pm b)$.
\end{lemma}

\begin{proof}
$\sigma$-constancy hypothesis implies that $Q(\tau)= \omega_M (\tau)^{n-1} / \omega_M^{n-1}$ is a function which depends only on $\tau$. We thus have
\begin{align*}
\int_{\mathbb{X}} f(\tau) \frac{\omega_{\varphi}^n}{n!} &=\int_{\mathbb{X}} \frac{\omega_M^{n-1}}{(n-1)!} \wedge \left(  \frac{ f(\tau) Q(\tau)}{\varphi} d \tau \wedge d^c \tau \right) \\
&=  2 \pi \mathrm{Vol}(M , \omega_M) \int_{-b}^b f(\tau) Q(\tau) d \tau ,
\end{align*}
by (\ref{omvpmcfbwvol}) in Lemma \ref{lemfbwv4pbps}.
\end{proof}

We summarise what we have obtained as follows.

\begin{lemma} \label{summcmcssmkcviavhi}
Suppose that the $\sigma$-constancy hypothesis is satisfied for our data. Let $\varphi : [-b,b] \to \rl_{\ge 0}$ be a real analytic momentum profile with $\varphi (\pm b)=0$ and $\varphi (-b) = 2$, $\varphi (-b) = - 2 \beta$, so that $\omega_{\varphi} = p^* \omega_M - \tau p^* \gamma + \frac{1}{\varphi} d \tau \wedge d^c \tau$ has cone singularities with cone angle $2 \pi \beta$ along the $\infty$-section. Let $\phi : [-b,b] \to \rl_{\ge 0}$ be another momentum profile with $\varphi (\pm b)=0$ and $\varphi (\pm b) = \mp 2$, so that $\omega_{\phi} = p^* \omega_M - \tau p^* \gamma + \frac{1}{\phi} d \tau \wedge d^c \tau$ is a smooth momentum-constructed metric. Then we have the following.

\begin{enumerate}
\item $\displaystyle [\omega_{\varphi}] = [\omega_{\phi}] $,
\item $\displaystyle \mathrm{Vol} (\mathbb{X} , \omega_{\varphi}) = 2 \pi \mathrm{Vol}(M , \omega_M) \int_{-b}^b Q(\tau) d \tau = \mathrm{Vol} (\mathbb{X} , \omega_{\phi})$,
\item $ \displaystyle \int_{\mathbb{X}} \tau \frac{\omega_{\varphi}^n}{n!} = 2 \pi \mathrm{Vol}(M , \omega_M) \int_{-b}^b \tau Q(\tau) d \tau =  \int_{\mathbb{X}} \tau \frac{\omega_{\phi}^n}{n!} $.
\end{enumerate}

\end{lemma}

\begin{proof}
The first item follows from Lemma \ref{lemfbwv4pbps}, and the second and the third from Lemma \ref{csqolemfbwv4pbps}.
\end{proof}

The second and the third item of the above lemma shows that the second ``distributional'' term in Corollary \ref{intfutakisingth} agrees with the ``correction'' term in the log Futaki invariant, as we saw in the case of conically singular metrics of elementary form (cf.~Corollary \ref{corfutcsmtypdist} and Remark \ref{exttermlogfutsing}). We thus get the following result.

\begin{corollary} \label{invlfimcmpbsfoxcc}
Suppose that the $\sigma$-constancy hypothesis is satisfied for our data $\{ p: (\mathcal{F} , h_{\mathcal{F}}) \to (M , \omega_M) , I\}$. Writing $\textup{Fut} (\Xi, {\omega}_{\varphi})$ for the Futaki invariant computed with respect to the momentum-constructed conically singular metric $\omega_{\varphi}$ with cone angle $2 \pi \beta$ and with real analytic momentum profile $\varphi$ and $0< \beta <1$, evaluated against the generator $\Xi$ of fibrewise $\cx^*$-action of $\mathbb{X} = \prj (\mathcal{F} \oplus \cx)$, we have
\begin{align*}
\textup{Fut} (\Xi , \omega_{\varphi} ) &= \int_{\mathbb{X} \setminus D} \tau (S (\omega_{\varphi}) - \underline{S}(\omega_{\varphi})) \frac{\omega_{\varphi}^n}{n!} \\
&\ \ \ \ \ + 2 \pi (1 - \beta) \left(    b \int_M  \frac{\omega_M(b)^{n-1}}{(n-1)!} - \frac{\textup{Vol} (M ,  \omega_M (b))}{\textup{Vol} (\mathbb{X} , \omega_{\phi})}\int_{\mathbb{X}} \tau \frac{\omega_{\phi}^n}{n!} \right) 
\end{align*}
where $\omega_{\phi}$ is a smooth momentum-constructed metric in the same \kah class as $ \omega_{\varphi} $. In particular,
\begin{equation*}
\textup{Fut}_{D , \beta} (\Xi , \omega_{\varphi} ) = \int_{\mathbb{X} \setminus D} \tau (S (\omega_{\varphi}) - \underline{S}(\omega_{\varphi})) \frac{\omega_{\varphi}^n}{n!} .
\end{equation*}

\end{corollary}

We now wish to establish the analogue of the first item of Theorem \ref{mtinvfutthird}. We first of all have to estimate the Ricci and scalar curvature of the metric $\omega_{\varphi} + \ai \ddbar \psi$ for $\psi \in C^{\infty} (\mathbb{X}, \rl)$. We show that this is exactly the same as the ones for the conically singular metrics of elementary form.

\begin{lemma}
$\mathrm{Ric} (\omega_{\varphi} + \ai \ddbar \psi)$ and $S (\omega_{\varphi} + \ai \ddbar \psi)$ satisfy the estimates as given in Lemma \ref{lemrictyf}.
\end{lemma}

\begin{proof}
Choose a local coordinate system $(z_1 , \dots , z_n)$ around a point in $X$ so that $D$ is locally given by $\{ z_1 =0  \}$. Lemma \ref{lemestgmc} and the estimate (\ref{estdtau}) imply that we have
\begin{align*} 
\omega_{\varphi} &+ \ai \ddbar \psi \\
=& p^*\omega_M - \tau p^* \gamma + \frac{1}{\varphi} d \tau \wedge d^c \tau + \ai \sum_{i,j=1}^n \frac{\partial^2 \psi}{\partial z_i \partial \bar{z}_j} \ai dz_i \wedge d \bar{z}_j \\
=& |z_1|^{2 \beta -2} \left( F_{11} \  + |z_1|^{2- 2 \beta} \frac{\partial^2 \psi}{\partial z_1 \partial \bar{z}_1}\right) \ai d z_1 \wedge d \bar{z}_1 \\
&+ \sum_{j=2}^n |z_1|^{2 \beta -2} \left( F_{1j}   \bar{z}_1 + |z_1|^{2- 2 \beta}\frac{\partial^2 \psi}{\partial z_1 \partial \bar{z}_j}\right) \ai d z_1 \wedge d \bar{z}_j + c.c.\\
&+ \sum_{i,j=2}^n \left( F_{ij}  |z_1|^{2 \beta}  + \frac{\partial^2 \psi + \psi_M}{\partial z_i \partial \bar{z}_j}\right) \ai d z_i \wedge d \bar{z}_j 
\end{align*}
where $F_{ij}$'s stand for locally uniformly convergent power series in $|z_1|^{2 \beta}$ with coefficients in smooth functions which depend only on the base coordinates $(z_2 , \dots , z_n)$. We also wrote $\psi_M$ for the local \kah potential for $p^* \omega_M$. When we Taylor expand $\psi$ and $\psi_M$, we thus get $(\omega_{\varphi} + \ai \ddbar \psi)^n = |z_1|^{2 \beta-2} [ O(1) + O(|z_1|^{2 \beta}) + O(|z_1|^{2 - 2 \beta}) ] \prod_{i=1}^n (\ai d z_i \wedge d \bar{z}_i) $. Writing $\omega_0 : = \prod_{i=1}^n (\ai d z_i \wedge d \bar{z}_i)$, we thus get
\begin{equation*} 
\log \frac{(\omega_{\varphi} + \ai \ddbar \psi)^n}{\omega_0^n} = (\beta -1) \log |z_1|^2 +  O(1) + O(|z_1|^{2 \beta}) + O(|z_1|^{2 - 2 \beta})  .
\end{equation*}
This is exactly the same as (\ref{tepoovpmcmf}), from which Lemma \ref{lemrictyf} follows (since $\ddbar \log |z_1|^2 =0$ on $\mathbb{X} \setminus D$).
\end{proof}

Since that the holomorphy potential for $\Xi$ with respect to $\omega_{\varphi} + \ai \ddbar \psi$ is given by $\tau - \ai \Xi (\psi) =  O(|z_1|^{2 \beta}) +O(1)$ (cf.~\cite[Lemma 4.10]{sze}), it is now straightforward to check that the calculations in \S \ref{invfutwrtcsmtf} apply word by word. We thus get the second item of Theorem \ref{mtinvfutthird}.


\section*{Acknowledgements}

The author thanks Michael Singer for the suggestion to look at the momentum-constructed metrics; indeed Theorem \ref{mainthmmc} was originally conjectured by him. The author also thanks him for the advice on the proof of Lemma \ref{lemestgmc}, and also on the materials presented in \S \ref{mccmcsad} and \S \ref{sipftcsmotf}. This work forms part of the author's PhD thesis submitted to the University College London, which he thanks for the financial support. The author thanks Eleonora Di Nezza and Yuji Odaka for helpful discussions, and Joel Fine, Jason Lotay, and Julius Ross for helpful comments that improved this paper. He is also grateful to the anonymous referees for helpful suggestions. The author is supported by the ``Investissements d'Avenir'' French Government programme, managed by the French National Research Agency (Agence Nationale de la Recherche), in the framework of the Labex Archim\`ede [ANR-11-LABX-0033] and the A*MIDEX [ANR-11-IDEX-0001-02].

\bibliographystyle{amsplain}
\bibliography{2015_22_Conic}

\begin{flushleft}
Aix Marseille Universit\'e, CNRS, Centrale Marseille, \\
Institut de Math\'ematiques de Marseille, UMR 7373, \\
13453 Marseille, France. \\
Email: \verb|yoshinori.hashimoto@univ-amu.fr|
\end{flushleft}

\end{document}